\documentclass[12pt]{amsart}

\usepackage{amsfonts, amstext, amsmath, amsthm, amscd, amssymb, enumitem, url}
\usepackage[table]{xcolor}
\usepackage{tikz-cd}
\usepackage{pdfpages}
\usepackage{pdflscape}

\usepackage{geometry}
\usepackage[parfill]{parskip}
\usepackage{microtype}

\usepackage{tabularx}

\newcolumntype{x}{>{\centering\arraybackslash}X}
\usepackage{makecell}
\usepackage{longtable}
\usepackage{hhline}
\usepackage{multirow}
\usepackage{subfig}

\usetikzlibrary{arrows.meta}

\usepackage[colorlinks,citecolor=cyan,linkcolor=magenta]{hyperref}
\usepackage{zref-clever} 
\zcsetup{cap = true, nameinlink = true}
\newcommand{\Cref}[1]{\zcref{#1}}

\newtheorem{thm}{Theorem}[section]
\zcRefTypeSetup{thm}{Name-sg = Theorem}

\usepackage{xparse} 

\NewCommandCopy{\newtheoremcopy}{\newtheorem}
\RenewDocumentCommand{\newtheorem}{m O{thm} m}{ 
\newtheoremcopy{#1}[#2]{#3}
\AddToHook{env/#1/begin}{\zcsetup{countertype={thm=#1}}}
\zcRefTypeSetup{#1}{Name-sg = #3}}

\newtheorem{lem}[thm]{Lemma}
\newtheorem{prop}[thm]{Proposition}
\newtheorem{cor}[thm]{Corollary}

\newtheorem{conj}[thm]{Conjecture}

\RenewDocumentCommand{\newtheorem}{m O{thm} m}{ 
\newtheoremcopy{#1}[#2]{#3}
\AddToHook{env/#1/begin}{\zcsetup{countertype={thm=#1}}}
\zcRefTypeSetup{#1}{Name-sg = #3}
\AtBeginEnvironment{#1}{\pushQED{\hfill $\lozenge$}}
\AtEndEnvironment{#1}{\popQED}}

\theoremstyle{definition}
\newtheorem{defn}[thm]{Definition}
\newtheorem{rmk}[thm]{Remark}
\newtheorem{constr}[thm]{Construction}
\newtheorem{eg}[thm]{Example}

\numberwithin{equation}{section}

\usepackage{todonotes}
\renewcommand{\epsilon}{\varepsilon}

\usepackage{stmaryrd}

\newcommand{\Solv}{\mathrm{Solv}}

\DeclareMathOperator{\intr}{\mathrm{int}}

\DeclareMathOperator{\sing}{\mathrm{sing}}
\DeclareMathOperator{\slope}{\mathrm{slope}}
\DeclareMathOperator{\hght}{\mathrm{height}}
\DeclareMathOperator{\cl}{\mathrm{cl}}

\allowdisplaybreaks

\begin{document}

\title[Legendrian position of veering triangulations]{Legendrian position of veering triangulations}

\author{Chi Cheuk Tsang}
\email{chicheuk@hotmail.com}

\begin{abstract}
We make a first step towards connecting the theory of veering triangulations and bicontact structures as tools for studying (pseudo-)Anosov flows:
We show that given a veering triangulation corresponding to an Anosov flow with orientable stable and unstable foliations, the edges of the triangulation can be realized as Legendrian arcs with respect to a strongly adapted bicontact structure that supports the Anosov flow.
Along the way, we show that every veering triangulation can be placed in `steady position', where each pair of edge projections that intersect in the orbit space only intersect once transversely. 
By a previous result of the author, this implies that horizontal surgery of veering triangulations correspond to horizontal Goodman surgery of pseudo-Anosov flows.
\end{abstract}

\maketitle


\section{Introduction} \label{sec:intro}

A flow on a 3-manifold is \textbf{Anosov} if it has a contracting and an expanding direction. A flow is \textbf{pseudo-Anosov} if it is Anosov except possibly having finitely many singular orbits where the contracting and expanding directions are pronged.
Aside from their interest as dynamical objects, work by many authors have shown a deep connection between these flows and many other topological and geometric structures on 3-manifolds, including taut foliations \cite{Mos96, Cal00, Fen02}, hyperbolic geometry \cite{Fen16, FL25}, and Floer homology \cite{AT25a, AT25b, Zun26}.

In recent years, a plethora of new tools for studying (pseudo-)Anosov flows have emerged. These include veering triangulations \cite{FSS19, Tsathesis}, orbit spaces \cite{BFraM25, BBonM24, BarMan25}, geometric types \cite{Iak22, Iak25}, and bicontact structures \cite{Mit95, CLMM22, Hoz24, Mas25, Sal25, AS26}. 
While there is some understanding of the relation between the first three of these \cite{LMT23, SS24, SS23, Hal25}, the literature on bicontact structures has been rather isolated from this cluster of ideas.
The purpose of this paper is to take a first step in making a connection, by showing an interaction between veering triangulations and bicontact structures.

Let us first give a brief review of these objects.
A \textbf{veering triangulation} is an ideal triangulation of an oriented 3-manifold along with combinatorial data, including coorientations on the faces and a coloring of the edges by red and blue, satisfying certain conditions.
Given a pseudo-Anosov flow $\phi$ on a closed oriented 3-manifold $M$ with no perfect fits relative to $\mathcal{C}$, there is an associated veering triangulation $\Delta$ of the drilled 3-manifold $M \backslash \mathcal{C}$.
It is known that $\Delta$ can be placed in \textbf{transverse position} with respect to $\phi$ in $M$ \cite{LMT23}. This means that the faces of $\Delta$ can be arranged to be positively transverse to $\phi$, and the edges of $\Delta$ can be arranged to be transverse to the stable and unstable foliations of $\phi$.
In fact, once placed in this position, $\Delta$ determines $\phi$ up to isotopic equivalence.

In more detail, the edges of $\Delta$ are determined by the \textbf{edge rectangles} in the orbit space of $\phi$. These are rectangles with two opposite corners on points of $\widetilde{\mathcal{C}}$.
Taking the convention that the stable foliation is vertical and the unstable foliation is horizontal in the orbit space, the rectangle corresponding to the top edge of each tetrahedron is taller and thinner than that of the bottom edge. 
That is, as one moves along the flow, the rectangles `rotate towards' the stable foliation.
The color of the edges determine which quadrant the rotation takes place in: the red edges are \textbf{positive}, i.e. lie in the first and third quadrants, while the blue edges are \textbf{negative}, i.e. lie in the second and fourth quadrants. See \Cref{fig:legpositionmotivation} left.

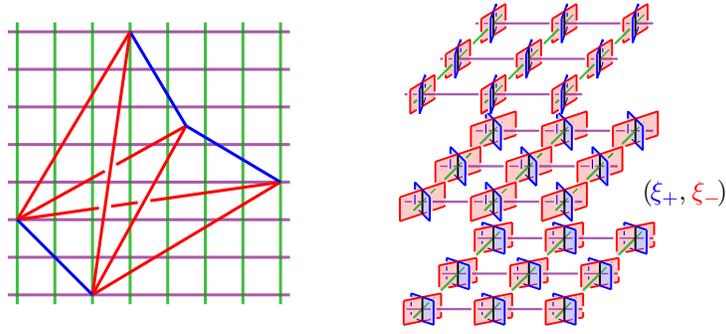
\begin{figure}
    \centering
    \fontsize{10pt}{10pt}
\begingroup%
  \makeatletter%
  \providecommand\color[2][]{%
    \errmessage{(Inkscape) Color is used for the text in Inkscape, but the package 'color.sty' is not loaded}%
    \renewcommand\color[2][]{}%
  }%
  \providecommand\transparent[1]{%
    \errmessage{(Inkscape) Transparency is used (non-zero) for the text in Inkscape, but the package 'transparent.sty' is not loaded}%
    \renewcommand\transparent[1]{}%
  }%
  \providecommand\rotatebox[2]{#2}%
  \newcommand*\fsize{\dimexpr\f@size pt\relax}%
  \newcommand*\lineheight[1]{\fontsize{\fsize}{#1\fsize}\selectfont}%
  \ifx\svgwidth\undefined%
    \setlength{\unitlength}{274.45941114bp}%
    \ifx\svgscale\undefined%
      \relax%
    \else%
      \setlength{\unitlength}{\unitlength * \real{\svgscale}}%
    \fi%
  \else%
    \setlength{\unitlength}{\svgwidth}%
  \fi%
  \global\let\svgwidth\undefined%
  \global\let\svgscale\undefined%
  \makeatother%
  \begin{picture}(1,0.4434198)%
    \lineheight{1}%
    \setlength\tabcolsep{0pt}%
    \put(0,0){\includegraphics[width=\unitlength,page=1]{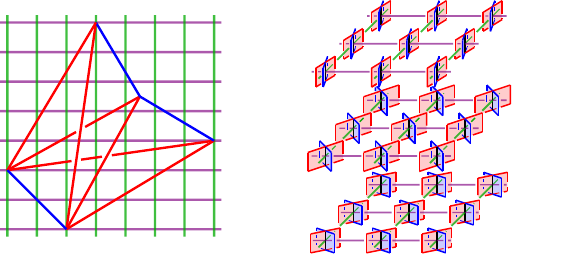}}%
    \put(0.88228195,0.17366437){\color[rgb]{0,0,1}\makebox(0,0)[lt]{\lineheight{1.25}\smash{\begin{tabular}[t]{l}$\xi_+$\end{tabular}}}}%
    \put(0.93930275,0.17366437){\color[rgb]{1,0,0}\makebox(0,0)[lt]{\lineheight{1.25}\smash{\begin{tabular}[t]{l}$\xi_-$\end{tabular}}}}%
    \put(0.87124183,0.17366437){\color[rgb]{0,0,0}\makebox(0,0)[lt]{\lineheight{1.25}\smash{\begin{tabular}[t]{l}$( \hspace{1em},\hspace{1em})$\end{tabular}}}}%
  \end{picture}%
\endgroup%

    \caption{Left: The edges of a veering triangulation rotate towards the stable foliation as one moves along the flow. Right: The contact planes of a supporting bicontact structure $(\xi_+,\xi_-)$ rotate towards the stable foliation as one moves along the flow.}
    \label{fig:legpositionmotivation}
\end{figure}

Meanwhile, for the purposes of this paper, a \textbf{bicontact structure} is a pair $(\xi_+,\xi_-)$ consisting of a positive contact structure and a negative contact structure that intersects transversely at every point.
We say that a bicontact structure \textbf{supports} a flow $\phi$ if the generating vector field $\dot{\phi}$ is contained in the intersection $\xi_+ \cap \xi_-$.
It is known that every Anosov flow with orientable stable and unstable foliations is supported by a bicontact structure.
Conversely, if a bicontact structure $(\xi_+,\xi_-)$ is \textbf{strongly adapted}, meaning $\xi_+$ admits a contact form whose Reeb flow $R_+$ is contained in $\xi_-$, then it supports an Anosov flow with orientable stable and unstable foliations.

When a bicontact structure $(\xi_+,\xi_-)$ supports an Anosov flow, the contact condition forces the contact planes $\xi_+$ and $\xi_-$ to rotate towards the stable foliation as one moves along the flow, with $\xi_+$ doing so within the second and fourth quadrants, and $\xi_-$ doing so within the first and third quadrants.
See \Cref{fig:legpositionmotivation} right.

Notice the similarity in behaviors of the blue/red edges of veering triangulations and the positive/negative contact planes of bicontact structures, respectively.
Our main theorem makes this observation precise:

\begin{thm} \label{thm:legendrianposition}
Let $\phi$ be an Anosov flow with orientable stable and unstable foliations on a closed oriented 3-manifold $M$ with no perfect fits relative to $\mathcal{C}$. 
Let $\Delta$ be the veering triangulation associated to $(\phi,\mathcal{C})$.
Then there exists a strongly adapted bicontact structure $(\xi_+,\xi_-)$ on $M$ supporting $\phi$, such that $\Delta$ can be put in transverse position with respect to $\phi$, with the blue edges of $\Delta$ being Legendrian with respect to $\xi_+$ and the red edges of $\Delta$ being Legendrian with respect to $\xi_-$.
\end{thm}

In the rest of this introduction, we explain an application to horizontal Goodman surgery, explain the strategy of proof to \Cref{thm:legendrianposition}, and finally present a speculative framework on the connection between veering triangulations and bicontact structures.

\subsection{Horizontal Goodman surgery}

In \cite{Tsa24a}, we defined a surgery operation on pseudo-Anosov flows which we called \textbf{horizontal Goodman surgery}.
To briefly summarize: A \textbf{positive/negative horizontal surgery curve} to a pseudo-Anosov flow $\phi$ is a curve $c$ that is positive/negative, respectively, and \textbf{steady}, in the sense that whenever there is \textbf{crossing} of $c$, i.e. an orbit segment from $x \in c$ to $y \in c$, then the tangent line $Tc|_y$ lies closer to the stable foliation than $Tc|_x$.
Cutting along an annulus $A$ containing $c$ and transverse to $\phi$, then regluing by a Dehn twist of positive/negative degree $n$, respectively, returns a surgered pseudo-Anosov flow $\phi_{\frac{1}{n}}(c)$ on the surgered manifold $M_{\frac{1}{n}}(c)$.

Furthermore, we conjectured that, in a precise way, horizontal Goodman surgery on pseudo-Anosov flows and \textbf{horizontal surgery} on veering triangulations correspond to each other.
The latter is a combinatorial operation defined in \cite{Tsa23} that takes in a veering triangulation $\Delta$, a \textbf{positive/negative horizontal surgery curve} $c$ satisfying some combinatorial conditions, and a positive/negative integer $n$, respectively, and returns a surgered veering triangulation $\Delta_{\frac{1}{n}}(c)$.

In \cite[Theorem 7.1]{Tsa24b}, we showed one direction of this conjecture for veering triangulations that can be placed in \textbf{steady position}, which means that the union of blue edges and the union of red edges are each steady, in the same sense as for horizontal surgery curves of pseudo-Anosov flows defined above.
In other words, we wish for the blue/red edges to rotate steadily towards the stable foliation as one moves along the flow.
Using the same ideas for proving \Cref{thm:legendrianposition}, we are now able to show that such a position can always be arranged.

\begin{thm} \label{thm:steadyposition}
Let $\phi$ be a pseudo-Anosov flow on a closed oriented 3-manifold $M$ with no perfect fits relative to $\mathcal{C}$. 
Then the veering triangulation $\Delta$ associated to $(\phi,\mathcal{C})$ can be placed in steady position with respect to $\phi$.
\end{thm}

Combining \Cref{thm:steadyposition} and \cite[Theorem 7.1]{Tsa24b}, we obtain the following corollary.

\begin{cor} \label{cor:horsurgerycorr}
Let $\phi$ be a pseudo-Anosov flow on a closed oriented 3-manifold $M$ with no perfect fits relative to $\mathcal{C}$. Let $\Delta$ be the veering triangulation associated to $(\phi,\mathcal{C})$. 
Then for every positive/negative horizontal surgery curve $c_\Delta$ of $\Delta$, there is an isotopic positive/negative horizontal surgery curve $c_\phi$ of the flow $\phi$, which is disjoint from $\mathcal{C}$.
Moreover, for every positive/negative integer $n$, the veering triangulation $\Delta_{\frac{1}{n}}(c_\Delta)$ is the veering triangulation associated to $(\phi_{\frac{1}{n}}(c_\Delta),\mathcal{C})$.
\end{cor}

\subsection{Strategy of proof}

Prior to this paper, \Cref{thm:steadyposition} has been known for layered veering triangulations.
We review the argument here, since our proof of \Cref{thm:legendrianposition} and \Cref{thm:steadyposition} will be built on these ideas.

A veering triangulation $\Delta$ is \textbf{layered} if it is the veering triangulation associated to some $(\phi,\mathcal{C})$ where the restriction $\phi^\circ$ of $\phi$ to the complement $M^\circ = M \backslash \mathcal{C}$ is the suspension of a pseudo-Anosov map $f$ on some surface $S^\circ$.
Here, recall that $f$ being a pseudo-Anosov map means that there exists \textbf{stable/unstable measure foliations} $(\ell^{s/u},\mu^{s/u})$ such that $f_*(\ell^s,\mu^s) = (\ell^s,\lambda^{-1} \mu^s)$ and $f_*(\ell^u,\mu^u) = (\ell^u,\lambda \mu^u)$ for some $\lambda > 1$.
In this case, the stable and unstable measured foliations induce a flat metric on $S^\circ$.
This in turn induces a flat metric on the orbit space $\widetilde{\mathcal{O}^\circ} = \widetilde{\mathcal{O} \backslash \widetilde{\mathcal{C}}} \cong \widetilde{S^\circ}$ of $\phi^\circ$, and a $\Solv$ metric on $M^\circ$.

We can now position edges of the veering triangulation by first pulling tight the diagonals of edges rectangles in $\widetilde{\mathcal{O}^\circ} \cong \widetilde{S^\circ}$, so that they become straight lines in the flat metric, and then taking the \textbf{canonical lifts} of these straight lines, i.e. lifting a straight line of slope $m$ to height $\frac{1}{2} \log_\lambda |m|$ in $\widetilde{M^\circ} \cong \mathbb{R} \times \widetilde{S^\circ}$.

There is a bicontact geometry perspective to this canonical lifting procedure:
The stable and unstable measured foliations on $S^\circ$ determine closed 1-forms $ds$ and $du$.
Using these, we can define the positive/negative contact forms
\begin{align} \label{eq:introsurfacebicontact}
\begin{split}
\alpha^S_+ = \lambda^t ds + \lambda^{-t} du \\
\alpha^S_- = \lambda^t ds - \lambda^{-t} du
\end{split}
\end{align}
on $\mathbb{R} \times S^\circ$.

A straightforward computation verifies that $(\ker \alpha^S_+, \ker \alpha^S_-)$ is a strongly adapted bicontact structure supporting the lift of $\phi^\circ$, which is generated by $\frac{\partial}{\partial t}$.
The canonical lift of a diagonal $d \subset \widetilde{\mathcal{O}^\circ}$ can now be defined as the unique arc $d^\wedge$ in $\widetilde{M^\circ}$ lying over $d$ that is Legendrian with respect to the lifted contact structure $\ker \widetilde{\alpha_\pm}$ on $\widetilde{M^\circ}$, where the sign on $\widetilde{\alpha_\pm}$ is negative/positive depending on whether $d$ has positive/negative slope, respectively.

It is clear that the union of positive/negative canonical lifts are steady.
The other important property is what we refer to as the \textbf{slope criterion}: Whenever $R_1$ and $R_2$ are edge rectangles where $R_2$ is taller and thinner than $R_1$, then at every intersection point of the diagonals $d_{R_1}$ and $d_{R_2}$, we have $|\slope(d_{R_2})| > |\slope(d_{R_1})|$.
This implies the following property which we refer to as the \textbf{crossing criterion}: Any crossings between the canonical lifts $d^\wedge_{R_1}$ and $d^\wedge_{R_2}$ are of $d^\wedge_{R_2}$ over $d^\wedge_{R_1}$.

The crossing criterion ensures that we can fill in the rest of the triangulation from the choice of the canonical lifts as the edges.
For instance, taking $R_2$ to be the top edge rectangle and $R_1$ to be the bottom edge rectangle of a tetrahedron rectangle, the crossing criterion ensure that the top edge $d^\wedge_{R_2}$ lies above the bottom edge $d^\wedge_{R_1}$, so that the tetrahedron is not `flipped inside out'.

Taking the quotient of the triangulation under $\pi_1 M^\circ$, we have realized the veering triangulation $\Delta$ in steady position.
Moreover, the forms $\alpha_+$ and $\alpha_-$ descend to contact forms on $M^\circ$, and the blue/red edges of $\Delta$ are Legendrian with respect to $\ker \alpha^S_{+/-}$, respectively.

Here, strictly speaking, the forms $ds, du$ on $S^\circ$ are only well-defined if $\ell^s$ and $\ell^u$ are orientable. Otherwise, they are locally well-defined up to a sign.
Similarly, $\alpha^S_+$ and $\alpha^S_-$ on $M^\circ$ are only well-defined if $\ell^s$ and $\ell^u$ are orientable, and $f$ preserves those orientations. Otherwise, they are locally well-defined up to a sign.

The main insight of this paper is that, with some work, the above argument can be carried over with the notion of a Birkhoff section.
Here, a \textbf{Birkhoff section} to a pseudo-Anosov flow $\phi$ is an immersed surface with boundary $S$ where 
\begin{itemize}
    \item the interior $\intr(S)$ is embedded and is transverse to the flow,
    \item the boundary $\partial S$ lies along closed orbits of $\phi$, and
    \item every orbit of $\phi$ meets $S$ in finite forward and backward time.
\end{itemize}

The first item implies that the stable/unstable foliations of $\phi$ induce singular 1-dimensional foliations $\ell^{s/u}$ on $\intr(S)$.
The third item implies that there is a first return map $f$ on $\intr(S)$, which must preserve the foliations $\ell^{s/u}$.
From this, one can show that $f$ is a pseudo-Anosov map with stable/unstable foliations $\ell^{s/u}$.

Then as above, we have a flat metric on $S^\circ = \intr(S) \backslash \sing(f)$, which induces a flat metric on the universal cover $\widetilde{\mathcal{O}^\circ} \cong \widetilde{S^\circ}$ of the orbit space $\mathcal{O}^\circ = \mathcal{O} \backslash (\widetilde{\partial S} \cup \widetilde{\sing(\phi)})$ of the restriction $\phi^\circ$ of $\phi$ to $M^\circ = M \backslash (\partial S \cup \sing(\phi))$, and a $\Solv$ metric on $M^\circ$.

The problem now is that $\partial S \cup \sing(\phi)$ may not be contained in $\mathcal{C}$, in which case the edge rectangles used to define $\Delta$ will not live in $\widetilde{\mathcal{O}^\circ}$.
Instead, they live in $\mathcal{O}$ and become punctured in $\mathcal{O}^\circ$.
This means that when we pull the diagonals tight, they will in general be caught on the punctures $\widetilde{\partial S}$, a priori in a possibly complicated way, so the slope criterion may not be satisfied.

What we will show is that if we first choose the diagonals to be the first-horizontal-then-vertical arcs from the corners to an appropriate choice of a center \textit{anchor}, then upon pulling it tight, as in \Cref{fig:intropldiag}, the resulting choice of piecewise linear diagonals will satisfy a proxy of the slope criterion.
Namely, diagonals of the same color may overlap along segments, but the segments immediately beyond the overlap satisfy the appropriate slope inequality. 
See \Cref{fig:introperturbdiag} left, and see \Cref{prop:pldiaggoals} for the detailed statement.
This uses nothing more than elementary Euclidean geometry, but does require a bit of casework.
It will also be convenient for us to add in some \textit{buoys} to the orbit space, so that the diagonal gets caught on the buoys instead of $\widetilde{\partial S}$.

\begin{figure}
    \centering
    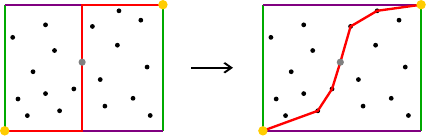
    \caption{Our first construction of the diagonals is to take the first-horizontal-then-vertical arcs from the corners to an appropriate choice of a center anchor, then pull it tight.}
    \label{fig:intropldiag}
\end{figure}

We will then show in \Cref{prop:smoothdiaggoals} that we can modify these piecewise linear diagonals to be smooth and satisfying the slope criterion on the nose.
This involves three substeps: We first peel apart the overlaps, then round the corners, and finally rotate near the intersection points of diagonals of opposite color to arrange for the slope criterion. See \Cref{fig:introperturbdiag}.
Among these, the first substep is the most technical, since we have to ensure that when we peel apart each individual pair of diagonals, we do not break the slope criterion for other pairs.

\begin{figure}
    \centering
    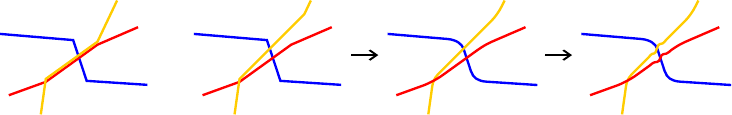
    \caption{Modifying the piecewise linear diagonals from \Cref{fig:intropldiag}. From left to right: Peeling apart the overlaps, rounding the corners, and rotating near the intersection points of diagonals of opposite color.}
    \label{fig:introperturbdiag}
\end{figure}

With the slope criterion being arranged, we can take the canonical lifts of the diagonals and fill in the triangulation as in the layered case. This proves \Cref{thm:steadyposition}.

The proof of \Cref{thm:legendrianposition} is more roundabout.
We first construct the veering triangulation $\Delta$ as in \Cref{thm:steadyposition}.
Since we are assuming that the flow $\phi$ is Anosov and has orientable stable and unstable foliations, the bicontact form $(\alpha^S_+,\alpha^S_-)$ from \Cref{eq:introsurfacebicontact} is well-defined.
However, it is only defined on $M^\circ = M \backslash \partial S$. 
We will show in \Cref{constr:fillbicontactformpos} and \Cref{constr:fillbicontactformneg} that we can extend the bicontact structure over each component of $\partial S$ while maintaining strong adaptedness.
This gives us a strongly adapted bicontact form $(\alpha_+,\alpha_-)$ on $M$. In particular, $(\alpha_+,\alpha_-)$ supports an Anosov flow $\psi$.

Meanwhile, since our diagonals were chosen to lie away from $\widetilde{\partial S}$, the canonical lifts, which are the edges of $\Delta$, lie away from $\partial S$. This means that $\Delta$ intersects $\partial S$ in the interior of its faces in meridional discs.
After our operation on the bicontact form, we can arrange for these meridional discs to be transverse to the new flow $\psi$.
This implies that $\Delta$ is in Legendrian position with respect to $\psi$. That is, we have proved \Cref{thm:legendrianposition} but for $\psi$.

The last step is to apply uniqueness of flows with respect to which $\Delta$ can be put in transverse position, in order to conclude that $\phi$ and $\psi$ are isotopically equivalent.
Transferring the triangulation and the bicontact structure through the isotopic equivalence, we deduce \Cref{thm:legendrianposition} for $\phi$.

\subsection{Hexality of veering triangulations}

We present a speculative picture of how the theories of veering triangulations and bicontact structures could intertwine.
To motivate things, let us consider the 3-torus $T^3 = \mathbb{R}^3/\mathbb{Z}^3$, and consider the three curves 
\begin{align*}
c_x &= \{(x,\frac{1}{2},0) \mid x \in \mathbb{R}/\mathbb{Z}\} \\ 
c_y &= \{(0,y,\frac{1}{2}) \mid y \in \mathbb{R}/\mathbb{Z}\} \\ 
c_z &= \{(\frac{1}{2},0,z) \mid z \in \mathbb{R}/\mathbb{Z}\} 
\end{align*}
and the three tori
\begin{align*}
T_{yz} &= \{(0,y,z) \mid y,z \in \mathbb{R}/\mathbb{Z}\} \\ 
T_{zx} &= \{(x,0,z) \mid x,z \in \mathbb{R}/\mathbb{Z}\} \\ 
T_{xy} &= \{(x,y,0) \mid x,y \in \mathbb{R}/\mathbb{Z}\}.
\end{align*}
For every triple of integers $n_x,n_y,n_z \in \mathbb{Z}$, we denote by $M_{n_x,n_y,n_z}$ the manifold $T^3_{\frac{1}{n_x},\frac{1}{n_y},\frac{1}{n_z}}(c_x,c_y,c_z)$ obtained by $\frac{1}{n_w}$ surgery on $c_w$, where we take the longitude on $\partial \nu(c_w)$ to be a coordinate curve, for $w=x,y,z$.

Suppose $n_x > 0, n_y < 0, n_z > 0$. 
Then $M_{n_x,n_y,0}$ is the mapping torus of the Anosov map $L^{n_x}R^{n_y} = \begin{bmatrix} 1 & 1 \\ 0 & 1 \end{bmatrix}^{n_x} \begin{bmatrix} 1 & 0 \\ 1 & 1 \end{bmatrix}^{n_y}$ on the torus $T_{xy}$, and, up to isotopy, $c_z$ is an orbit of the Anosov suspension flow. 
Since $M_{n_x,n_y,n_z}$ can be obtained from $M_{n_x,n_y,0}$ by performing $\frac{1}{n_z}$-surgery along the orbit $c_z$, it admits the Goodman-Fried surgered Anosov flow, which we denote by $\phi_{xy}$.

Interchanging the roles of $x$ and $z$, we can consider $M_{n_x,n_y,n_z}$ as a surgery of $M_{0,n_y,n_z}$ along the orbit $c_x$. As such, it admits another Anosov flow $\phi_{yz}$.

By construction, $\overline{T_{xy} \backslash c_z}$ is a Birkhoff section for $\phi_{xy}$, thus we can obtain a bicontact form $(\alpha_+,\alpha_-)$ supporting $\phi_{xy}$ by first defining the contact forms as in \Cref{eq:introsurfacebicontact}, then filling in along $c_z$ as in \Cref{constr:fillbicontactformpos}.
Then one can check that $\overline{T_{yz} \backslash c_x}$ is a Birkhoff section to the Reeb flow $R_+$ of $\alpha_+$.
In particular, the monodromy of $R_+$ on $T_{yz} \backslash c_x$ is in the same mapping class as that of $\phi_{yz}$.
We conjecture that we can arrange things so that $R_+ = \phi_{yz}$.

\begin{conj} \label{conj:3toruspulltight}
There exists a bicontact form $(\alpha_+,\alpha_-)$ supporting $\phi_{xy}$ where the Reeb flow $R_+$ to $\alpha_+$ is isotopically equivalent to $\phi_{yz}$.
\end{conj}

Symmetrically, we could make the same conjecture regarding Reeb flows of bicontact forms supporting $\phi_{yz}$.
The following conjecture morally combines these two versions of \Cref{conj:3toruspulltight}.

\begin{conj} \label{conj:3torusswap}
There exists two positive contact structures $\xi_x,\xi_z$ and a negative contact structure $\xi_y$ on $M_{n_x,n_y,n_z}$ such that
\begin{itemize}
    \item $\phi_{xy}$ is isotopically equivalent to a Reeb flow of $\xi_z$ and is supported by $(\xi_x,\xi_y)$.
    \item $\phi_{yz}$ is isotopically equivalent to a Reeb flow of $\xi_x$ and is supported by $(\xi_z,\xi_y)$.
\end{itemize}
\end{conj}

In other words, we predict that between $\phi_{xy}$ and $\phi_{yz}$, the role of the Anosov flow orbits and the positive Reeb flow orbits are interchanged.
In particular, (vertical) Goodman-Fried surgery of $\phi_{xy}$ along the orbit $c_z$ would be the same as horizontal Goodman surgery of $\phi_{yz}$ along the positive horizontal surgery curve $c_z$. Moreover, performing this surgery gets us from the setup on $M_{n_x,n_y,n_z}$ to that on $M_{n_x,n_y,n_z+1}$.

\begin{figure}
    \centering
    \fontsize{10pt}{10pt}
\begingroup%
  \makeatletter%
  \providecommand\color[2][]{%
    \errmessage{(Inkscape) Color is used for the text in Inkscape, but the package 'color.sty' is not loaded}%
    \renewcommand\color[2][]{}%
  }%
  \providecommand\transparent[1]{%
    \errmessage{(Inkscape) Transparency is used (non-zero) for the text in Inkscape, but the package 'transparent.sty' is not loaded}%
    \renewcommand\transparent[1]{}%
  }%
  \providecommand\rotatebox[2]{#2}%
  \newcommand*\fsize{\dimexpr\f@size pt\relax}%
  \newcommand*\lineheight[1]{\fontsize{\fsize}{#1\fsize}\selectfont}%
  \ifx\svgwidth\undefined%
    \setlength{\unitlength}{131.22142762bp}%
    \ifx\svgscale\undefined%
      \relax%
    \else%
      \setlength{\unitlength}{\unitlength * \real{\svgscale}}%
    \fi%
  \else%
    \setlength{\unitlength}{\svgwidth}%
  \fi%
  \global\let\svgwidth\undefined%
  \global\let\svgscale\undefined%
  \makeatother%
  \begin{picture}(1,0.59108982)%
    \lineheight{1}%
    \setlength\tabcolsep{0pt}%
    \put(0,0){\includegraphics[width=\unitlength,page=1]{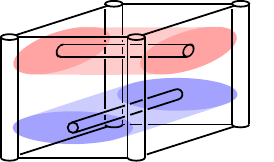}}%
    \put(0.95603956,0.35443007){\color[rgb]{0,0,0}\makebox(0,0)[lt]{\lineheight{1.25}\smash{\begin{tabular}[t]{l}$c_z$\end{tabular}}}}%
    \put(0.70566517,0.2250478){\color[rgb]{0,0,0}\makebox(0,0)[lt]{\lineheight{1.25}\smash{\begin{tabular}[t]{l}$c_x$\end{tabular}}}}%
    \put(0.74353756,0.40849194){\color[rgb]{0,0,0}\makebox(0,0)[lt]{\lineheight{1.25}\smash{\begin{tabular}[t]{l}$c_y$\end{tabular}}}}%
  \end{picture}%
\endgroup%

    \caption{The setup on the 3-torus.}
    \label{fig:3torus}
\end{figure}
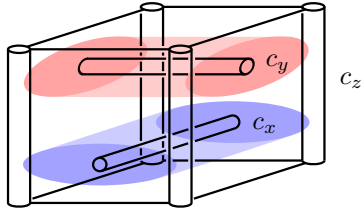

We now bring veering triangulations into the picture.
The Anosov flow $\phi_{xy}$ has no perfect fits relative to $c_z$, thus there is an associated veering triangulation $\Delta_{xy}$.
More specifically, $\Delta_{xy}$ can be built from a blue shearing region with core $c_x$ and a red shearing region with core $c_y$. See \Cref{fig:3torus}.
Here, a \textbf{blue/red shearing region} is a solid torus cellulated in a specific way. We refer to \cite{SS23} for details. 
For this discussion, it suffices to know that every veering triangulation has a canonical, combinatorially defined, decomposition into shearing regions. 

Similarly, there is a veering triangulation $\Delta_{yz}$ associated to $(\phi_{yz},c_x)$, whose blue shearing region has core $c_z$ and red shearing region has core $c_y$.
In other words, between $\Delta_{xy}$ and $\Delta_{yz}$, the roles of the drilled out orbit and the core of the blue shearing region are swapped.
We propose considering this as a veering triangulation analogue of \Cref{conj:3torusswap}.

Generalizing this example, we make the following ambitious `hexality' conjecture.

\begin{conj} \label{conj:hexality}
Let $\phi$ be an Anosov flow on $M$ with orientable stable and unstable foliations, with no perfect fits relative to $\mathcal{C}$.
Let $\Delta$ be the associated veering triangulation.
Let $\mathcal{A}$ be the collection of cores of the blue shearing regions of $\Delta$, and $\mathcal{B}$ be the collection of cores of the red shearing regions of $\Delta$.

\begin{enumerate}[label=(\Roman*)]

\item For every sufficiently positive $n_a$ and $n_c$, the surgered manifold $M_{\frac{1}{n_a},\frac{1}{n_c}}(\mathcal{A},\mathcal{C})$ admits the Anosov flow $\phi_{\frac{1}{n_a},\frac{1}{n_c}}(\mathcal{A},\mathcal{C})$ obtained by performing horizontal Goodman surgery along $\mathcal{A}$ and (vertical) Goodman-Fried surgery along $\mathcal{C}$. 

We conjecture that there exists two positive contact structures $\xi_a,\xi_c$, a negative contact structure $\xi_b$, and an Anosov flow $\phi^{\sigma_+}$ on $M_{\frac{1}{n_a},\frac{1}{n_c}}(\mathcal{A},\mathcal{C})$ such that
\begin{itemize}
    \item $\phi_{\frac{1}{n_a},\frac{1}{n_c}}(\mathcal{A},\mathcal{C})$ is isotopically equivalent to a Reeb flow of $\xi_c$ and is supported by $(\xi_a,\xi_b)$.
    \item $\phi^{\sigma_+}$ is isotopically equivalent to a Reeb flow of $\xi_a$ and is supported by $(\xi_c,\xi_b)$.
\end{itemize}

Here, the superscript $\sigma_+$ denotes that we have swapped the roles of the Anosov orbits and the \emph{positive} Reeb flow orbits.
While we have chosen not to reflect this in the notation, $\phi^{\sigma_+}$ should depend on $n_a$ and $n_c$.
Larger choices of $n_a$ and $n_c$ amount to performing (vertical) Goodman-Fried surgery along $\mathcal{A}$ and horizontal Goodman surgery along $\mathcal{C}$ on $\phi^{\sigma_+}$.

\item From \Cref{cor:horsurgerycorr}, we know that the surgered triangulation $\Delta_{\frac{1}{n_a}}(\mathcal{A})$ is the veering triangulation associated to $(\phi_{\frac{1}{n_a},\frac{1}{n_c}}(\mathcal{A},\mathcal{C}),\mathcal{C})$.

We conjecture that the Anosov flow $\phi^{\sigma_+}$ from (I) has no perfect fits relative to $\mathcal{A}$. 
This would allow us to talk about the veering triangulation $\Delta^{\sigma_+}$ associated to $(\phi^{\sigma_+},\mathcal{A})$.

While we have chosen not to reflect this in the notation, $\Delta^{\sigma_+}$ should depend on $n_c$.
A larger choice of $n_c$ amount to performing horizontal surgery along $\mathcal{C}$ on $\Delta^{\sigma_+}$.
On the other hand, $\Delta^{\sigma_+}$ should not depend on $n_a$.

\item Swapping the roles of $\mathcal{A}$ and $\mathcal{B}$ in (I) and (II) gives 
\begin{itemize}
    \item an Anosov flow $\phi^{\sigma_-}$, well-defined up to (vertical) Goodman-Fried surgery along $\mathcal{B}$ and horizontal Goodman surgery along $\mathcal{C}$, and
    \item a veering triangulation $\Delta^{\sigma_-}$, well-defined up to horizontal Goodman surgery along $\mathcal{C}$.
\end{itemize} 

We conjecture that we have the following isotopic equivalences:
\begin{itemize}
    \item $(\phi^{\sigma_+})^{\sigma_+} \cong \phi$ where both sides are taken up to (vertical) Goodman-Fried surgery along $\mathcal{C}$ and horizontal Goodman surgery along $\mathcal{A}$
    \item $(\phi^{\sigma_-})^{\sigma_-} \cong \phi$ where both sides are taken up to (vertical) Goodman-Fried surgery along $\mathcal{C}$ and horizontal Goodman surgery along $\mathcal{B}$
    \item $((\phi^{\sigma_+})^{\sigma_-})^{\sigma_+} \cong ((\phi^{\sigma_-})^{\sigma_+})^{\sigma_-}$ where both sides are taken up to (vertical) Goodman-Fried surgery along $\mathcal{C}$ and horizontal Goodman surgery along $\mathcal{A}$ and $\mathcal{B}$.
\end{itemize}
In other words, we have a $S_3$-torsor of Anosov flows $\{\phi^\rho\}_{\rho \in S_3}$. See \Cref{fig:hexality} left.

Similarly, we conjecture that we have the following combinatorial isomorphisms:
\begin{itemize}
    \item $(\Delta^{\sigma_+})^{\sigma_+} \cong \Delta$ where both sides are taken up to horizontal surgery along $\mathcal{A}$
    \item $(\Delta^{\sigma_-})^{\sigma_-} \cong \Delta$ where both sides are taken up to horizontal surgery along $\mathcal{B}$
    \item $((\Delta^{\sigma_+})^{\sigma_-})^{\sigma_+} \cong ((\Delta^{\sigma_-})^{\sigma_+})^{\sigma_-}$ where both sides are taken up to horizontal surgery along $\mathcal{A}$ and $\mathcal{B}$.
\end{itemize}
In other words, we have a $S_3$-torsor of veering triangulations $\{\Delta^\rho\}_{\rho \in S_3}$. See \Cref{fig:hexality} right.

\end{enumerate}

\end{conj}

\begin{figure}
    \centering
    \begin{tikzpicture}[font=\small, scale=1.2]
    \def\r{2}
    
    \node (A) at (90:\r)  {$\phi$};
    \node (B) at (30:\r)  {$\phi^{\sigma_-}$};
    \node (C) at (-30:\r) {$(\phi^{\sigma_-})^{\sigma_+}$};
    \node (D) at (-90:\r) {$((\phi^{\sigma_+})^{\sigma_-})^{\sigma_+} \cong ((\phi^{\sigma_-})^{\sigma_+})^{\sigma_-}$};
    \node (E) at (-150:\r) {$(\phi^{\sigma_+})^{\sigma_-}$};
    \node (F) at (150:\r) {$\phi^{\sigma_+}$};

    \draw[<->] (A) -- node[above right, font=\scriptsize] {$\sigma_-$} (B);
    \draw[<->] (B) -- node[right, font=\scriptsize] {$\sigma_+$} (C);
    \draw[<->] (C) -- node[below right, font=\scriptsize] {$\sigma_-$} (D);
    \draw[<->] (D) -- node[below left, font=\scriptsize] {$\sigma_+$} (E);
    \draw[<->] (E) -- node[left, font=\scriptsize] {$\sigma_-$} (F);
    \draw[<->] (F) -- node[above left, font=\scriptsize] {$\sigma_+$} (A);
    \end{tikzpicture}
    \hspace{0.5cm}
    \begin{tikzpicture}[font=\small, scale=1.2]
    \def\r{2}
    
    \node (A) at (90:\r)  {$\Delta$};
    \node (B) at (30:\r)  {$\Delta^{\sigma_-}$};
    \node (C) at (-30:\r) {$(\Delta^{\sigma_-})^{\sigma_+}$};
    \node (D) at (-90:\r) {$((\Delta^{\sigma_+})^{\sigma_-})^{\sigma_+} \cong ((\Delta^{\sigma_-})^{\sigma_+})^{\sigma_-}$};
    \node (E) at (-150:\r) {$(\Delta^{\sigma_+})^{\sigma_-}$};
    \node (F) at (150:\r) {$\Delta^{\sigma_+}$};

    \draw[<->] (A) -- node[above right, font=\scriptsize] {$\sigma_-$} (B);
    \draw[<->] (B) -- node[right, font=\scriptsize] {$\sigma_+$} (C);
    \draw[<->] (C) -- node[below right, font=\scriptsize] {$\sigma_-$} (D);
    \draw[<->] (D) -- node[below left, font=\scriptsize] {$\sigma_+$} (E);
    \draw[<->] (E) -- node[left, font=\scriptsize] {$\sigma_-$} (F);
    \draw[<->] (F) -- node[above left, font=\scriptsize] {$\sigma_+$} (A);
    \end{tikzpicture}
    \caption{An illustration of \Cref{conj:hexality}(III).}
    \label{fig:hexality}
\end{figure}
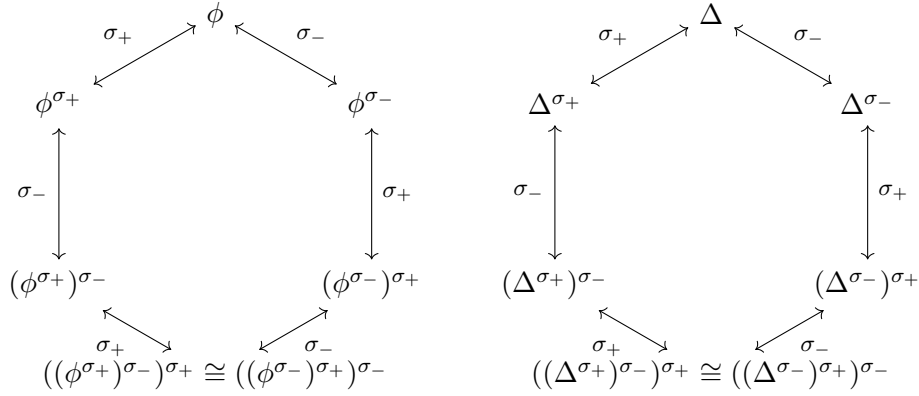

\subsection*{Outline of paper}

In \Cref{sec:paflow} and \Cref{sec:veertri}, we review some background on pseudo-Anosov flows and veering triangulations, respectively.

In \Cref{sec:edgestoveertri}, we explain how the crossing criterion on an edge candidate system allows one to fill in a veering triangulation from the edges (\Cref{prop:edgecandidatestoveertri}).
In \Cref{sec:diagstoveertri}, we explain how a Birkhoff section allows one to talk about canonical lifts, and how this allows one to translate the crossing criterion on an edge candidate system to the slope criterion on a diagonal system (\Cref{prop:diagtoveertri}).

The proof of the main theorems \Cref{thm:legendrianposition} and \Cref{thm:steadyposition} will span \Cref{sec:pldiag}, \Cref{sec:perturbdiag}, and \Cref{sec:fillinginbicontactform}.
In \Cref{sec:pldiag}, we construct a piecewise linear diagonal system that satisfies a proxy of the slope criterion (\Cref{prop:pldiaggoals}).
In \Cref{sec:perturbdiag}, we modify this piecewise linear diagonal system into one that satisfies the slope criterion on the nose (\Cref{prop:smoothdiaggoals}).
From there, \Cref{thm:steadyposition} will follow.

In \Cref{sec:fillinginbicontactform}, we show how one can fill in bicontact forms over the boundary orbits of a Birkhoff section.
From there, \Cref{thm:legendrianposition} will follow.

\subsection*{Acknowledgement}

We thank Michael Landry for discussions about approaches to \Cref{prop:strictanchorsystem}.
We thank Antonio Alfieri, Surena Hozoori, Federico Salmoiraghi, and Samuel Taylor for helpful conversations. 

This project was initiated while the author was supported by a CRM postdoctoral fellowship at Centre Interuniversitaire de Recherches en Géométrie et Topologie (CIRGET) at Montréal, Québec, during the Fall 2025 semester, and completed while the author was supported by the National Science Foundation under Grant No. DMS-2424139 at the Simons Laufer Mathematical Sciences Institute (SLMath) in Berkeley, California, during the Spring 2026 semester.

\section{Pseudo-Anosov flows} \label{sec:paflow}

\subsection{Basic definitions} \label{subsec:flowdefn}

A \textbf{flow} on a 3-manifold $M$ is a continuous map $\phi: \mathbb{R} \times M \to M$, which we write as $\phi^t(x) = \phi(t,x)$, satisfying $\phi^0(x) = x$ and $\phi^s(\phi^t(x)) = \phi^{s+t}(x)$ for every $s,t \in \mathbb{R}, x \in M$.

\begin{eg}[Suspension flow] \label{eg:susflow}
Let $f:S \to S$ be a homeomorphism of a surface. 
The \textbf{mapping torus} of $f$ is the 3-manifold $M_f = \mathbb{R} \times S/(s,x) \sim (s-1,f(x))$.
The \textbf{suspension flow} on $M_f$ is defined by $\phi^t_f(s,x) = (s+t,x)$.
\end{eg}

A flow $\phi$ is \textbf{smooth on an open set $U$} if the function $\phi: \mathbb{R} \times M \to M$ is smooth on $\mathbb{R} \times U$.
In this case, we define the \textbf{generating vector field} $\dot{\phi}$ by $\dot{\phi}|_x = \frac{d}{dt}\big|_{t=0} \phi^t(x)$ for every $x \in U$.
We also define the \textbf{tangent line field} $T\phi$ to be the line field spanned by $\dot{\phi}$.
A flow $\phi$ is \textbf{smooth} if it is smooth on the whole 3-manifold $M$.

Let $\phi_1$ and $\phi_2$ be flows defined on 3-manifolds $M_1$ and $M_2$ respectively. 
We say that $\phi_1$ and $\phi_2$ are \textbf{orbit equivalent} if there is a homeomorphism $h:M_1 \to M_2$ sending the oriented flow lines of $\phi_1$ to that of $\phi_2$, but not necessarily preserving their parametrizations.

\begin{defn}[Anosov flow] \label{defn:aflow}
An \textbf{Anosov flow} is a smooth flow $\phi$ on a closed oriented $3$-manifold $M$ for which there is a Riemannian metric $g$ and a splitting of the tangent bundle into three line bundles $TM=E^s \oplus T\phi \oplus E^u$ such that, for some $C, \lambda > 1$,
$$||d\phi^t(v)||_g < C \lambda^{-t} ||v||_g$$
for every $v \in E^s, t>0$, and 
$$||d\phi^t(v)||_g < C \lambda^t ||v||_g$$
for every $v \in E^u, t<0$.
\end{defn}

A pseudo-Anosov flow is essentially an Anosov flow with finitely many singular orbits. The singular orbits are modeled after the following construction.

\begin{constr}[Pseudo-hyperbolic orbit] \label{constr:phorbit}
Consider the map $\begin{pmatrix} \lambda & 0 \\ 0 & \lambda^{-1} \end{pmatrix}: \mathbb{R}^2 \to \mathbb{R}^2$. 
By first quotienting $\mathbb{R}^2$ by $(x,y) \sim (-x,-y)$, then taking the $n$-fold branched cover over the origin, we obtain a uniquely defined map $f_{n,0,\lambda}:\mathbb{R}^2 \to \mathbb{R}^2$ that preserves the lifts of the quadrants.
Let $f_{n,k, \lambda}: \mathbb{R}^2 \to \mathbb{R}^2$ be the composition of $f_{n,0,\lambda}$ and rotating by $\frac{2\pi k}{n}$ anticlockwise. 
Let $M_{n,k,\lambda}$ be the mapping torus of $f_{n,k,\lambda}$ and consider the suspension flow on $M_{n,k,\lambda}$. 
Let $\gamma_{n,k,\lambda}$ be the suspension of the origin. We refer to it as the \textbf{pseudo-hyperbolic orbit}.
\end{constr}

\begin{defn}[Pseudo-Anosov flow] \label{defn:paflow}
A \textbf{pseudo-Anosov flow} is a flow $\phi$ on a closed oriented $3$-manifold $M$ for which there is a path metric $d$ that is induced from a Riemannian metric $g$ away from a finite collection of closed orbits $\sing(\phi)$, which we call the \textbf{singular orbits}, such that:
\begin{itemize}
    \item Away from the singular orbits, 
    \begin{itemize}
        \item $\phi$ is smooth, and
        \item there is a splitting of the tangent bundle into three line bundles $TM=E^s \oplus T\phi \oplus E^u$, such that, for some $C, \lambda > 1$, 
        $$||d\phi^t(v)||_g < C \lambda^{-t} ||v||_g$$
        for every $v \in E^s, t>0$, and 
        $$||d\phi^t(v)||_g < C \lambda^t ||v||_g$$
        for every $v \in E^u, t<0$.
    \end{itemize}
    \item Around each singular orbit $\gamma$, there exists a neighborhood $N$ of $\gamma$ in $M$, a neighborhood $N_{n,k,\lambda}$ of the pseudo-hyperbolic orbit $\gamma_{n,k,\lambda}$ in $M_{n,k,\lambda}$, for some $n \geq 3$, $k \in \mathbb{Z}$, and $\lambda > 1$, and a homeomorphism $h:N_{n,k,\lambda} \to N$ such that
    \begin{itemize}
        \item $h$ is bi-Lipschitz on $N_{n,k,\lambda}$ and smooth away from $\gamma_{n,k,\lambda}$, and
        \item $h$ sends oriented flow lines to oriented flow lines, but not necessarily preserving their parametrization.
    \end{itemize}
\end{itemize}
\end{defn}

Given a pseudo-Anosov flow, one can show that the splitting $TM = E^s \oplus T\phi \oplus E^u$ in $M \backslash \sing(\phi)$ is automatically $\phi^t$-invariant and continuous, see \cite[Proposition 5.1.4]{FH19}. Furthermore, the plane field $E^s \oplus T\phi$ integrates uniquely to a 2-dimensional foliation $\Lambda^s$ on $M \backslash \sing(\phi)$, see \cite[Theorem 6.1.1]{FH19}, which can be extended into a singular 2-dimensional foliation on $M$. We call $\Lambda^s$ the \textbf{stable foliation}.
Similarly, we have the \textbf{unstable foliation} $\Lambda^u$ which is tangent to $T\phi \oplus E^u$ away from the singular orbits.
See \Cref{fig:paflow} for the local form of the stable and unstable foliations near a point on a nonsingular orbit (left) and near a point on a 3-pronged singular orbit (right).

\begin{figure}
    \centering
    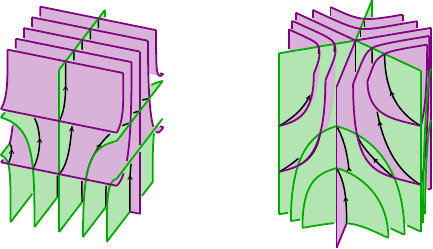
    \caption{The stable and unstable foliations near a point on a nonsingular orbit (left) and near a point on a 3-pronged singular orbit (right).}
    \label{fig:paflow}
\end{figure}

Let $\nu$ be a tubular neighborhood of a closed orbit $\gamma$. The local stable and unstable leaves of $\gamma$ cut $\nu$ into an even number of components. We refer to these components as the \textbf{quadrants} at $\gamma$.

If $\gamma$ is non-singular and has four quadrants, we say that it is \textbf{orientation-preserving}. Otherwise it has two quadrants and we say that it is \textbf{orientation-reversing}.

A flow is \textbf{transitive} if it has a dense orbit. 
For a pseudo-Anosov flow, this is equivalent to the set of closed orbits being dense, see \cite[Proposition 1.4.9]{BarMan25}.
All pseudo-Anosov flows in this paper will be transitive.

We remark that \Cref{defn:aflow} and \Cref{defn:paflow} are usually referred to as the definition of \emph{smooth} (pseudo-)Anosov flows. 
In comparison, one can find a definition of \emph{topological} (pseudo-)Anosov flows in, for example, \cite[Definition 1.1.10]{BarMan25}.
Every smooth pseudo-Anosov flow is a topological Anosov flow.
Conversely, it is known that every transitive topological pseudo-Anosov flow is orbit equivalent to a smooth pseudo-Anosov flow, see \cite[Theorem A]{Sha21} and \cite[Theorem 5.11]{AT24}.

\subsection{Orbit space} \label{subsec:orbitspace}

Let $\phi$ be a pseudo-Anosov flow on $M$. 
Let $\widetilde{\phi}$ be the lifted flow on the universal cover $\widetilde{M}$. 
It can be shown that the space of orbits $\mathcal{O}$ of $\widetilde{\phi}$, with the quotient topology, is homeomorphic to $\mathbb{R}^2$, see \cite[Theorem 1.3.14]{BarMan25}. 
The lifted stable/unstable foliations $\widetilde{\Lambda^{s/u}}$ induce singular 1-dimensional foliations $\mathcal{O}^{s/u}$ on $\mathcal{O}$.
The deck transformations $\pi_1 M$ act on $\mathcal{O}$ preserving the foliations $\mathcal{O}^{s/u}$.
We refer to the space $\mathcal{O}$ with the foliations $\mathcal{O}^{s/u}$ as the \textbf{orbit space} of $\phi$.

Throughout this paper, we will often use the same notation for an orbit of $\widetilde{\phi}$ and the point of $\mathcal{O}$ that it projects down to.
Also, we will adopt the convention of illustrating leaves of the stable foliation as green vertical lines and leaves of the unstable foliation as purple horizontal lines.
Finally, we adopt the convention of orienting the orbit space $\mathcal{O}$ so that orbits come out of the page.

The stable and unstable leaves of each point $\widetilde{\gamma} \in \mathcal{O}$ cut $\mathcal{O}$ into an even number of components. We refer to these components as the \textbf{quadrants} at $\widetilde{\gamma}$.

The stabilizer of each point $\widetilde{\gamma} \in \mathcal{O}$ is either trivial or cyclic.
The latter case is true if and only if $\widetilde{\gamma}$ projects to a closed orbit $\gamma$.
In this case the stabilizer is $\langle [\widetilde{\gamma}] \rangle \cong \mathbb{Z}$, where $[\widetilde{\gamma}]$ is a suitable conjugate of $[\gamma] \in \pi_1 M$, and we say that the point $\widetilde{\gamma}$ is \textbf{periodic}.
Furthermore, the closed orbit $\gamma$ is orientation-preserving if and only if $\widetilde{\gamma}$ is nonsingular and the action of $[\widetilde{\gamma}]$ preserves each quadrant at $\widetilde{\gamma}$. 
In this case, we also say that $\widetilde{\gamma}$ is \textbf{orientation-preserving}.

Each stable or unstable leaf in $\mathcal{O}$ contains at most one periodic point.
We will also use the following fact repeatedly in this paper.

\begin{prop} \label{prop:closedorbitliftdiscrete}
Let $\mathcal{C}$ be a finite collection of closed orbits of $\phi$. 
Let $\widetilde{\mathcal{C}}$ be the collection of orbits of $\widetilde{\phi}$ that project down to an orbit in $\mathcal{C}$. 
Then $\widetilde{\mathcal{C}}$ is a discrete $\pi_1 M$-invariant collection of points in $\mathcal{O}$. \qed
\end{prop}

A \textbf{rectangle} in $\mathcal{O}$ is a subspace homeomorphic to $[0,1]^2$ with the restriction of $\mathcal{O}^{s/u}$ identified with the foliation by vertical/horizontal lines respectively. 

Let $R_1$ and $R_2$ be two rectangles. 
We say that $R_2$ is \textbf{taller} than $R_1$ if every $\mathcal{O}^s$-leaf that intersects $R_1$ also intersects $R_2$. 
We say that $R_1$ is \textbf{wider} than $R_2$ if every $\mathcal{O}^u$-leaf that intersects $R_2$ also intersects $R_2$. 
If $R_2$ is taller than $R_1$ and $R_1$ is wider than $R_2$, then we write $R_2 > R_1$ and say that $R_2$ \textbf{lies above} $R_1$.
The reason for this terminology shall become clear in \Cref{subsec:agolgueritaud}.

We say that two points $\widetilde{\gamma}_1$ and $\widetilde{\gamma}_2$ \textbf{span} a rectangle $R$ if they are opposite corners of $R$. In this case we write $R = R(\widetilde{\gamma}_1,\widetilde{\gamma}_2)$.

\subsection{Steadiness} \label{subsec:steady}

Let $\phi$ be a pseudo-Anosov flow on $M$.
Fix a Riemannian metric $g$ on $M \backslash \sing(\phi)$ as in \Cref{defn:paflow}. 
Let $x$ be a point in $M \backslash \sing(\phi)$. 
Pick unit vectors $e^{s/u}$ in $E^{s/u}|_x$, respectively, such that $(e^s, \dot{\phi}, e^u)$ determines a positive basis of $TM|_x$. 
Suppose $\ell$ is a line in $TM|_x$ spanned by a vector $ae^s+b\dot{\phi}+ce^u$.
We say that $\ell$ is \textbf{transverse to $\phi$} if $(a,c) \neq (0,0)$.
In this case, we define the \textbf{slope} of $\ell$ (with respect to $g$) to be $\frac{a}{c} \in \mathbb{R}P^1 = \mathbb{R} \cup \{\infty\}$.
We say that $\ell$ is \textbf{positive/negative} if the slope of $\ell$ is positive/negative, respectively. 
Note that while the slope of $\ell$ depends on the choice of the Riemannian metric $g$, whether $\ell$ is positive/negative depends only on $\ell$.

Now let $c$ be a $1$-manifold, possibly disconnected and with boundary, in $M \backslash \sing(\phi)$.
We say that $c$ is \textbf{transverse to $\phi$} if $c$ is smoothly embedded and its tangent line field $Tc$ is transverse to $\phi$ at every point of $c$.
In this case, we say that $c$ is \textbf{positive/negative} if $Tc$ is positive/negative, respectively, at every point of $c$.
We adopt the convention of illustrating positive 1-manifolds in red and negative 1-manifolds in blue.

Let $c_1$ and $c_2$ be 1-manifolds in $M$. 
We say that a triple $(x,y,t) \in c_1 \times c_2 \times (0,\infty)$ is a time $t$ \textbf{crossing} of $c_2$ over $c_1$ if $y=\phi^t(x)$. 
If $c_1=c_2=c$, we abbreviate this to a time $t$ \textbf{crossing} of $c$. 

\begin{defn}[Steady 1-manifold] \label{defn:steady}
Suppose $e$ is a positive or negative 1-manifold in $M \backslash \sing(\phi)$. 
We say that $e$ is \textbf{steady} if for every crossing $(x,y,t)$ of $c$, we have $|\slope(Te|_y)| > |\slope(d\phi^t(Te|_x))|$.
\end{defn}

Note that steadiness is independent of the choice of the Riemannian metric.
We refer to \cite[Section 3.1]{Tsa24a} for an explanation of the terminology.

\subsection{Contact structures} \label{subsec:bicontact}

Let $M$ be an oriented 3-manifold. 
A \textbf{positive contact form} on $M$ is a 1-form $\alpha_+$ satisfying $\alpha_+ \wedge d\alpha_+ > 0$.
The \textbf{Reeb vector field} of a positive contact form $\alpha_+$ is the unique vector field $R_+$ satisfying $\alpha_+(R_+) = 1$ and $d\alpha_+(R_+, \cdot) = 0$.
A \textbf{positive contact structure} on $M$ is a plane field that can be locally defined as the kernel of a positive contact 1-form.
Similarly, a \textbf{negative contact form} is a 1-form $\alpha_-$ satisfying $\alpha_- \wedge d\alpha_- < 0$.
A \textbf{negative contact structure} is a plane field that can be locally defined as the kernel of a negative contact 1-form.
Note that we allow contact structures that are non-coorientable as plane fields.

Contact structures offer us a convenient way of checking for steadiness.
Recall that a 1-manifold $e$, possibly disconnected and with boundary, is \textbf{Legendrian} with respect to a contact structure $\xi$ if it is tangent to $\xi$ at every point.

\begin{prop} \label{prop:legendriansteady}
Let $\phi$ be a pseudo-Anosov flow. 
Let $U$ be a $\phi^t$-invariant open subset in $M \backslash \sing(\phi)$. 
Suppose $\xi_+$ is a positive contact structure on $U$ such that $T\phi \subset \xi_+$ at every point of $U$.
Then every 1-manifold $e$ in $U$ that is transverse to $\phi$ and Legendrian with respect to $\xi_+$ is negative and steady. 

The symmetric statement holds with `positive' and `negative' interchanged. 
\end{prop}
\begin{proof}
This is only a slightly more general version of \cite[Proposition 3.18]{Tsa24a}, and the same proof carries over:
The positive contact criterion implies that for every $y \in U$, $d\phi^t(\xi_+|_{\phi^{-t}(y)})$ rotates strictly monotonically counterclockwise around $\dot{\phi}|_y$ as $t$ increases.
Thus for any crossing $(x,y,t)$, we have $\slope(Te|_y) < \slope(d\phi^t(Te|_x)) < 0$.
\end{proof}

Conversely, contact structures can be used to build Anosov flows.
For the purposes of this paper, a \textbf{bicontact structure} is a pair $(\xi_+,\xi_-)$ consisting of a positive contact structure $\xi_+$ and a negative contact structure $\xi_-$ that are transverse to each other at every point.
We say that a bicontact structure $(\xi_+,\xi_-)$ \textbf{supports} a flow $\phi$ if $T\phi = \xi_+ \cap \xi_-$ at every point.

Meanwhile, a \textbf{bicontact form} is a pair $(\alpha_+,\alpha_-)$ consisting of a positive contact form $\alpha_+$ and a negative contact form $\alpha_-$ such that $\xi_+ = \ker \alpha_+$ and $\xi_- = \ker \alpha_-$ are transverse to each other at every point.
We say that a bicontact form $(\alpha_+,\alpha_-)$ \textbf{supports} a flow $\phi$ if its associated bicontact structure $(\xi_+, \xi_-)$ supports $\phi$.

\begin{defn}[Strongly adapted] \label{defn:stronglyadapted}
A bicontact form $(\alpha_+,\alpha_-)$ is \textbf{strongly adapted} if the Reeb vector field $R_+$ of $\alpha_+$ is contained in $\xi_-$ at every point.
\end{defn}

\begin{thm}[{\cite[Theorem 1.10]{Hoz24}}]
Let $\phi$ be a smooth flow on a closed oriented 3-manifold $M$.
Then $\phi$ is Anosov with orientable stable and unstable foliations if and only if it is supported by a strongly adapted bicontact form.
\end{thm}

\section{Veering triangulations} \label{sec:veertri}

\subsection{Definition} \label{subsec:veertridefn}

An \textbf{ideal tetrahedron} is a tetrahedon with its 4 vertices removed. The removed vertices are called the \textbf{ideal vertices}. 
An \textbf{ideal triangulation} of a $3$-manifold $M$ is a decomposition of $M$ into finitely many ideal tetrahedra glued along pairs of faces.

A \textbf{taut structure} on an ideal triangulation is a labeling of the dihedral angles by $0$ or $\pi$, such that:
\begin{itemize}
    \item Each tetrahedron has exactly two dihedral angles labeled by $\pi$, and they are opposite to each other.
    \item The angle sum around each edge in the triangulation is $2\pi$.
\end{itemize}

A \textbf{transverse taut structure} is a taut structure along with a coorientation on each face, such that for any edge $e$ with dihedral angle labeled by $0$ in a tetrahedron $t$, exactly one of the faces of $t$ that are adjacent to $e$ is cooriented into $t$.
A \textbf{transverse taut ideal triangulation} is an ideal triangulation with a transverse taut structure.

A \textbf{veering triangulation} is a transverse taut ideal triangulation with a coloring of the edges by red or blue, so that going counterclockwise around the four $0$-labelled edges, starting from an endpoint of a $\pi$-labelled edge, the edges are colored red, blue, red, blue, in that order.
See \Cref{fig:veertet}.

\begin{figure} 
    \centering
    \fontsize{10pt}{10pt}\selectfont
\begingroup%
  \makeatletter%
  \providecommand\color[2][]{%
    \errmessage{(Inkscape) Color is used for the text in Inkscape, but the package 'color.sty' is not loaded}%
    \renewcommand\color[2][]{}%
  }%
  \providecommand\transparent[1]{%
    \errmessage{(Inkscape) Transparency is used (non-zero) for the text in Inkscape, but the package 'transparent.sty' is not loaded}%
    \renewcommand\transparent[1]{}%
  }%
  \providecommand\rotatebox[2]{#2}%
  \newcommand*\fsize{\dimexpr\f@size pt\relax}%
  \newcommand*\lineheight[1]{\fontsize{\fsize}{#1\fsize}\selectfont}%
  \ifx\svgwidth\undefined%
    \setlength{\unitlength}{110.55212186bp}%
    \ifx\svgscale\undefined%
      \relax%
    \else%
      \setlength{\unitlength}{\unitlength * \real{\svgscale}}%
    \fi%
  \else%
    \setlength{\unitlength}{\svgwidth}%
  \fi%
  \global\let\svgwidth\undefined%
  \global\let\svgscale\undefined%
  \makeatother%
  \begin{picture}(1,0.68593175)%
    \lineheight{1}%
    \setlength\tabcolsep{0pt}%
    \put(0,0){\includegraphics[width=\unitlength,page=1]{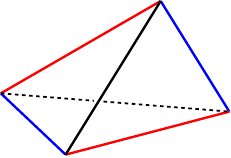}}%
    \put(0.25483533,0.49876369){\color[rgb]{1,0,0}\makebox(0,0)[lt]{\lineheight{1.25}\smash{\begin{tabular}[t]{l}0\end{tabular}}}}%
    \put(0.88822032,0.46067318){\color[rgb]{0,0,1}\makebox(0,0)[lt]{\lineheight{1.25}\smash{\begin{tabular}[t]{l}0\end{tabular}}}}%
    \put(0.63073436,0.00109528){\color[rgb]{1,0,0}\makebox(0,0)[lt]{\lineheight{1.25}\smash{\begin{tabular}[t]{l}0\end{tabular}}}}%
    \put(0.07443572,0.04858473){\color[rgb]{0,0,1}\makebox(0,0)[lt]{\lineheight{1.25}\smash{\begin{tabular}[t]{l}0\end{tabular}}}}%
    \put(0,0){\includegraphics[width=\unitlength,page=2]{veertet.pdf}}%
    \put(0.41844533,0.34464497){\color[rgb]{0,0,0}\makebox(0,0)[lt]{\lineheight{1.25}\smash{\begin{tabular}[t]{l}$\pi$\end{tabular}}}}%
    \put(0.49200224,0.16775368){\color[rgb]{0,0,0}\makebox(0,0)[lt]{\lineheight{1.25}\smash{\begin{tabular}[t]{l}$\pi$\end{tabular}}}}%
  \end{picture}%
\endgroup%

    \caption{A tetrahedron in a veering triangulation. There are no restrictions on the colors of the top and bottom edges.} 
    \label{fig:veertet}
\end{figure}

\subsection{Agol-Guéritaud construction} \label{subsec:agolgueritaud}

Let $\phi$ be a pseudo-Anosov flow on $M$.
Let $\mathcal{C}$ be a nonempty finite collection of closed orbits of $\phi$, containing all the singular orbits.
Let $\widetilde{\mathcal{C}}$ be the collection of orbits of $\widetilde{\phi}$ that project down to an orbit of $\mathcal{C}$.
By \Cref{prop:closedorbitliftdiscrete}, $\widetilde{\mathcal{C}}$ is a $\pi_1 M$-invariant discrete collection of points in $\mathcal{O}$.

A \textbf{perfect fit} in $\mathcal{O}$ is a properly embedded subspace homeomorphic to $[0,1]^2 \backslash \{(1,1)\}$ with the restriction of $\mathcal{O}^{s/u}$ identified with the foliation by vertical/horizontal lines respectively. 
We say that $\phi$ has \textbf{no perfect fits} relative to $\mathcal{C}$ if every perfect fit in $\mathcal{O}$ contains at least one point of $\widetilde{\mathcal{C}}$.
We remark that in this case $\phi$ must be transitive, see \cite[Remark 2.11]{Tsa24}.

When $\phi$ has no perfect fits relative to $\mathcal{C}$, Agol-Guéritaud showed that one can build a veering triangulation from the data of the pair $(\phi, \mathcal{C})$.
To explain this, we introduce a few more definitions.

\begin{defn}[Edge/face/tetrahedron rectangle] \label{defn:edgefacetetrect}
An \textbf{edge rectangle} in $\mathcal{O}$ is a rectangle with two opposite corners on $\widetilde{\mathcal{C}}$. 
An edge rectangle is \textbf{red} if its bottom-left and top-right corners lie in $\widetilde{\mathcal{C}}$, and \textbf{blue} if its top-left and bottom-right corners lie in $\widetilde{\mathcal{C}}$.

A \textbf{face rectangle} in $\mathcal{O}$ is a rectangle with one corner on $\widetilde{\mathcal{C}}$ and the two opposite sides to the corner containing elements of $\widetilde{\mathcal{C}}$ in their interior. 
Note that each face rectangle contains three edge subrectangles.
Two of these edge rectangles have the same color, while the remaining one has the opposite color.

A \textbf{tetrahedron rectangle} in $\mathcal{O}$ is a rectangle all of whose sides contain elements of $\widetilde{\mathcal{C}}$ in their interior. 
Note that each tetrahedron rectangle $T$ contains four face subrectangles and six edge subrectangles.
We refer to the edge subrectangle that is as tall as $T$ as the \textbf{top} edge subrectangle, and the edge subrectangle that is as wide as $T$ as the \textbf{bottom} edge subrectangle.

See \Cref{fig:rectangles} for examples of edge/face/tetrahedron rectangles.
\end{defn}

\begin{figure}
    \centering
    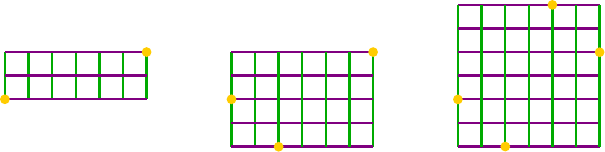
    \caption{From left to right: an edge rectangle, a face rectangle, and a tetrahedron rectangle. Yellow dots denote elements of $\widetilde{\mathcal{C}}$.}
    \label{fig:rectangles}
\end{figure}

The key consequence of the no perfect fit condition is the following.

\begin{prop} \label{prop:astroidlemma}
Suppose $\phi$ has no perfect fits relative to $\mathcal{C}$.
Let $c$ be a point on $\mathcal{O}$. 
Let $S$ be the collection of points $s \in \widetilde{\mathcal{C}}$ in a fixed quadrant of $c$ that span a rectangle with $c$ with no points of $\widetilde{\mathcal{C}}$ in the interior of $R(c,s)$.
The elements of $S$ are totally ordered by the relation $s_1 < s_2$ if $R(c,s_1) < R(c,s_2)$.
Under this order, $S$ is order isomorphic to $\mathbb{Z}$.

Furthermore, suppose $c$ is periodic. Let $k$ be the smallest positive exponent such that $[c]^k$ preserves each quadrant at $c$.
Then $[c]^k$ acts as a translation $m \mapsto m+m_0$ for some $m_0 \in \mathbb{Z}_+$. See \Cref{fig:astroidlemma}.
\end{prop}

\begin{figure}
    \centering
    \fontsize{8pt}{8pt}
\begingroup%
  \makeatletter%
  \providecommand\color[2][]{%
    \errmessage{(Inkscape) Color is used for the text in Inkscape, but the package 'color.sty' is not loaded}%
    \renewcommand\color[2][]{}%
  }%
  \providecommand\transparent[1]{%
    \errmessage{(Inkscape) Transparency is used (non-zero) for the text in Inkscape, but the package 'transparent.sty' is not loaded}%
    \renewcommand\transparent[1]{}%
  }%
  \providecommand\rotatebox[2]{#2}%
  \newcommand*\fsize{\dimexpr\f@size pt\relax}%
  \newcommand*\lineheight[1]{\fontsize{\fsize}{#1\fsize}\selectfont}%
  \ifx\svgwidth\undefined%
    \setlength{\unitlength}{93.07067871bp}%
    \ifx\svgscale\undefined%
      \relax%
    \else%
      \setlength{\unitlength}{\unitlength * \real{\svgscale}}%
    \fi%
  \else%
    \setlength{\unitlength}{\svgwidth}%
  \fi%
  \global\let\svgwidth\undefined%
  \global\let\svgscale\undefined%
  \makeatother%
  \begin{picture}(1,1.00822292)%
    \lineheight{1}%
    \setlength\tabcolsep{0pt}%
    \put(0,0){\includegraphics[width=\unitlength,page=1]{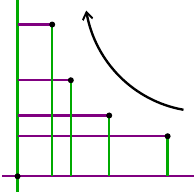}}%
    \put(-0.00487909,0.01136344){\color[rgb]{0,0,0}\makebox(0,0)[lt]{\lineheight{1.25}\smash{\begin{tabular}[t]{l}$c$\end{tabular}}}}%
    \put(0.74603934,0.57084193){\color[rgb]{0,0,0}\makebox(0,0)[lt]{\lineheight{1.25}\smash{\begin{tabular}[t]{l}$[c]$\end{tabular}}}}%
    \put(0.91662799,0.34011138){\color[rgb]{0,0,0}\makebox(0,0)[lt]{\lineheight{1.25}\smash{\begin{tabular}[t]{l}$1$\end{tabular}}}}%
    \put(0.60781791,0.45645051){\color[rgb]{0,0,0}\makebox(0,0)[lt]{\lineheight{1.25}\smash{\begin{tabular}[t]{l}$2$\end{tabular}}}}%
    \put(0.41253832,0.64024369){\color[rgb]{0,0,0}\makebox(0,0)[lt]{\lineheight{1.25}\smash{\begin{tabular}[t]{l}$3$\end{tabular}}}}%
    \put(0.30892133,0.93861256){\color[rgb]{0,0,0}\makebox(0,0)[lt]{\lineheight{1.25}\smash{\begin{tabular}[t]{l}$4$\end{tabular}}}}%
  \end{picture}%
\endgroup%

    \caption{The setting of \Cref{prop:astroidlemma}.}
    \label{fig:astroidlemma}
\end{figure}

\begin{proof}
Let $\ell^s$ be the stable half-leaf of $c$ that borders the fixed quadrant.
We can define an order-preserving map $\pi:S \to \ell^s \cong \mathbb{R}$ by mapping each $s \in S$ to the point of $\ell^s$ that shares an unstable leaf with $s$.
\cite[Lemma 4.10]{SS24} states that the image of $\pi$ has no accumulation points. Thus $S$ must be order isomorphic to $\mathbb{Z}$.

If $c$ is periodic, then $[c]^k$ acts as a translation on $\ell^s \cong \mathbb{R}$. Thus it must act as a translation on $S$ as well.
\end{proof}

We are now ready to state the Agol-Guéritaud construction.

\begin{constr}[Agol-Guéritaud {\cite[Section 4.1]{LMT23}}, {\cite[Section 5.8]{SS24}}] \label{constr:agolgueritaud}
Take an ideal tetrahedron $t_R$ for each tetrahedron rectangle $R$. Fix a bijection between the four ideal vertices of $t_R$ with the four points of $\widetilde{\mathcal{C}}$ on the boundary of $R$. This induces a bijection between:
\begin{itemize}
    \item The six edges of $t_R$ with the six edge subrectangles of $R$.
    \item The four faces of $t_R$ with the four face subrectangles of $R$.
\end{itemize}
Whenever two tetrahedron rectangles $R_1$ and $R_2$ intersect in rectangle that is a face subrectangle of both $R_1$ and $R_2$, we glue $t_{R_1}$ and $t_{R_2}$ along their corresponding faces.
This gives a ideal triangulation $\widetilde{\Delta}$.

We define a veering structure on $\widetilde{\Delta}$ by declaring the edge of $t_R$ corresponding to the top/bottom edge rectangle to be the top/bottom edge, respectively, and coloring an edge red/blue if its corresponding edge rectangle is red/blue, respectively.

The deck transformations $\pi_1 M$ act naturally on $\widetilde{\Delta}$, preserving the veering structure.
In fact, $\pi_1 M$ acts freely on the set of edge/face/tetrahedron rectangles, thus the action of $\pi_1 M$ on the edges/faces/tetrahedra is free, so the quotient $\Delta = \widetilde{\Delta}/\pi_1 M$ is a triangulation.
Moreover, by \Cref{prop:astroidlemma} applied to points in $\widetilde{\mathcal{C}}$, one can show that $\Delta$ is a finite triangulation of a 3-manifold homeomorphic to $M \backslash \mathcal{C}$.

We refer to $\Delta$ as the \textbf{veering triangulation associated to $(\phi,\mathcal{C})$}.
\end{constr}

\subsection{Positioning of veering triangulations} \label{subsec:veertriposition}

In \Cref{constr:agolgueritaud}, we mentioned that the veering triangulation $\Delta$ lives on a 3-manifold homeomorphic to $M \backslash \mathcal{C}$.
In particular, we can transfer $\Delta$ into a triangulation of $M \backslash \mathcal{C}$.
However, the construction does not specify how $\Delta$ interacts with $\phi$ after this transfer.
In this subsection, we will define three increasingly strong ways $\Delta$ can be positioned with respect to $\phi$.

\begin{defn}[Transverse position] \label{defn:transposition}
Let $\phi$ be a pseudo-Anosov flow on $M$ and let $\mathcal{C}$ be a nonempty finite collection of closed orbits of $\phi$.
Let $\Delta$ be a veering triangulation on $M \backslash \mathcal{C}$.
We say that $\Delta$ is in \textbf{transverse position} with respect to $\phi$ if:
\begin{itemize}
    \item each face $f$ of $\Delta$ is positively transverse to the orbits of $\phi$, and
    \item each edge $e$ of $\Delta$ extends to a smoothly embedded compact arc $\overline{e}$ that is transverse to the stable and unstable foliations of $\phi$. \qedhere
\end{itemize}

\end{defn}

Transverse position is sufficient for basic applications of the correspondence theory between pseudo-Anosov flows and veering triangulations.
For example, we have the following existence and uniqueness result.

\begin{thm}[Landry-Minsky-Taylor {\cite[Theorem 5.1]{LMT23}}, Landry-Taylor {\cite[Theorem 6.1]{Tsa24b}}] \label{thm:transpositionexistunique}
Let $\phi$ be a pseudo-Anosov flow on $M$ with no perfect fits relative to $\mathcal{C}$. Then the veering triangulation $\Delta$ associated to $(\phi,\mathcal{C})$ can be placed in transverse position with respect to $\phi$.
Moreover, $\phi$ is the only pseudo-Anosov flow on $M$ with respect to which $\Delta$ can be placed in transverse position, up to orbit equivalence by a map isotopic to identity. \qed
\end{thm}

However, we need the following stronger condition in order to apply the horizontal surgery machinery of \cite{Tsa24a} and \cite{Tsa24b}.

\begin{defn}[Steady position] \label{defn:steadyposition}
Let $\phi$ be a pseudo-Anosov flow with no perfect fits relative to $\mathcal{C}$, and let $\Delta$ be the veering triangulation associated to $(\phi,\mathcal{C})$.
Suppose $\Delta$ is in transverse position with respect to $\phi$.
We further say that $\Delta$ is in \textbf{steady position} with respect to $\phi$ if:
\begin{itemize}
    \item the union of red edges of $\Delta$ is positive and steady, and
    \item the union of blue edges of $\Delta$ is negative and steady. \qedhere
\end{itemize}
\end{defn}

In this paper, we will certify steady position using the following even stronger condition.

\begin{defn}[Legendrian position] \label{defn:legendrianposition}
Let $\phi$ be a pseudo-Anosov flow with no perfect fits relative to $\mathcal{C}$, and let $\Delta$ be the veering triangulation associated to $(\phi,\mathcal{C})$.
Suppose $\Delta$ is in transverse position with respect to $\phi$.

Let $(\xi_+,\xi_-)$ be a bicontact structure on a $\phi^t$-invariant open subset $U \subset M \backslash \mathcal{C}$ that supports the restriction of $\phi$ to $U$.
We say that $\Delta$ is in \textbf{Legendrian position} with respect to $\phi$ and $(\xi_+,\xi_-)$ if:
\begin{itemize}
    \item the union of red edges lies in $U$ and is Legendrian with respect to $\xi_-$, and
    \item the union of blue edges lies in $U$ and is Legendrian with respect to $\xi_+$. \qedhere
\end{itemize}
\end{defn}

\begin{prop}
If $\Delta$ is in Legendrian position then it is in steady position.
\end{prop}
\begin{proof}
If $\Delta$ is in Legendrian position, then the red edges are transverse to $\phi$ and Legendrian to a negative contact structure containing $T\phi$, so by \Cref{prop:legendriansteady}, they are steady.
Similarly, we deduce that the blue edges are steady.
\end{proof}

\section{From edges to triangulations} \label{sec:edgestoveertri}

In this section, we state a set of sufficient criteria for a collection of arcs to be realized as the edges of a veering triangulation. 

\subsection{Edge candidates} \label{subsec:edgecandidate}

We first set up some terminology.
Let $\phi$ be a pseudo-Anosov flow on $M$ with no perfect fits relative to $\mathcal{C}$.

\begin{defn}[Diagonal] \label{defn:diag}
Let $R$ be an edge rectangle in the orbit space $\mathcal{O}$.
A \textbf{diagonal} of $R$ is an arc in $\mathcal{O}$ between the two elements of $\widetilde{\mathcal{C}}$ at the corners of $R$ that is transverse to the foliations $\mathcal{O}^s$ and $\mathcal{O}^u$.
A \textbf{diagonal system} is a collection $\{d_R \mid \text{$R$ is an edge rectangle}\}$ that is equivariant, i.e. $g \cdot d_R = d_{g \cdot R}$ for every $g \in \pi_1 M$.
\end{defn}

\begin{defn}[Edge candidate] \label{defn:edgecandidate}
Let $R$ be an edge rectangle in the orbit space $\mathcal{O}$.
An \textbf{edge candidate} for $R$ is a smoothly embedded arc in $\widetilde{M}$ between the two elements of $\widetilde{\mathcal{C}}$ at the corners of $R$ that projects homeomorphically to a diagonal of $R$.
An \textbf{edge candidate system} is a collection $\{d^\wedge_R \mid \text{$R$ is an edge rectangle}\}$ that is equivariant, i.e. $g \cdot d^\wedge_R = d^\wedge_{g \cdot R}$ for every $g \in \pi_1 M$.
Note that the projection of an edge candidate system onto $\mathcal{O}$ is a diagonal system.
\end{defn}

\subsection{Crossing criterion} \label{subsec:crossingcriterion}

We are now ready to state our criteria. 

\begin{defn}[Face embeddedness criterion] \label{defn:faceembcriterion}
We say that a diagonal system $\{d_R\}$ satisfies the \textbf{face embeddedness criterion} if for every pair of edge rectangles $R_1 < R_2$ that share a corner, $\intr(d_{R_1})$ and $\intr(d_{R_2})$ are disjoint.

We say that an edge candidate system $\{d^\wedge_R\}$ satisfies the \textbf{face embeddedness criterion} if 
for every pair of edge rectangles $R_1 < R_2$ that share a corner, there are no crossings between $\intr(d^\wedge_{R_1})$ and $\intr(d^\wedge_{R_2})$, equivalently, if the diagonal system obtained by projecting $\{d^\wedge_R\}$ to the orbit space satisfies the face embeddedness criterion.
\end{defn}

The following lemma justifies our terminology.

\begin{lem} \label{lem:faceembcriterion}
A diagonal system $\{d_R\}$ satisfies the face embeddedness criterion if and only if for every face rectangle $Q$, the interiors of diagonals $\intr(d_{R_1}), \intr(d_{R_2}), \intr(d_{R_3})$, for the three edge subrectangles $R_1,R_2,R_3 \subset Q$ are disjoint.
\end{lem}
\begin{proof}
Recall that each face rectangle has exactly two edge subrectangles of the same color.
The interiors of the edge subrectangles of opposite color are disjoint, so the interiors of their diagonals are automatically disjoint.
The remaining pair of edge subrectangles that have the same color share a corner, so the interiors of their diagonals are disjoint if the face embeddedness criterion is satisfied.

Conversely, if $R_1 < R_2$ are two edge rectangles that share a corner, then by \Cref{prop:astroidlemma}, there exists a sequence of edge rectangles $R_1 = R'_1 < R'_2 < ... < R'_n = R_2$ such that $R'_i$ and $R'_{i+1}$ are edge subrectangles of a common face rectangle.
If $\intr(d_{R'_i})$ and $\intr(d_{R'_{i+1}})$ are disjoint for each $i$, then $d_{R'_j}$ separates $d_{R'_1},...,d_{R'_{j-1}}$ from $d_{R'_{j+1}},...,d_{R'_n}$ in $\bigcup_{i=1}^n R'_i$ for any $j$. Hence $\intr(d_{R_1})$ and $\intr(d_{R_2})$ are disjoint.
\end{proof}

\begin{defn}[Crossing criterion] \label{defn:edgecrossingcriterion}
We say that an edge candidate system $\{d^\wedge_R\}$ satisfies the \textbf{crossing criterion} if  
the edge candidates $d^\wedge_R$ are smoothly embedded, and for every $R_1 < R_2$, 
every intersection point between the projections of $\intr(d^\wedge_{R_1})$ and $\intr(d^\wedge_{R_2})$ in $\mathcal{O}$ is the projection of a crossing of $\intr(d^\wedge_{R_2})$ over $\intr(d^\wedge_{R_1})$.
\end{defn}

\begin{prop} \label{prop:edgecandidatestoveertri}
Let $\phi$ be a pseudo-Anosov flow on $M$ with no perfect fits relative to $\mathcal{C}$.
Let $\{d^\wedge_R\}$ be an edge candidate system satisfying the face embeddedness and the crossing criteria.
Then the veering triangulation $\Delta$ associated to $(\phi,\mathcal{C})$ can be put in transverse position with respect to $\phi$, with $\intr(d^\wedge_R)$ being the edges of the lifted triangulation $\widetilde{\Delta}$ on $\widetilde{M}$.
\end{prop}
\begin{proof}
The strategy of the proof is to build the triangulation $\widetilde{\Delta}$ in the universal cover $\widetilde{M}$ starting with $\intr(d^\wedge_R)$ as the edges.

For each face rectangle $Q$, let $F_Q$ be the ideal triangle bounded by the projected diagonals in the three edge subrectangles of $Q$, minus the three vertices at $\widetilde{\mathcal{C}}$.
Note that \Cref{lem:faceembcriterion} is used here.
Pick a lift $F^\wedge_Q$ of $F_Q$ in $\widetilde{M}$, such that $F^\wedge_Q$ restricts to $\intr(d^\wedge_R)$ over the interior of the diagonal in each edge subrectangle $R$.

We do this for one representative in each $\pi_1 M$-orbit of face rectangles, and then extend this over all face rectangles equivariantly.
Here, we use the fact that $\pi_1 M$ acts freely on the face rectangles.

We now analyze the intersections between the triangles $F^\wedge_Q$.
By first making a $\pi_1 M$-equivariant perturbation of $F^\wedge_Q$, or equivalently, making a perturbation of the projections of $F^\wedge_Q$ in $M$, then taking the lift of those projections, we can assume that the triangles intersect in a $\pi_1 M$-invariant collection of arcs and curves, and that the number of $\pi_1 M$-orbits of such arcs and curves is finite.
Here, each arc can either end at a point in the interior of a triangle, at a point in the interior of an edge, or at an ideal vertex.

We first isotope the triangles so that the arcs cannot end on the interior of edges:
For every edge rectangle $R$, let $D_R \cong \intr(d^\wedge_R) \times \mathbb{R}$ be the union of orbits of $\widetilde{\phi}$ that pass through $\intr(d^\wedge_R)$.
Each triangle $F^\wedge_Q$ that intersects $D_R$ must do so either in a compact arc, a half-infinite arc, or a bi-infinite arc, depending on whether $Q$ shares zero, one, or two points of $\widetilde{\mathcal{C}}$ with $R$, respectively.
See \Cref{fig:pushfacesoffedges} left. 

\begin{figure}
    \centering
    \fontsize{10pt}{10pt}
\begingroup%
  \makeatletter%
  \providecommand\color[2][]{%
    \errmessage{(Inkscape) Color is used for the text in Inkscape, but the package 'color.sty' is not loaded}%
    \renewcommand\color[2][]{}%
  }%
  \providecommand\transparent[1]{%
    \errmessage{(Inkscape) Transparency is used (non-zero) for the text in Inkscape, but the package 'transparent.sty' is not loaded}%
    \renewcommand\transparent[1]{}%
  }%
  \providecommand\rotatebox[2]{#2}%
  \newcommand*\fsize{\dimexpr\f@size pt\relax}%
  \newcommand*\lineheight[1]{\fontsize{\fsize}{#1\fsize}\selectfont}%
  \ifx\svgwidth\undefined%
    \setlength{\unitlength}{263.56090876bp}%
    \ifx\svgscale\undefined%
      \relax%
    \else%
      \setlength{\unitlength}{\unitlength * \real{\svgscale}}%
    \fi%
  \else%
    \setlength{\unitlength}{\svgwidth}%
  \fi%
  \global\let\svgwidth\undefined%
  \global\let\svgscale\undefined%
  \makeatother%
  \begin{picture}(1,0.38444957)%
    \lineheight{1}%
    \setlength\tabcolsep{0pt}%
    \put(0,0){\includegraphics[width=\unitlength,page=1]{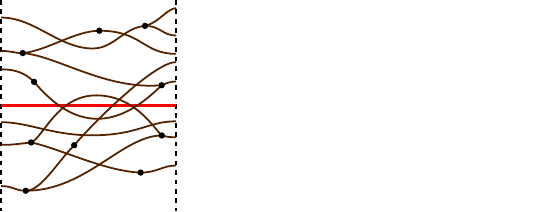}}%
    \put(0.34292924,0.18809271){\color[rgb]{1,0,0}\makebox(0,0)[lt]{\lineheight{1.25}\smash{\begin{tabular}[t]{l}$d^\wedge_R$\end{tabular}}}}%
    \put(0,0){\includegraphics[width=\unitlength,page=2]{pushfacesoffedges.pdf}}%
    \put(0.94168665,0.18808795){\color[rgb]{1,0,0}\makebox(0,0)[lt]{\lineheight{1.25}\smash{\begin{tabular}[t]{l}$d^\wedge_R$\end{tabular}}}}%
    \put(0,0){\includegraphics[width=\unitlength,page=3]{pushfacesoffedges.pdf}}%
  \end{picture}%
\endgroup%

    \caption{Analyzing arcs $F^\wedge_Q \cap D_R$.}
    \label{fig:pushfacesoffedges}
\end{figure}

In the compact arc case, we claim that the endpoints of the arc must both lie below $\intr(d^\wedge_R)$ or both lie above $\intr(d^\wedge_R)$.
Indeed, since $R$ is an edge rectangle and $Q$ is a face rectangle that intersect, either we have $Q < R$ or $Q > R$. Without loss of generality suppose the former is true.
Then also using the fact that $Q$ and $R$ do not share points of $\widetilde{\mathcal{C}}$, we see that $Q$ has exactly two edge rectangles $R'$ and $R''$ that intersect $R$, both of which lie below $R$, as in \Cref{fig:edgerectfacerect}.
Thus applying the crossing criterion, the endpoints of the arc $F^\wedge_Q \cap D_R$, which must lie on $\intr(d^\wedge_{R'})$ and $\intr(d^\wedge_{R''})$ respectively, must lie below $\intr(d^\wedge_R)$.

\begin{figure}
    \centering
    \fontsize{8pt}{8pt}
\begingroup%
  \makeatletter%
  \providecommand\color[2][]{%
    \errmessage{(Inkscape) Color is used for the text in Inkscape, but the package 'color.sty' is not loaded}%
    \renewcommand\color[2][]{}%
  }%
  \providecommand\transparent[1]{%
    \errmessage{(Inkscape) Transparency is used (non-zero) for the text in Inkscape, but the package 'transparent.sty' is not loaded}%
    \renewcommand\transparent[1]{}%
  }%
  \providecommand\rotatebox[2]{#2}%
  \newcommand*\fsize{\dimexpr\f@size pt\relax}%
  \newcommand*\lineheight[1]{\fontsize{\fsize}{#1\fsize}\selectfont}%
  \ifx\svgwidth\undefined%
    \setlength{\unitlength}{93.86233232bp}%
    \ifx\svgscale\undefined%
      \relax%
    \else%
      \setlength{\unitlength}{\unitlength * \real{\svgscale}}%
    \fi%
  \else%
    \setlength{\unitlength}{\svgwidth}%
  \fi%
  \global\let\svgwidth\undefined%
  \global\let\svgscale\undefined%
  \makeatother%
  \begin{picture}(1,0.90125653)%
    \lineheight{1}%
    \setlength\tabcolsep{0pt}%
    \put(0,0){\includegraphics[width=\unitlength,page=1]{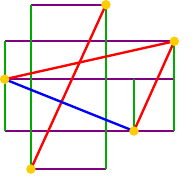}}%
    \put(0.950478,0.57675955){\color[rgb]{1,0,0}\makebox(0,0)[lt]{\lineheight{1.25}\smash{\begin{tabular}[t]{l}$R'$\end{tabular}}}}%
    \put(0.56856296,0.00375611){\color[rgb]{1,0,0}\makebox(0,0)[lt]{\lineheight{1.25}\smash{\begin{tabular}[t]{l}$R$\end{tabular}}}}%
    \put(0.58977362,0.11757004){\color[rgb]{0,0,1}\makebox(0,0)[lt]{\lineheight{1.25}\smash{\begin{tabular}[t]{l}$R''$\end{tabular}}}}%
    \put(0.95761328,0.21723255){\color[rgb]{0,0,0}\makebox(0,0)[lt]{\lineheight{1.25}\smash{\begin{tabular}[t]{l}$Q$\end{tabular}}}}%
  \end{picture}%
\endgroup%

    \caption{If $R$ is an edge rectangle and $Q$ is a face rectangle that intersect, then either we have $Q < R$ or $Q > R$.}
    \label{fig:edgerectfacerect}
\end{figure}
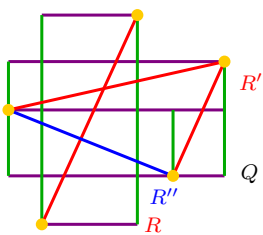

From this claim, we can isotope the triangles, along orbits of $\widetilde{\phi}$ and relative to their edges, so that no triangles intersect $\intr(d^\wedge_R)$, see \Cref{fig:pushfacesoffedges}.
Doing so for one representative $R$ in each $\pi_1 M$-orbit of edge rectangle, then extending it by $\pi_1 M$-equivariance, we will have arranged it so that the triangles $F^\wedge_Q$ only intersect in arcs that end on ideal vertices and in circles.

We can now get rid of these arcs and circles of intersection by cutting and pasting along them, see \Cref{fig:cutandpaste}.
After doing so, we now have a collection of mutually disjoint triangles $F^\wedge_Q$.
We declare these to be the faces of the constructed triangulation.

\begin{figure}
    \centering
    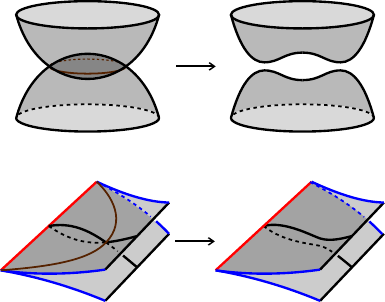
    \caption{We can get rid of arcs and circles of intersection between faces by cutting and pasting along them.}
    \label{fig:cutandpaste}
\end{figure}

For each tetrahedron rectangle $T$, the faces $F^\wedge_Q$ corresponding to the four face subrectangles of $T$ bound a region in $\widetilde{M} \backslash \widetilde{\mathcal{C}}$ foliated by orbit segments of $\widetilde{\phi}$.
Moreover, $\widetilde{\phi}$ enters the region on two faces and exits the region on the remaining two.
Thus the region is an ideal tetrahedron.
This shows that we have constructed an ideal triangulation of $\widetilde{M} \backslash \widetilde{\mathcal{C}}$. 
Moreover, by construction, this is the veering triangulation $\widetilde{\Delta}$.
Taking the quotient by $\pi_1 M$ gives us the proposition.
\end{proof}

\section{From diagonals to edges} \label{sec:diagstoveertri}

\Cref{prop:edgecandidatestoveertri} gives a method of putting veering triangulations in transverse position. However, this is not very wieldy in practice since there is too much freedom in picking arcs in 3-manifolds.
In this section, we build on \Cref{prop:edgecandidatestoveertri} and give another set of sufficient criteria for building veering triangulations, but this time in terms of just the diagonals and also achieving Legendrian position automatically.

\subsection{Birkhoff sections} \label{subsec:birkhoffsection}

The key notion in this section is that of a Birkhoff section.

\begin{defn}[Birkhoff section] \label{defn:birkhoffsection}
A \textbf{Birkhoff section} to a pseudo-Anosov flow $\phi$ is an immersed surface with boundary $S$ where 
\begin{itemize}
    \item the interior $\intr(S)$ is embedded and is transverse to the flow,
    \item the boundary $\partial S$ lies along closed orbits of $\phi$, and
    \item every orbit of $\phi$ meets $S$ in finite forward and backward time.
\end{itemize}
\end{defn}

It is a classical result that Birkhoff sections exist in plenty:

\begin{thm}[{\cite{Bru95}}] \label{thm:birkhoffsectionexist}
Every transitive pseudo-Anosov flow admits a Birkhoff section $S$. 
Moreover, given any finite collection of closed orbits $\mathcal{C}$,
the boundary orbits of $S$ can be chosen to be orientation-preserving and disjoint from $\mathcal{C}$.
\end{thm}
\begin{proof}
This follows from the proof of \cite[Theorem 1]{Bru95}, using the fact that since $\mathcal{C}$ is finite, the set of orientation-preserving closed orbits that do not belong to $\mathcal{C}$ is dense.
\end{proof}

The importance of Birkhoff sections is that they reduce the 3-dimensional dynamics of a pseudo-Anosov flow into the 2-dimensional dynamics of a pseudo-Anosov map:
Let $S$ be a Birkhoff section to a pseudo-Anosov flow $\phi$.
The first item of \Cref{defn:birkhoffsection} implies that the stable/unstable foliations of $\phi$ induce singular 1-dimensional foliations $\ell^{s/u}$ on $\intr(S)$.
The last item of \Cref{defn:birkhoffsection} implies that there is a first return map $f$ on $\intr(S)$, which must preserve the foliations $\ell^{s/u}$.

From this, one can show that $f$ is a pseudo-Anosov map with stable/unstable foliations $\ell^{s/u}$, i.e. there exists transverse measure $\mu^{s/u}$ on $\ell^{s/u}$ such that $f_*\mu^s = \lambda^{-1} \mu^s$, $f_*\mu^u = \lambda \mu^u$ for some $\lambda > 1$ which we refer to as the \textbf{dilatation} of $f$, see \cite[Lemma 14.12]{FLP79}. 

\begin{prop} \label{prop:smoothalmostequivalence}
There is an orbit equivalence $h$ from the suspension flow on $M_f$ to the restriction of $\phi$ to $M \backslash \partial S$. Moreover, $h$ can be arranged to be smooth on $M^\circ_f = M_f \backslash \sing(\phi_f)$.
\end{prop}
\begin{proof}
By an arbitrarily small homotopy along flow lines, we can first ensure that $\intr(S)$ is smoothly embedded in $M \backslash \sing(\phi)$.
Once this is arranged, the first return time $T$ is a smooth function on $\intr(S)$ away from the singular points.
Pick a function $\rho:\{(t,x) \in \mathbb{R} \times \intr(S) \mid t \in [0,T(x)]\} \to \mathbb{R}$ that is smooth away from $[0,T(x)] \times \{x\}$ for singular points $x$, and where $\rho(t,x) = t$ for $t$ close to $0$ or $T(x)$.
We can then extend the embedding $\intr(S) \subset M \backslash \partial S$ into an orbit equivalence $h$ by sending $(t,x)$ to $\phi^{\rho(t,x)}(x)$. This orbit equivalence is smooth away from singular orbits.
\end{proof}

\subsection{Quasi-translation structures} \label{subsec:quasitranslationstructure}

We introduce the following terminology.

\begin{defn}[(Quasi-)translation structure]
For the purposes of this paper, we define a \textbf{translation structure} on a surface $S$ to be an atlas of charts into $\mathbb{R}^2$ with transition functions of the form $(x,y) \mapsto (x+x_0, y+y_0)$.

Building on this, we define a \textbf{quasi-translation structure} on $S$ to be an atlas of charts into $\mathbb{R}^2$ with transition functions of the form $(x,y) \mapsto (\pm \lambda x+x_0, \pm \lambda^{-1} y+y_0)$.
\end{defn}

On a surface equipped with a quasi-translation structure, we can talk about quantities on $\mathbb{R}^2$ that are preserved by homeomorphisms of the form $(x,y) \mapsto (\pm \lambda x+x_0, \pm \lambda^{-1} y+y_0)$.
This includes:
\begin{itemize}
    \item Linear arcs, smooth arcs.
    \item For a point $v$ on a smooth arc $p$, whether $\slope_v(p)$ is positive or negative.
    \item For an intersection point $v$ between two smooth arcs $p$ and $q$, whether $|\slope_v(p)| < |\slope_v(q)|$.
\end{itemize}

We also introduce the following terminology.

\begin{defn}[Piecewise linear arc] \label{defn:plarc}
Let $S$ be a surface equipped with a quasi-translation structure.
A \textbf{piecewise linear arc} on $S$ is an arc $p$ along with a finite collection of points $v_1,...,v_{k-1} \in p$ such that $p$ is a concatenation of linear subarcs $p_1 \ast_{v_1} ... \ast_{v_{k-1}} p_k$.
We refer to the points $v_1,...,v_{k-1}$ as well as the endpoints $v_0,v_k$ of $p$ as the \textbf{nodes}, and the linear subarcs $p_i$ as the \textbf{segments}.

Note that we allow for the possibility that $\slope(p_i) = \slope(p_{i+1})$ at $v_i$. 
In particular, we can always add nodes in the interior of the segments.
This is the reason why we have to specify the nodes as part of the data of a piecewise linear arc.
If $\slope(p_i) \neq \slope(p_{i+1})$ then we say that $v_i$ is a \textbf{turn}.
This subtle point will come up in \Cref{subsec:perturbfortransverse}.

We say that $e$ is \textbf{positive} if every $p_i$ has positive slope.
In this case, we say that $p$ is \textbf{convex} if either $0 < \slope(p_1) < ... < \slope(p_k)$ or $\slope(p_1) > ... > \slope(p_k) > 0$.
Similarly we define \textbf{negative} and \textbf{negative convex} piecewise linear paths.
\end{defn}

\begin{defn}[Overlap of piecewise linear arcs] \label{defn:plarcslopes}
Let $S$ be a surface equipped with a quasi-translation structure.
Let $p$ and $q$ be two piecewise linear arcs in $S$, which are either both positive or both negative.
Suppose $p$ and $q$ overlap along a (possibly degenerate) subarc $r$.
That is, we can write $p = p_1 \ast_{v_1} r \ast_{v_2} p_2$ and $q = q_1 \ast_{v_1} r \ast_{v_2} q_2$, for (possibly degenerate) subarcs $p_1,p_2,q_1,q_2$. Here, by a \textbf{degenerate subarc}, we just mean a point. 
Note that we do not assume that $v_1$ or $v_2$ is a node of $p$ or $q$.

In this setting, we write $|\slope_r(p)| < |\slope_r(q)|$ if either
\begin{itemize}
    \item $p_1$ and $q_1$ are degenerate and $|\slope_{v_2}(p_2)| < |\slope_{v_2}(q_2)|$, or
    \item $p_1,p_2,q_1,q_2$ are all nondegenerate and $|\slope_{v_i}(p_i)| < |\slope_{v_i}(q_i)|$ for $i=1,2$.
\end{itemize}
See \Cref{fig:plarcslopes}.
\end{defn}

\begin{figure}
    \centering
    \fontsize{8pt}{8pt}
\begingroup%
  \makeatletter%
  \providecommand\color[2][]{%
    \errmessage{(Inkscape) Color is used for the text in Inkscape, but the package 'color.sty' is not loaded}%
    \renewcommand\color[2][]{}%
  }%
  \providecommand\transparent[1]{%
    \errmessage{(Inkscape) Transparency is used (non-zero) for the text in Inkscape, but the package 'transparent.sty' is not loaded}%
    \renewcommand\transparent[1]{}%
  }%
  \providecommand\rotatebox[2]{#2}%
  \newcommand*\fsize{\dimexpr\f@size pt\relax}%
  \newcommand*\lineheight[1]{\fontsize{\fsize}{#1\fsize}\selectfont}%
  \ifx\svgwidth\undefined%
    \setlength{\unitlength}{152.28272686bp}%
    \ifx\svgscale\undefined%
      \relax%
    \else%
      \setlength{\unitlength}{\unitlength * \real{\svgscale}}%
    \fi%
  \else%
    \setlength{\unitlength}{\svgwidth}%
  \fi%
  \global\let\svgwidth\undefined%
  \global\let\svgscale\undefined%
  \makeatother%
  \begin{picture}(1,0.41707182)%
    \lineheight{1}%
    \setlength\tabcolsep{0pt}%
    \put(0,0){\includegraphics[width=\unitlength,page=1]{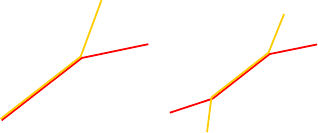}}%
    \put(0.90898868,0.20337502){\color[rgb]{1,0,0}\makebox(0,0)[lt]{\lineheight{1.25}\smash{\begin{tabular}[t]{l}$p_2$\end{tabular}}}}%
    \put(0.35788745,0.2014196){\color[rgb]{1,0,0}\makebox(0,0)[lt]{\lineheight{1.25}\smash{\begin{tabular}[t]{l}$p_2$\end{tabular}}}}%
    \put(0.5537998,0.10729872){\color[rgb]{1,0,0}\makebox(0,0)[lt]{\lineheight{1.25}\smash{\begin{tabular}[t]{l}$p_1$\end{tabular}}}}%
    \put(0.69221798,0.02405288){\color[rgb]{1,0.8,0}\makebox(0,0)[lt]{\lineheight{1.25}\smash{\begin{tabular}[t]{l}$q_1$\end{tabular}}}}%
    \put(0.78830138,0.31458245){\color[rgb]{1,0.8,0}\makebox(0,0)[lt]{\lineheight{1.25}\smash{\begin{tabular}[t]{l}$q_2$\end{tabular}}}}%
    \put(0.20483818,0.33627161){\color[rgb]{1,0.8,0}\makebox(0,0)[lt]{\lineheight{1.25}\smash{\begin{tabular}[t]{l}$q_2$\end{tabular}}}}%
    \put(0.78372618,0.13777585){\color[rgb]{0,0,0}\makebox(0,0)[lt]{\lineheight{1.25}\smash{\begin{tabular}[t]{l}$r$\end{tabular}}}}%
    \put(0.13525772,0.08972179){\color[rgb]{0,0,0}\makebox(0,0)[lt]{\lineheight{1.25}\smash{\begin{tabular}[t]{l}$r$\end{tabular}}}}%
  \end{picture}%
\endgroup%

    \caption{The setting of \Cref{defn:plarcslopes}.}
    \label{fig:plarcslopes}
\end{figure}

\subsection{From Birkhoff sections to quasi-translation structures}

In this subsection, we explain how quasi-translation structures naturally come up when studying Birkhoff sections to pseudo-Anosov flows.

Let $S$ be a Birkhoff section to a pseudo-Anosov flow $\phi$ on $M$ with orbit space $\mathcal{O}$.
Let $f$ be the pseudo-Anosov first return map defined on $\intr(S)$.
By \Cref{prop:smoothalmostequivalence}, there is an orbit equivalence between the suspension flow on $M_f$ and the restriction of $\phi$ to $M \backslash \partial S$.
This induces an orbit equivalence $h$ between the suspension flow on $M^\circ_f = M_f \backslash \sing(\phi_f)$ and the restriction of $\phi$ to $M^\circ = M \backslash (\partial S \cup \sing(\phi))$.
Taking a lift, we have an orbit equivalence $\widetilde{h}$ between the universal covers $\widetilde{M^\circ_f}$ and $\widetilde{M^\circ}$.

The space of orbits in $\widetilde{M^\circ_f}$ can be identified with the universal cover of $S^\circ = \intr(S) \backslash \sing(f)$, while the space of orbits in $\widetilde{M^\circ}$ is the universal cover of the \textbf{punctured orbit space} $\mathcal{O}^\circ = \mathcal{O} \backslash (\widetilde{\partial S} \cup \widetilde{\sing(\phi)})$.
Thus $\widetilde{h}$ induces an identification of universal covers $\widetilde{S^\circ} \cong \widetilde{\mathcal{O}^\circ}$ preserving the stable and unstable foliations.
We refer to this common space, together with the stable and unstable foliations, as the \textbf{translation orbit space}.

The reason for this terminology is that the measures on the stable and unstable foliations of $f$ induce a translation structure on $\widetilde{S^\circ}$.
Each element $g \in \pi_1(M^\circ_f)$ acts on this translation structure by a homeomorphism locally of the form $(x,y) \mapsto (\pm \lambda^k x+x_0, \pm \lambda^{-k} y+y_0)$, where $\lambda$ is the dilatation of $f$ and $k \in \mathbb{Z}$.
We refer to $k$ as the \textbf{height} of $g$.

For $k \in \mathbb{Z}$, let $\pi_1(M^\circ)_k$ be the set of elements $g \in \pi_1(M^\circ)$ with height $k$.
The set $\pi_1(M^\circ)_0$ is exactly the subgroup $\pi_1(S^\circ) < \pi_1(M^\circ_f) \cong \pi_1(M^\circ)$. 
Thus $\pi_1(M^\circ)_0$ acts properly discontinuously on $\widetilde{\mathcal{O}^\circ} \cong \widetilde{S^\circ}$, with quotient $S^\circ$.

One can compactify $S^\circ$ into a compact surface with boundary $\cl(S^\circ)$ by adding a boundary component to each puncture. 
Accordingly, one can extend $\widetilde{\mathcal{O}^\circ}$ into a space $\cl(\widetilde{\mathcal{O}^\circ}) \cong \widetilde{\cl(S^\circ)}$, so that $\pi_1(M^\circ)_0$ also acts properly discontinuously on the extended space, with quotient $\cl(S^\circ)$.

We record the following fact, which will play a key role in our analysis in \Cref{sec:perturbdiag}.

\begin{lem} \label{lem:transorbitspacelocfinite}
Let $K$ be a compact set in $\cl(\widetilde{\mathcal{O}^\circ})$.
Then the collection of $\pi_1(M^\circ)_h$-orbits of $K$ is locally finite in $\cl(\widetilde{\mathcal{O}^\circ})$ for every $h$.
\end{lem}
\begin{proof}
For $h=0$, this follows from the fact that $\pi_1(M^\circ)_0 \cong \pi_1(S^\circ)$ acts properly discontinuously on $\cl(\widetilde{\mathcal{O}^\circ}) \cong \widetilde{\cl(S^\circ)}$.

For general $k$, the set $\pi_1(M^\circ)_k$ is a coset $\pi_1(M^\circ)_0 \cdot g_k$. 
Thus the collection of orbits $\pi_1(M^\circ)_k \cdot K = \pi_1(M^\circ)_0 \cdot (g_k \cdot K)$ is also locally finite.
\end{proof}

Since the deck transformations of $\widetilde{S^\circ} \cong \widetilde{\mathcal{O}^\circ}$ over $\mathcal{O}^\circ$ is the subgroup $\pi_1(\widehat{M^\circ}) < \pi_1(M^\circ) \cong \pi_1(M^\circ_f)$, where $\widehat{M^\circ} := \widetilde{M} \backslash (\widetilde{\partial S} \cup \widetilde{\sing(\phi)})$, there is an induced quasi-translation structure on the punctured orbit space $\mathcal{O}^\circ$. 
In the sequel, we will refer to this structure as the \textbf{quasi-translation structure induced by $S$}.

For these quasi-translation surfaces, we have the following construction of convex piecewise linear arcs.

\begin{constr}[Tight arc] \label{constr:tightarc}
Let $\mathcal{B}$ be a discrete collection of points in $\mathcal{O}$, containing $\widetilde{\partial S} \cup \widetilde{\sing(\phi)}$.
Let $R$ be a rectangle in $\mathcal{O}$. Let $\ell_1$ and $\ell_2$ be the two unstable sides of $R$.
Let $R^!$ be the surface obtained by slitting $R$ along the stable leaf segments lying between points of $\mathcal{B}$ and points on $\ell_2$, see \Cref{fig:tightarc}.
Observe that $R^!$ is simply connected and contained in $\mathcal{O}^\circ$, thus we can lift $R^!$ to $\widetilde{\mathcal{O}^\circ}$ then use the translation structure charts of $\widetilde{\mathcal{O}^\circ}$ to identify $R^!$ with a subset of $\mathbb{R}^2$.

Let $s_1$ and $s_2$ be two opposite corners of $R$, where $s_i \in \ell_i$ for $i=1,2$.
Then the geodesic $p$ between $s_1$ and $s_2$ in $R^! \subset \mathbb{R}^2$ (with respect to the standard flat Riemannian metric on $\mathbb{R}^2$) is a positive convex piecewise linear arc, where we declare the nodes on $p$ to be the points of $\mathcal{B}$ that lie on $p$.

Intuitively, $p$ is the arc that one would get if one places pegs at $\mathcal{B}$, lays a tight piece of string hooked across the unstable side $\ell_1$ from $s_1$ to the corner $w$ and the stable side from $w$ to $s_2$, then releases the hook.
See \Cref{fig:tightarc}.
Motivated by this mental picture, we refer to $p$ as a \textbf{tight arc} between $s_1$ and $s_2$ with respect to $\mathcal{B}$.
We refer to $w$ as the \textbf{hook} of $p$.
\end{constr}

\begin{figure}
    \centering
    \fontsize{8pt}{8pt}
    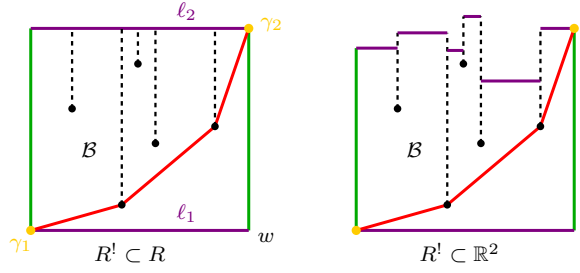
    \caption{Constructing the tight arc between $\gamma_1$ and $\gamma_2$ near $w$ with respect to $\mathcal{B}$}
    \label{fig:tightarc}
\end{figure}

\subsection{Canonical lifts} \label{subsec:canonicallift}

We continue the setting from the previous subsection.
Let $d$ be a smooth arc on $\mathcal{O}^\circ$.
Suppose that $d$ either has positive slope everywhere or negative slope everywhere.
Consider the collection of arcs $\mathcal{D}$ that consists of the $\pi_1 M$-translates of $d$.
Let $\widetilde{\mathcal{D}}$ be the collection of preimages of arcs in $\mathcal{D}$ under the cover $\widetilde{S^\circ} \cong \widetilde{\mathcal{O}^\circ} \to \mathcal{O}^\circ$. 
This is a $\pi_1(M^\circ) \cong \pi_1(M^\circ_f)$-invariant collection.

Recall that $M^\circ_f = \mathbb{R} \times S^\circ/(s,x) \sim (s-1,f(x))$, thus $\widetilde{M^\circ_f} = \mathbb{R} \times S^\circ$.
For each element $\widetilde{d}$ of $\widetilde{\mathcal{D}}$, we pick a smooth parametrization of $\widetilde{d}$, and define the \textbf{canonical lift} of $\widetilde{d}$ (with respect to $S$) to be the curve $\widetilde{d}^\wedge(t) = (\frac{1}{2} \log_\lambda |\slope_{\widetilde{d}(t)}(\widetilde{d})|, \widetilde{d}(t))$ in $\widetilde{M^\circ_f}$.
In particular, $\widetilde{d}^\wedge$ projects to $\widetilde{d}$.
Let $\widetilde{\mathcal{D}}^\wedge$ be the collection of such canonical lifts $\widetilde{d}^\wedge$.

We claim that $g \cdot \widetilde{d}^\wedge = (g \cdot \widetilde{d})^\wedge$ for every $g \in \pi_1(M^\circ_f)$. 
It is clear that both arcs project to $g \cdot \widetilde{d}$, so it suffices to show that the arcs have the same $\mathbb{R}$-coordinate over each point $g \cdot \widetilde{d}(t)$.
The element $g$ acts on $\widetilde{M^\circ_f}$ locally of the form $(s,x,y) \mapsto (s + k, \pm \lambda^k x + x_0, \pm \lambda^{-k} y + y_0)$.
Thus the point of $g \cdot \widetilde{d}^\wedge$ lying over $g \cdot \widetilde{d}(t)$ has $\mathbb{R}$-coordinate $\frac{1}{2} |\log_\lambda \slope_{\widetilde{d}(t)}(\widetilde{d})| + k$, while the point of $(g \cdot \widetilde{d})^\wedge$ lying over $g \cdot \widetilde{d}(t)$ has $\mathbb{R}$-coordinate $\frac{1}{2} \log_\lambda |\slope_{g \cdot \widetilde{d}(t)}(g \cdot \widetilde{d})| = \frac{1}{2} \log_\lambda \left| \lambda^{2k} \slope_{\widetilde{d}(t)}(\widetilde{d}) \right|$, and the two expressions agree.

In particular, this implies that $\widetilde{\mathcal{D}}^\wedge$ is a $\pi_1(M^\circ) \cong \pi_1(M^\circ_f)$-invariant collection as well.
Let $\mathcal{D}^\wedge$ be the image of $\widetilde{\mathcal{D}}^\wedge$ in $\widehat{M^\circ}$.
The projection of $\mathcal{D}^\wedge$ to $\mathcal{O}$ are exactly the $\pi_1 M$-translates of the initial arc $d$.
In particular, there is an element $d^\wedge \in \mathcal{D}^\wedge$ lying over $d$. We refer to it as the \textbf{canonical lift} of $d$ (with respect to $S$).

\begin{constr}[Bicontact form associated to Birkhoff section] \label{constr:surfacebicontact}
We give another perspective on the canonical lift construction in terms of bicontact structures:
The measured foliations $(\ell^s, \mu^s)$ and $(\ell^u, \mu^u)$ on $S$ determine closed 1-forms $ds$ and $du$ on $S^\circ$, locally well-defined up to a sign.
Using these, we can locally define the 1-forms
\begin{align} \label{eq:surfacebicontact}
\begin{split}
\alpha^S_+ &= \lambda^t ds + \lambda^{-t} du \\
\alpha^S_- &= \lambda^t ds - \lambda^{-t} du
\end{split}
\end{align}
on $\mathbb{R} \times S^\circ$.
These are locally positive/negative contact forms, respectively, under the orientation $dtdsdu>0$.
It is straightforward to verify that $\xi^S_+ = \ker \alpha^S_+$ and $\xi^S_- = \ker \alpha^S_-$ are transverse, intersecting in the line spanned by $\frac{\partial}{\partial t}$ at every point.
Thus $(\xi^S_+,\xi^S_-)$ is a bicontact structure, and if $f$ has orientable stable/unstable foliations, so that $\alpha^S_\pm$ are well-defined, then $(\alpha^S_+,\alpha^S_-)$ is a bicontact form.

In the latter case, one can further compute that the Reeb flow to $\alpha^S_+$ is 
$$R_+ = \frac{1}{2} \lambda^{-t} \frac{\partial}{\partial s} + \frac{1}{2} \lambda^t \frac{\partial}{\partial u} \subset \xi^S_-$$
thus $(\alpha^S_+,\alpha^S_-)$ is a strongly adapted bicontact form.

\begin{rmk}
Similarly, one can compute that 
$$R_- = \frac{1}{2} \lambda^{-t} \frac{\partial}{\partial s} - \frac{1}{2} \lambda^t \frac{\partial}{\partial u} \subset \xi^S_+.$$
In the terminology of \cite{Hoz25}, $(\alpha^S_+,\alpha^S_-)$ is \textbf{strongly bi-adapted}. 
However, the filling of the bicontact forms that we will do over $\partial S$ in \Cref{sec:fillinginbicontactform} will destroy the property that $R_- \subset \xi_+$ and only retain strong adaptedness. 
\end{rmk}

Since $f^*ds = \lambda ds$ and $f^*du = \lambda^{-1} du$ up to a sign, $\alpha_+$ and $\alpha_-$ descend to locally defined positive/negative contact forms on $M^\circ_f$.
They determine a bicontact structure $(\xi_+,\xi_-)$ supporting $\phi_f$, and if $f$ has orientable stable/unstable foliations, and $f$ preserves those orientations, so that $\alpha_\pm$ are well-defined, then $(\alpha_+,\alpha_-)$ is a strongly adapted bicontact form supporting $\phi_f$.

By \Cref{prop:smoothalmostequivalence}, the orbit equivalence $M^\circ_f \cong M^\circ$ can be assumed to be smooth.
Thus we can transfer $(\xi_+,\xi_-)$ into a bicontact structure that supports $\phi$ on $M^\circ$.
The canonical lift of $d \subset \mathcal{O}^\circ$ can now be defined as the unique arc in $\widehat{M^\circ}$ lying over $d$ that is Legendrian with respect to the lifted contact structure $\widetilde{\xi_\pm}$ on $\widetilde{M}$, where the sign is negative/positive depending on whether $d$ has positive/negative slope, respectively.
\end{constr}

\begin{rmk} \label{rmk:constantslopeendpoint}
Suppose $\partial S$ is disjoint from $\sing(\phi)$, then the canonical lift construction can be extended to arcs $d$ with endpoints on $\widetilde{\sing(\phi)}$ and where the slope of $\intr(d)$ is constant near the endpoints:
One first takes the canonical lift $\intr(d)^\wedge$ of $\intr(d)$ away from the endpoints.
Since the slope of $\intr(d)$ is constant near the endpoints, each end of $\intr(d)^\wedge$ limits onto a single point on $\widetilde{\sing(\phi)}$.
The \textbf{canonical lift} of $d$ can thus be defined to be the union of $\intr(d)^\wedge$ and these limit points.
\end{rmk}

\subsection{Slope criterion} \label{subsec:slopecriterion}

We are now ready to state our new criterion for diagonals.

\begin{defn}[Slope criterion] \label{defn:slopecriterion}
Let $S$ be a Birkhoff section to $\phi$ with boundary orbits disjoint from $\mathcal{C}$.
We say that a diagonal system $\{d_R\}$ satisfies the \textbf{slope criterion} with respect to $S$ if 
under the induced quasi-translation structure on $\mathcal{O}^\circ$, the interior of every $d_R$ is a smooth arc in $\mathcal{O}^\circ$ with constant slope near its endpoints, and 
for every $R_1 < R_2$, 
\begin{itemize}
    \item if $R_1$ and $R_2$ share a corner, then $\intr(d_{R_1})$ and $\intr(d_{R_2})$ are disjoint, and
    \item if $R_1$ and $R_2$ do not share a corner, then $\intr(d_{R_1})$ and $\intr(d_{R_2})$ intersect at exactly one point $v$, with $|\slope_v(d_{R_1})| < |\slope_v(d_{R_2})|$.
\end{itemize}
\end{defn}

\begin{prop} \label{prop:diagtoveertri}
Let $\phi$ be a pseudo-Anosov flow on $M$ with no perfect fits relative to $\mathcal{C}$.
Let $\{d_R\}$ be a diagonal system satisfying the slope criterion with respect to a Birkhoff section $S$.
Then the veering triangulation $\Delta$ associated to $(\phi,\mathcal{C})$ can be put in Legendrian position with respect to $\phi$ and the bicontact form $(\alpha^S_+,\alpha^S_-)$, with $d_R$ being the projections to $\mathcal{O}$ of the edges of the lift $\widetilde{\Delta}$ to $\widetilde{M}$.
\end{prop}
\begin{proof}
For each $R$, we let $d^\wedge_R$ be the canonical lift of $d_R$. Here we use the fact that $d_R$ has constant slope near its endpoints, so that we can take its canonical lift even if it has an endpoint on a singular point, as explained in \Cref{rmk:constantslopeendpoint}.

The first item in \Cref{defn:slopecriterion} immediately implies that $\{d_R\}$, thus $\{d^\wedge_R\}$, satisfies the face embeddedness criterion.

Next, we verify that $\{d^\wedge_R\}$ satisfies the crossing criterion:
Suppose $R_1 < R_2$ and suppose $v$ is an intersection point between $\intr(d_{R_1})$ and $\intr(d_{R_2})$.
By the slope criterion, we have $|\slope_v(d_{R_1})| < |\slope_v(d_{R_2})|$.
But for $i=1,2$, the $\mathbb{R}$-coordinate of $d^\wedge_{R_i}$ over $v$ is $\frac{1}{2} \log_\lambda |\slope_v(d_{R_i})|$, so $v$ is the projection of a crossing of $d^\wedge_{R_2}$ over $d^\wedge_{R_1}$.

Thus by \Cref{prop:edgecandidatestoveertri}, the veering triangulation $\Delta$ can be placed in transverse position with respect to $\phi$, with $d^\wedge_R$ being the edges of $\widetilde{\Delta}$.
As explained in \Cref{constr:surfacebicontact}, $d^\wedge_R$ are Legendrian with respect to $\widetilde{\alpha^S_\pm}$, lifted from the contact forms $\alpha^S_\pm$ associated to $S$ and which support $\phi$ on $M^\circ$.
Thus $\Delta$ is in fact in Legendrian position with respect to $\phi$.
\end{proof}

\section{Constructing piecewise linear diagonals} \label{sec:pldiag}

The final three sections of this paper is devoted to proving the main theorems \Cref{thm:legendrianposition} and \Cref{thm:steadyposition}.
In this section, we will construct an initial diagonal system.
The precise objective is the following proposition.

\begin{prop} \label{prop:pldiaggoals}
Let $\phi$ be a pseudo-Anosov flow on $M$ with no perfect fits relative to $\mathcal{C}$.
Let $S$ be a Birkhoff section with boundary orbits disjoint from $\mathcal{C}$.
Then there exists a diagonal system $\{d_R\}$ such that:
\begin{enumerate}
    \item The interior of each $d_R$ is a piecewise linear arc in $\mathcal{O}^\circ$, under the quasi-translation structure induced from $S$.
    \item The system $\{d_R\}$ satisfies the following proxy of the slope criterion:
    For every $R_1 < R_2$ of the same color, 
    \begin{itemize}
        \item if $R_1$ and $R_2$ share a corner, then $\intr(d_{R_1})$ and $\intr(d_{R_2})$ are disjoint, and
        \item if $R_1$ and $R_2$ do not share a corner, then $\intr(d_{R_1})$ and $\intr(d_{R_2})$ overlap along a (possibly degenerate) subarc $r$, with $|\slope_r(d_{R_1})| < |\slope_r(d_{R_2})|$.
    \end{itemize}
    \item If $b$ is a node of some diagonal $d_{R_1}$ that lies on some other diagonal $d_{R_2}$, then it is a node of $d_{R_2}$.
    \item If $b_1$ and $b_2$ are adjacent nodes on a diagonal, then $b_1$ and $b_2$ do not lie in the same $\pi_1 M$-orbit.
\end{enumerate}
\end{prop}

We provide a quick outline of our construction:
In \Cref{subsec:anchor}, we will choose a point $\alpha_R$ in the interior of each edge rectangle $R$, which we refer to as the \textbf{anchor}.
This terminology is taken from \cite{LMT23}, but we will strengthen their work in order to choose anchors with better properties.
Writing the elements of $\mathcal{C}$ at the corners of $R$ to be $s_1$ and $s_2$, we refer to the subrectangles $R(s_1,\alpha_R)$ and $R(s_2,\alpha_R)$ as the \textbf{anchor subrectangles}. 

Our diagonals $d_R$ will be the union of two half-diagonals, each contained in an anchor subrectangle.
In turn, each half-diagonal is the tight arc whose hook is the corner that lies on the same stable leaf as the anchor, with respect to a collection $\mathcal{B}$, which we refer to as the \textbf{buoys} (terminology also taken from \cite{LMT23}).
The larger the collection $\mathcal{B}$, the more `tight' the half-diagonals will be.
Thus the diagonals will end up resembling \Cref{fig:pldiagform}.

\begin{figure}
    \centering
    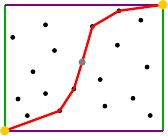
    \caption{Our diagonals $d_R$ will be the union of two half-diagonals, each half-diagonal being a tight arc near near the corner that lies on the same stable leaf as the anchor, with respect to the buoys.}
    \label{fig:pldiagform}
\end{figure}

By choosing enough buoys, we can ensure that the diagonals lie within $\mathcal{O}^\circ$. Then (1) will be true. We will explain this in \Cref{subsec:pldiagawayfromboundaryorbits}.

The bulk of the remaining work is to show (2), regarding the slope inequality.
To this end, we inspect the intersections of anchor subrectangles. 
There are many possible configurations here, but the strategy is the uniform: We tighten the half-diagonals by enlarging the collection of buoys $\mathcal{B}$. 
For the sake of presentation, we divide up the cases into three batches, to be covered from \Cref{subsec:anchorrectcases0} to \Cref{subsec:anchorrectcases2}.

The third and fourth items are technical conditions that will come into play in \Cref{subsec:perturbfortransverse}.
We will arrange for them in \Cref{subsec:pldiaggoalsproof}.

\subsection{Casting anchors} \label{subsec:anchor}

We make the following definition.

\begin{defn}[(Strict) anchor system] \label{defn:anchor}
An \textbf{anchor} for an edge rectangle $R$ is a point contained in the interior of $R$.

An \textbf{anchor system} is a collection $\{\alpha_R \mid \text{$R$ is an edge rectangle}\}$ that is 
\begin{itemize}
    \item equivariant, i.e. $g \cdot \alpha_R = \alpha_{g \cdot R}$, and
    \item \textbf{staircase monotone}, i.e. for every $R_1 < R_2$ that share a corner $s$, $R(s,\alpha_{R_1})$ is wider and strictly shorter than $R(s,\alpha_{R_2})$.
\end{itemize}
This is the same definition as in \cite[Section 5.1.1]{LMT23}. 
Meanwhile, the following definition is new.

A \textbf{strict anchor system} is a collection of one anchor $\alpha_R$ per edge rectangle $R$ that is 
\begin{itemize}
    \item equivariant, and
    \item \textbf{strictly staircase monotone}, i.e. for every $R_1 < R_2$ that share a corner $s$, $R(s,\alpha_{R_1})$ is strictly wider and strictly shorter than $R(s,\alpha_{R_2})$.
    See \Cref{fig:staircasemonotone}.
\end{itemize}
\end{defn}

\begin{figure}
    \centering
    \fontsize{8pt}{8pt}
    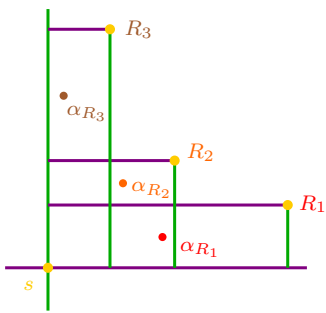
    \caption{Definition of strictly staircase monotone.}
    \label{fig:staircasemonotone}
\end{figure}

In \cite[Lemma 5.10]{LMT23}, Landry-Minsky-Taylor showed that anchor systems exist. 
The goal of this subsection is to prove the following upgrade.

\begin{prop} \label{prop:strictanchorsystem}
There exists a strict anchor system $\{\alpha_R\}$. Moreover, given any finite collection of closed orbits $\mathcal{B}$, the anchors $\alpha_R$ can be chosen to be periodic and disjoint from $\widetilde{\mathcal{B}}$.
\end{prop}

We reuse the following notions in \cite{LMT23} and blackbox some lemmas whose proofs are not important to this paper.

\begin{defn}[Core point, staircase, pinched]
Let $R$ be an edge rectangle. There exists a bi-infinite sequence of edge rectangles $(R_n)_{n \in \mathbb{Z}}$ such that $R_0 = R$, and such that for each $n$, there exists a tetrahedron rectangle $T_n$ such that $R_{n+1}$ is the top edge subrectangle and $R_n$ is the bottom edge subrectangle of $T_n$.
It can be shown that $\bigcap_n R_n$ is a single periodic point, see \cite[Fact 4.4]{LMT23}.
We refer to this point as the \textbf{core point} of $R$ and denote it by $c(R)$.

Let $s$ be a point of $\widetilde{\mathcal{C}}$. 
Consider the collection of edge rectangles that occupy a common quadrant of $s$. 
This collection is totally ordered by the relation `lying above'. 
We refer to the union of these rectangles as a \textbf{staircase} at $s$.
If $R_1,...,R_n$ are consecutive rectangles in the staircase with $c(R_1) = ... = c(R_n)$, we say that $R_1,...,R_n$ are \textbf{pinched}, and also that $c$ is a \textbf{pinched} core point.

If $c$ is a pinched core point, the \textbf{preimage} of $c$ is the collection $\mathcal{O}^\circ$ of edge rectangles that have core point at $c$.
It can be shown that all edge rectangles in $\mathcal{O}^\circ$ have the same color, see \cite[Lemma 5.6]{LMT23}.
Depending on whether these rectangles are red/blue, we say that $c$ is \textbf{red/blue} respectively.
\end{defn}

\begin{lem}[{\cite[Lemma 5.5]{LMT23}}]
For any edge rectangles $Q_1$ and $Q_2$ that share a corner $s$, if $Q_2 > Q_1$, then either $c(Q_1) = c(Q_2)$, or $R(s,c(Q_1))$ is strictly wider and strictly shorter than $R(s,c(Q_2))$.
\end{lem}

In other words, the core points are almost a strict anchor system, except they might overlap for pinched rectangles. 
Thus one needs to push the pinched core points apart.
The following definition gives a permissible region for one to do so.

\begin{defn}[Core box] \label{defn:corebox}
A \textbf{core box system} is a collection of rectangles $\{b(R) \mid \text{$R$ is an edge rectangle}\}$ satisfying:
\begin{itemize}
    \item $c(R) \in b(R) \subset \intr(R)$ for every edge rectangle $R$,
    \item $b(g \cdot R) = g \cdot b(R)$ for every $g \in \pi_1 M$ and edge rectangle $R$, and
    \item for any edge rectangles $R_1$ and $R_2$ that share a corner $s$, if $R_1 < R_2$ and $c(R_1) \neq c(R_2)$, then $R(s,t_1)$ is strictly wider and strictly shorter than $R(s,t_2)$, for every $t_1 \in b(R_1)$ and $t_2 \in b(R_2)$. \qedhere
\end{itemize}
\end{defn}

\begin{lem}[{\cite[Claim 5.7]{LMT23}}]
There exists a core box system.
\end{lem}

Meanwhile, the following lemma gives a framework for the pushing process.

\begin{lem} \label{lem:pinchedpreimagetogrid}
Let $c$ be a pinched core point.
Let $\mathcal{P}$ be the preimage of $c$.
There exists an embedding of partially ordered sets 
$$\iota:\mathcal{P} \hookrightarrow G = 
\begin{cases}
\{(m,n) \in \mathbb{Z}^2 \mid \frac{n_0}{m_0}m-r \leq n \leq \frac{n_0}{m_0}m+r \} & \text{if $c$ is orientation-preserving} \\
\{(m,n) \in \mathbb{Z}^2 \mid m \leq n \leq m+r \} & \text{if $c$ is orientation-reversing}
\end{cases}$$
for some $m_0,n_0,r \in \mathbb{Z}_+$,
where elements of $\mathcal{P}$ are ordered by the relation `lying above' and the element of $G$ are ordered by $(m_1,n_1) \leq (m_2,n_2)$ if $m_1 \leq m_2$ and $n_1 \leq n_2$, such that:
\begin{itemize}
    \item $\iota$ is equivariant under the action of $\langle [c] \rangle$ on $\mathcal{P}$ and the action of $\langle [c] \rangle$ on $G$ with 
    $$[c] \cdot (m,n) =
    \begin{cases}
    (m+m_0,n+n_0) & \text{if $c$ is orientation-preserving,} \\
    (n,m+r) & \text{if $c$ is orientation-reversing.}
    \end{cases}$$
    \item $R_1, R_2 \in \mathbb{P}$ share a corner if and only if $\iota(R_1)$ and $\iota(R_2)$ share a coordinate. 
\end{itemize}
\end{lem}
\begin{proof}
We first suppose that $c$ is orientation-preserving. Let $M,N \subset \widetilde{\mathcal{C}}$ be the two sets of corners of rectangles in $\mathcal{P}$ occupying the two quadrants of $c$.
By \Cref{prop:astroidlemma}, $M$ and $N$ are order isomorphic to $\mathbb{Z}$ under the relation $m_1 < m_2$ if $R(c,m_1) < R(c,m_2)$. 
Moreover, identifying $M,N \cong \mathbb{Z}$, the element $[c]$ acts by $m \mapsto m+m_0$ and $n \mapsto n+n_0$ on $M$ and $N$ respectively, for some $m_0,n_0 \in \mathbb{Z}_+$.
We define the embedding $\iota$ by mapping a rectangle $R$ to its pair of corners in $\widetilde{\mathcal{C}}$.
We demonstrate a local example in \Cref{fig:pinchedpreimagetogrid}.

\begin{figure}
    \centering
    \fontsize{8pt}{8pt}
    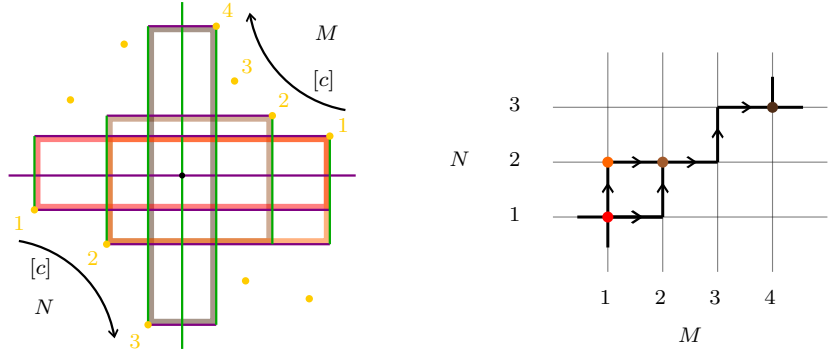
    \caption{Defining the embedding $\iota:\mathcal{P} \to G$.}
    \label{fig:pinchedpreimagetogrid}
\end{figure}

It remains to show that the image of $\iota$ lies in some strip $\{(m,n) \in \mathbb{Z}^2 \mid \frac{m_0}{n_0}n-r \leq m \leq \frac{m_0}{n_0}n+r \}$.
First note that for each $m \in M$, there can only be finitely many elements of $\mathcal{P}$ that have a corner at $m$.
This is because otherwise there will be a staircase $S$ at $m$ with infinitely many consecutive pinched elements. But $\langle [m] \rangle$ acts cofinitely on $S$, and for any $R \in S$, $[m] \cdot R$ cannot be pinched with $R$, since $c([m] \cdot R) = [m] \cdot c(R) \neq c(R)$.
Now using the equivariance under $\langle [c] \rangle$, we deduce that every $m \in M$ can have only boundedly many elements of $\mathcal{P}$ that have a corner at $m$.
A similar statement holds for every $n \in N$. It follows from this that the image of $\iota$ lies within a strip.

The proof when $c$ is orientation-reversing is similar. We outline the key differences: 
We define $M,N$ in the same way. 
Now $[c]$ maps $M$ to $N$ and $N$ to $M$.
We choose the identifications $M,N \cong \mathbb{Z}$ such that the map $[c]:N \to M$ is identity. 
Then the map $[c]:M \to N$ is $m \mapsto m+r$ for some $r \in \mathbb{Z}_+$.
Since $[c] \cdot R > R$ for each $R \in \mathbb{R}$, the image of $\iota$ must lie within the strip $\{(m,n) \in \mathbb{Z}^2 \mid m \leq n \leq m+r \}$.
\end{proof}

\begin{lem} \label{lem:gridtobox}
Fix a core box system $\{b(R)\}$.
Let $c$ be a pinched core point. 
Let $\mathcal{P}$ be the collection of edge rectangles that have core point $c$. 
Then there exists a collection of rectangles $\{b'(R) \mid R \in \mathcal{P}\}$ satisfying
\begin{itemize}
    \item $b'(R) \subset b(R)$,
    \item $[c] \cdot b'(R) = b'([c] \cdot R)$ for every $R \in \mathcal{P}$, and
    \item for every $R_1, R_2 \in \mathcal{P}$ that share a corner $s$, if $R_1 < R_2$, then $R(s,t_1)$ is strictly wider and strictly shorter than $R(s,t_2)$, for every $t_1 \in b'(R_1)$ and $t_2 \in b'(R_2)$.
\end{itemize}
\end{lem}
\begin{proof}
Let $P$ be the union of rectangles in $\mathcal{P}$. There exists an embedding $\Psi:P \hookrightarrow \mathbb{R}^2$ that maps $\mathcal{O}^{s/u}$ to vertical/horizontal lines and is equivariant under the action of $\langle [c] \rangle$ on $\mathcal{P}$ and the action of $\langle [c] \rangle$ on $\mathbb{R}^2$ with 
$$[c] \cdot (x,y) = 
\begin{cases}
(\lambda^{-1} x, \lambda y) & \text{if $c$ is orientation-preserving} \\
(-\lambda^{-1} x, -\lambda y) & \text{if $c$ is orientation-reversing}
\end{cases}$$
for some $\lambda > 1$, as explained in \cite[Claim 5.8]{LMT23}.

Without loss of generality suppose $c$ is red.
We fix an embedding $\mathcal{P} \hookrightarrow G$ as in \Cref{lem:pinchedpreimagetogrid}.
Up to composing $\Psi$ with the reflection $(x,y) \mapsto (y,x)$, we can assume that if $\iota(R_1)$ and $\iota(R_2)$ have the same first/second coordinate, then $\Psi(R_1)$ and $\Psi(R_2)$ share a corner that lies in the first/third quadrant of $0$ in $\mathbb{R}^2$, respectively.

We first suppose that $c$ is orientation-preserving. 
Consider the map $r:\mathbb{R}_+ \times \mathbb{R}_- \to \mathbb{R}^2$ defined by $(x,y) \mapsto (u,v) = (\log_\lambda x - \log_\lambda (-y), - \log_\lambda x - \log_\lambda (-y))$. 
The stable leaves are sent to lines of slope $1$, while the unstable leaves are sent to lines of slope $-1$. 
The action of $\langle [c] \rangle$ is conjugated to the translation $[c] \cdot (u,v) = (u-2,v)$.

We now define an embedding $t:\mathbb{Z}^2 \to \mathbb{R}^2$ by $t(m,n) = (-\frac{m}{m_0}-\frac{n}{n_0},\frac{2m}{m_0}-\frac{2n}{n_0})$.
Observe that the composed map $r^{-1}t$ satisfies the following properties:
\begin{itemize}
    \item $[c] \cdot r^{-1}t(m,n) = r^{-1}t(m+m_0,n+n_0)$ for every $(m,n) \in \mathbb{Z}^2$,
    \item for every $m,n_1,n_2 \in \mathbb{Z}$ with $n_1<n_2$, $R(0,r^{-1}t(m,n_1))$ is strictly thinner and strictly shorter than $R(0,r^{-1}t(m,n_2))$, and
    \item for every $m_1,m_2,n \in \mathbb{Z}$ with $m_1<m_2$, $R(0,r^{-1}t(m_1,n))$ is strictly wider and strictly taller than $R(0,r^{-1}t(m_2,n))$.
\end{itemize}
See \Cref{fig:gridtoboxorientable}.

\begin{figure}
    \centering
    \fontsize{8pt}{8pt}
    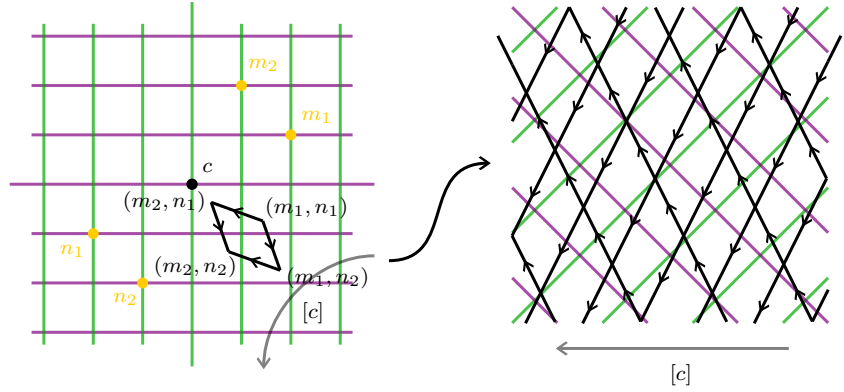
    \caption{Defining the maps $r,t$.}
    \label{fig:gridtoboxorientable}
\end{figure}

This implies that the composed map $T:= \Psi^{-1}r^{-1}t\iota$ satisfies the following properties:
\begin{itemize}
    \item $[c] \cdot T(R) = T([c] \cdot R)$ for every $R \in \mathcal{P}$, and
    \item for every $R_1, R_2 \in \mathcal{P}$ that share a corner $s$, if $R_1 < R_2$, then $R(s,T(R_1))$ is strictly wider and strictly shorter than $R(s,T(R_2))$.
\end{itemize}
Using the fact that the image of $\iota$ lies within a strip $\{(m,n) \in \mathbb{Z}^2 \mid \frac{n_0}{m_0}m-r \leq n \leq \frac{n_0}{m_0}m+r \}$, we can compose $\Psi$ with a dilation $(x,y) \mapsto (\rho x, \rho y)$ for large $\rho$, so that $T(R) \in b(R)$ for every $R \in \mathcal{P}$.
Thus it remains to take $b'(R)$ to be a small rectangle neighborhood of $T(R)$, $\langle [c] \rangle$-equivariantly, for each $R \in \mathcal{P}$.

Next, we tackle the case when $c$ is orientation-reversing.
We set $H = \{(m,n) \in G \mid 0 \leq m < r, n < r\}$, then $H$ is a fundamental domain of the action of $\langle [c] \rangle$ on $G$.
We define the maps $r$ and $t$ as in the orientation-preserving case, but now we set
$$T(R) = \begin{cases}
\Psi^{-1}r^{-1}t\iota(R) & \text{if $\iota(R) \in [c]^{2k}H$} \\
[c]^{-1} \cdot \Psi^{-1}r^{-1}t\iota([c] \cdot R) & \text{if $\iota(R) \in [c]^{2k+1}H$.}
\end{cases}$$
In other words, we define $T$ to be $\Psi^{-1}r^{-1}t\iota$ on $\iota^{-1}(H)$, and extend by equivariance.
We then claim that $T$ satisfies the following properties:
\begin{itemize}
    \item $[c] \cdot T(R) = T([c] \cdot R)$ for every $R \in \mathcal{P}$, and
    \item for every $R_1, R_2 \in \mathcal{P}$ that share a corner $s$, if $R_1 < R_2$, then $R(s,T(R_1))$ is strictly wider and strictly shorter than $R(s,T(R_2))$.
\end{itemize}
The first item follows from our definition by equivariance. 
For the second item, observe that if $\iota(R_1) = (m,n_1)$ and $\iota(R_2) = (m,n_2)$ with $n_1<n_2$, then either $(m,n_1), (m,n_2) \in [c]^k H$, in which case the property is inherited from $\Psi^{-1}r^{-1}t\iota$, or $(m,n_1) \in [c]^{2k-1} H$ and $(m,n_2) \in [c]^{2k} H$, in which case the property holds since $r^{-1}t(m,n_1)$ lies in the second quadrant $\mathbb{R}_- \times \mathbb{R}_+$ while $r^{-1}t(m,n_2)$ lies in the fourth quadrant $\mathbb{R}_+ \times \mathbb{R}_-$.
One can check similarly if $\iota(R_1) = (m_1,n)$ and $\iota(R_2) = (m_2,n)$ with $m_1<m_2$.
See \Cref{fig:gridtoboxnonorientable}.

\begin{figure}
    \centering
    \fontsize{8pt}{8pt}
    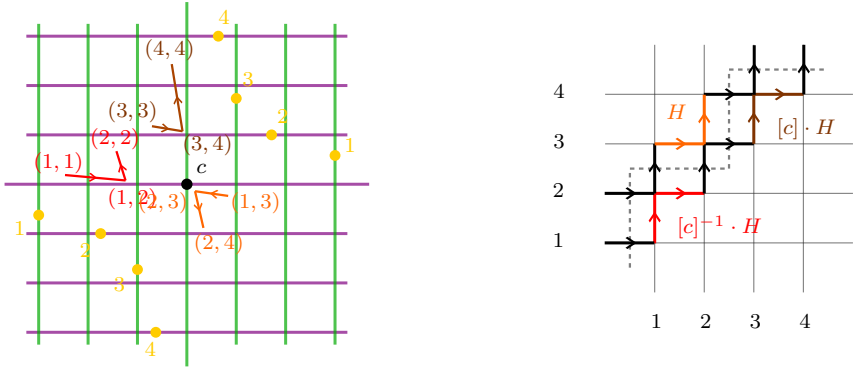
    \caption{The map $T$ in the orientation-reversing case. This figure illustrates the case when $r=2$.}
    \label{fig:gridtoboxnonorientable}
\end{figure}

As in the orientation-preserving case, it remains to compose $\Psi$ with a dilation so that $T(R) \in b(R)$ for every $R \in \mathcal{P}$, and take $b'(R)$ to be a small rectangle neighborhood of $T(R)$, $\langle [c] \rangle$-equivariantly, for each $R \in \mathcal{P}$.
\end{proof}

Putting the core boxes and the smaller boxes in \Cref{lem:gridtobox} together, we deduce the following lemma.

\begin{lem} \label{lem:equivariantsmallbox}
There exists a collection of rectangles $\{b'(R) \mid \text{$R$ is an edge rectangle}\}$ satisfying
\begin{itemize}
    \item $b'(R) \subset \intr(R)$,
    \item $g \cdot b'(R) = b'(g \cdot R)$ for every $g \in \pi_1 M$ and edge rectangle $R$, and
    \item for any edge rectangle $R_1$ and $R_2$ that share a corner $s$, if $R_1 < R_2$, then $R(s,t_1)$ is strictly wider and strictly shorter than $R(s,t_2)$, for every $t_1 \in b'(R_1)$ and $t_2 \in b'(R_2)$.
\end{itemize}
\end{lem}
\begin{proof}
Fix a core box system $\{b(R)\}$. 
For each non-pinched rectangle $R$, we set $b'(R) = b(R)$.
Since each pinched rectangle has a unique pinched core point, we can apply \Cref{lem:gridtobox} to one representative in each $\pi_1 M$-orbit of pinched core points, then extend equivariantly, in order to define $b'(R)$ for all pinched rectangles.
The first two items are immediate. The third item follows from the last item of \Cref{defn:corebox} and the last item in \Cref{lem:gridtobox}. 
\end{proof}

\begin{proof}[Proof of \Cref{prop:strictanchorsystem}]
Fix a collection of rectangles $\{b'(R)\}$ as in \Cref{lem:equivariantsmallbox}.
We pick $\alpha_R$ to be a periodic point in $b'(R)$ disjoint from $\widetilde{\mathcal{B}}$ for a representative $R$ of each $\pi_1 M$-orbit of edge rectangles and extend to all other edge rectangles by equivariance.
\end{proof}

\subsection{Keeping diagonals away from boundary orbits} \label{subsec:pldiagawayfromboundaryorbits}

For the rest of this section, we fix a strict anchor system $\{\alpha_R\}$ that is disjoint from $\widetilde{\partial S}$.
We use the following terminology:
If $R$ is an edge rectangle with corners $s_1, s_2 \in \mathcal{C}$, we refer to the subrectangles $R(s_1,\alpha_R)$ and $R(s_2,\alpha_R)$ as the \textbf{anchor subrectangles}. 

Let $\mathcal{B}$ be a discrete collection of points in $\mathcal{O}$.
We refer to the points in $\mathcal{B}$ as the \textbf{buoys}.
The \textbf{anchor half-diagonal} in an anchor subrectangle $R(s_i,\alpha_R)$ with respect to $\mathcal{B}$ is the tight arc between $s_i$ and $\alpha_R$ whose hook is the corner that lies on the same stable leaf as $\alpha_R$, with respect to $\mathcal{B}$.
The \textbf{anchor diagonal} in $R$ with respect to $\mathcal{B}$ is the union of the two anchor half-diagonals with respect to $\mathcal{B}$. 
If $\mathcal{B}$ is $\pi_1 M$-invariant, then the collection of anchor diagonals with respect to $\mathcal{B}$ is a diagonal system.

We first add enough buoys so that the interior of the anchor diagonals lie within $\mathcal{O}^\circ$. 

\begin{lem} \label{lem:diagonalsoffboundaryorbits}
There exists a discrete $\pi_1 M$-invariant collection of points $\mathcal{B}_0$ such that for every discrete collection of points $\mathcal{B} \supset \mathcal{B}_0$, the interior of the anchor diagonals with respect to $\mathcal{B}$ lie within $\mathcal{O}^\circ$.
\end{lem}
\begin{proof}
Since $\widetilde{\partial S}$ is discrete, for every anchor subrectangle $R'$, we can find a finite collection of periodic points $\mathcal{B}_{R'} \subset \intr(R')$ such that all points of $\widetilde{\partial S} \cap R'$ lie strictly to one side of the anchor half-diagonal in $R'$ with respect to $\mathcal{B}_{R'}$.
This property continues to hold for the anchor half-diagonal in $R'$ with respect to any $\mathcal{B} \supset \mathcal{B}_{R'}$.
In particular, it implies that the interior of the anchor half-diagonal lies within $\mathcal{O}^\circ$.

We take the union of $\mathcal{B}_{R'}$ over one representative in each $\pi_1 M$-orbit of anchor subrectangles, then take its $\pi_1 M$-orbit, in order to get $\mathcal{B}_0$.
Here we use the fact that the anchors are disjoint from $\widetilde{\partial S}$ to ensure that the anchor diagonals have interior disjoint from $\widetilde{\partial S}$, after concatenating the anchor half-diagonals. 
\end{proof}

\subsection{Analyzing half-diagonal intersections} \label{subsec:anchorrectcases0}

In this subsection, we begin analyzing intersection points between diagonals.
Since our diagonals come as the union of two half-diagonals, it is natural to analyze intersection points between half-diagonals instead.
We do so by categorizing pairs of anchor rectangles as follows.

\begin{defn}[Type (0)/(I)/(II) pairs] \label{defn:type012}
Suppose $R_1$ and $R_2$ are edge rectangles of the same color, with $R_1 < R_2$.
Suppose $R'_1 = R(s_1,\alpha_{R_1})$ and $R'_2 = R(s_2,\alpha_{R_2})$ are anchor subrectangles of $R_1$ and $R_2$ that intersect.

We say that the pair $(R'_1,R'_2)$ is of \textbf{type (0)} if $R'_1$ and $R'_2$ overlap along some subsegment of a side.

We say that the pair $(R'_1,R'_2)$ is of \textbf{type (I)} if $R'_1$ and $R'_2$ are not of type (0), and the \textbf{hooked arcs} $h'_i$ obtained by traversing the unstable side of $R'_i$ containing $s_i$ then the stable side of $R'_i$ containing $\alpha_{R_i}$, for $i=1,2$, are disjoint.

We say that the pair $(R'_1,R'_2)$ is of \textbf{type (II)} if $R'_1$ and $R'_2$ are not of type (0), and the hooked arcs $h'_1$ and $h'_2$ defined above intersect.
\end{defn}

The analysis for type (0) pairs is relatively straightforward and only uses the following lemma.

\begin{lem} \label{lem:rectcrosssharecorner}
Let $\mathcal{B}$ be a discrete collection of points in $\mathcal{O}$.
Let $P$ and $Q$ be two rectangles in $\mathcal{O}$, with $Q > P$, and which share a corner $v$.
Let $p$ be a tight arc between $v$ and its opposite corner in $P$ with respect to $\mathcal{B} \cap \intr(P)$, and $q$ be a tight arc between $v$ and its opposite corner in $Q$ with respect to $\mathcal{B} \cap \intr(Q)$.
Suppose that the interiors of $p$ and $q$ lie in $\mathcal{O}^\circ$.

Furthermore, suppose that $v$ and the hooks $w_p$ and $w_q$ of $p$ and $q$ either all lie on the same stable or all lie on the same unstable leaf.
Then $p$ and $q$ overlap along a (possibly degenerate) connected subarc $r$ containing $v$, and are disjoint away from $r$.
Moreover, we have $|\slope_r(p)| < |\slope_r(q)|$.
\end{lem}
\begin{proof}
Without loss of generality, suppose $p$ and $q$ are positive and $v$, $w_p$, $w_q$ lie on the same unstable leaf. Let $R = P \cap Q$ and define $R^!$ as in \Cref{constr:tightarc}. Choose a chart of $R^!$ into $\mathbb{R}^2$ such that $v$ is mapped to the bottom-left corner.
Then $p \cap R$ and $q \cap R$ are geodesics in $R^! \subset \mathbb{R}^2$.
See \Cref{fig:rectcrosssharecorner} left.

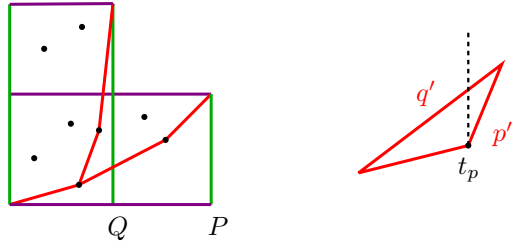
\begin{figure}
    \centering
    \fontsize{10pt}{10pt}
\begingroup%
  \makeatletter%
  \providecommand\color[2][]{%
    \errmessage{(Inkscape) Color is used for the text in Inkscape, but the package 'color.sty' is not loaded}%
    \renewcommand\color[2][]{}%
  }%
  \providecommand\transparent[1]{%
    \errmessage{(Inkscape) Transparency is used (non-zero) for the text in Inkscape, but the package 'transparent.sty' is not loaded}%
    \renewcommand\transparent[1]{}%
  }%
  \providecommand\rotatebox[2]{#2}%
  \newcommand*\fsize{\dimexpr\f@size pt\relax}%
  \newcommand*\lineheight[1]{\fontsize{\fsize}{#1\fsize}\selectfont}%
  \ifx\svgwidth\undefined%
    \setlength{\unitlength}{196.96864331bp}%
    \ifx\svgscale\undefined%
      \relax%
    \else%
      \setlength{\unitlength}{\unitlength * \real{\svgscale}}%
    \fi%
  \else%
    \setlength{\unitlength}{\svgwidth}%
  \fi%
  \global\let\svgwidth\undefined%
  \global\let\svgscale\undefined%
  \makeatother%
  \begin{picture}(1,0.45164408)%
    \lineheight{1}%
    \setlength\tabcolsep{0pt}%
    \put(0,0){\includegraphics[width=\unitlength,page=1]{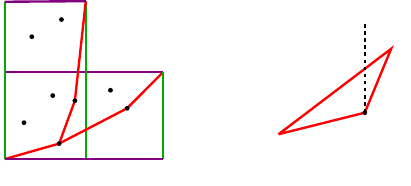}}%
    \put(0.19894198,0.0058454){\color[rgb]{0,0,0}\makebox(0,0)[lt]{\lineheight{1.25}\smash{\begin{tabular}[t]{l}$Q$\end{tabular}}}}%
    \put(0.39065052,0.0058454){\color[rgb]{0,0,0}\makebox(0,0)[lt]{\lineheight{1.25}\smash{\begin{tabular}[t]{l}$P$\end{tabular}}}}%
    \put(0.9362208,0.1848875){\color[rgb]{1,0,0}\makebox(0,0)[lt]{\lineheight{1.25}\smash{\begin{tabular}[t]{l}$p'$\end{tabular}}}}%
    \put(0.7892494,0.26280693){\color[rgb]{1,0,0}\makebox(0,0)[lt]{\lineheight{1.25}\smash{\begin{tabular}[t]{l}$q'$\end{tabular}}}}%
    \put(0.86895633,0.12211473){\color[rgb]{0,0,0}\makebox(0,0)[lt]{\lineheight{1.25}\smash{\begin{tabular}[t]{l}$t_p$\end{tabular}}}}%
  \end{picture}%
\endgroup%

    \caption{Left: The setting of \Cref{lem:rectcrosssharecorner}. Right: There cannot exists subarcs $p' \subset p$ and $q' \subset q$ that cobound a disc.}
    \label{fig:rectcrosssharecorner}
\end{figure}

To show that $p$ and $q$ overlap in a (possibly degenerate) connected subarc $r$ containing $v$, it suffices to show that there cannot exist subarcs $p' \subset p$ and $q' \subset q$ which cobound a disc.
Suppose otherwise, and suppose $p'$ has a turn $t_p = (x_t,y_{t,p})$. Then $q'$ must contain a point $(x_t,y_{t,q})$ with $y_{t,q} \leq y_{t,p}$, since otherwise $q$ would pass through one of the slits in $R \backslash R^!$.
See \Cref{fig:rectcrosssharecorner} right.
In fact we must have $y_{t,q} < y_{t,p}$, otherwise $p'$ and $q'$ would not cobound a disc.
Similarly, if $q'$ has a turn $t_q = (x_t,y_{t,q})$. Then $p'$ must contain a point $(x_t,y_{t,p})$ with $y_{t,p} < y_{t,q}$.
Thus $p'$ and $q'$ cannot both have turns.

Without loss of generality, suppose $p'$ has no turns, so that $p'$ is actually a straight arc. But then $q'$ must also be straight, since otherwise the path obtained by excising $q'$ out of $q$ and replacing it with $p'$ is shorter. 
Thus $p'$ and $q'$ completely overlap and cannot cobound a disc.

We write $p = r \ast_{v_2} p_2$ and $q = r \ast_{v_2} q_2$. 
If $|\slope_{v_2}(p_2)| > |\slope_{v_2}(q_2)|$, then $p_2$ and $q_2$ must intersect somewhere in their interiors. This contradicts the connectedness of the overlap between $p$ and $q$ which we have just proved.
Simiarly, if $|\slope_{v_2}(p_2)| = |\slope_{v_2}(q_2)|$, then $p_2$ and $q_2$ must overlap along some nondegenerate subarc, contradicting the definition of $r$.
Thus $|\slope_{v_2}(p_2)| < |\slope_{v_2}(q_2)|$, so we have $|\slope_r(p)| < |\slope_r(q)|$.
\end{proof}

\begin{lem} \label{lem:anchorrectcases0}
Let $\mathcal{B}$ be a discrete $\pi_1 M$-invariant collection of points.
Let $\{d_R\}$ be the diagonal system consisting of anchor diagonals with respect to $\mathcal{B}$.
Then whenever $(R'_1,R'_2)$ is a type (0) pair of anchor rectangles, the anchor half-diagonals in $R'_1$ and $R'_2$ overlap along a (possibly degenerate) subarc $r$, with $|\slope_r(d_{R_1})|<|\slope_r(d_{R_2})|$.
\end{lem}
\begin{proof}
Write $R'_1 = R(s_1,\alpha_{R_1})$ and $R'_2 = R(s_2,\alpha_{R_2})$. 
Since both $s_i$ and $\alpha_{R_i}$ are periodic, if $R'_1$ and $R'_2$ overlap along some subsegment of a side, we must have either $\alpha_{R_1} = \alpha_{R_2}$ or $s_1 = s_2$.
In the former case, we have $R'_1 < R'_2$ since we are assuming that $R_1 < R_2$, thus we conclude by \Cref{lem:rectcrosssharecorner}.
In the latter case, we have $R'_1 < R'_2$ since $\{\alpha_R\}$ is a strict anchor system, thus we conclude by \Cref{lem:rectcrosssharecorner} as well.
\end{proof}

In the next two subsections, we will analyze the type (I) and type (II) pairs.
Suppose we have such a pair $(R'_1,R'_2)$.
The stable and unstable leaves that contain the sides of $R'_2$ cut the orbit space $\mathcal{O}$ into nine components.
Since $R'_1$ and $R'_2$ do not overlap along any sides, $s_1$ and $\alpha_{R_1}$ must lie in the interior of one of these components.
Thus there are 81 possible configurations here.

Of these configurations, we can rule out all but 8 by using the assumptions that:
\begin{itemize}
    \item $R_1$ is an edge rectangle, thus $s_2$ cannot lie within $R'_1$.
    \item $R_2$ is an edge rectangle, thus $s_1$ cannot lie within $R'_2$.
    \item $R_1<R_2$.
    \item $R_1$ and $R_2$ are of the same color.
    \item $R'_1$ and $R'_2$ intersect.
\end{itemize}

Of the remaining 8 configurations, 4 of them are of type (I).
We exhibit these in \Cref{fig:anchorrectcases} top row (in the case when $R_1$ and $R_2$ are red), and label them type (I-1) - (I-4) accordingly.

The last 4 configurations are of type (II).
We exhibit these in \Cref{fig:anchorrectcases} bottom row (in the case when $R_1$ and $R_2$ are red), and label them type (II-1) - (II-4) accordingly.

\begin{figure}
    \centering
    \fontsize{10pt}{10pt}
    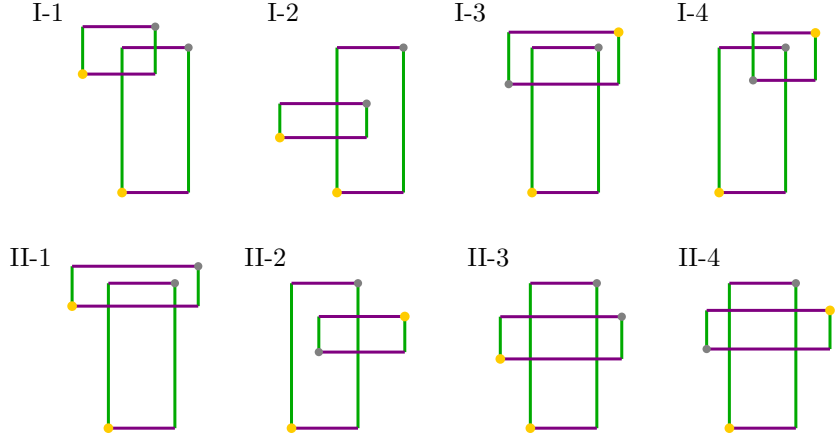
    \caption{The possible configurations of type (I) and (II) pairs.}
    \label{fig:anchorrectcases}
\end{figure}

\subsection{Eliminating avoidable intersections} \label{subsec:anchorrectcases1}

We turn to analyze the type (I) pairs.

\begin{lem} \label{lem:anchorrectcases1}
There exists a discrete $\pi_1 M$-invariant collection of points $\mathcal{B}_0$ such that for every discrete $\pi_1 M$-collection of points $\mathcal{B} \supset \mathcal{B}_0$, if $\{d_R\}$ is the diagonal system consisting of anchor diagonals with respect to $\mathcal{B}$, then whenever $(R'_1,R'_2)$ is a type (I) pair of anchor rectangles, the half-diagonals in $R'_1$ and $R'_2$ are disjoint.
\end{lem}
\begin{proof}
We will argue for each sub type (I-1)-(I-4) one-by-one, using a similar strategy in each case.

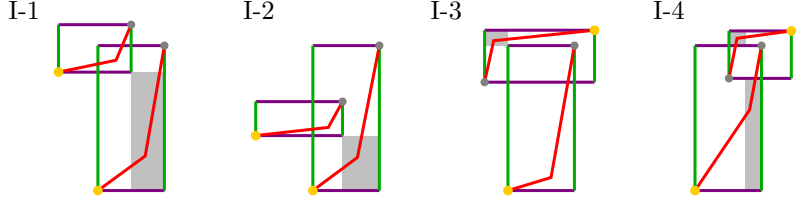
\begin{figure}
    \centering
    \fontsize{10pt}{10pt}
\begingroup%
  \makeatletter%
  \providecommand\color[2][]{%
    \errmessage{(Inkscape) Color is used for the text in Inkscape, but the package 'color.sty' is not loaded}%
    \renewcommand\color[2][]{}%
  }%
  \providecommand\transparent[1]{%
    \errmessage{(Inkscape) Transparency is used (non-zero) for the text in Inkscape, but the package 'transparent.sty' is not loaded}%
    \renewcommand\transparent[1]{}%
  }%
  \providecommand\rotatebox[2]{#2}%
  \newcommand*\fsize{\dimexpr\f@size pt\relax}%
  \newcommand*\lineheight[1]{\fontsize{\fsize}{#1\fsize}\selectfont}%
  \ifx\svgwidth\undefined%
    \setlength{\unitlength}{296.28152442bp}%
    \ifx\svgscale\undefined%
      \relax%
    \else%
      \setlength{\unitlength}{\unitlength * \real{\svgscale}}%
    \fi%
  \else%
    \setlength{\unitlength}{\svgwidth}%
  \fi%
  \global\let\svgwidth\undefined%
  \global\let\svgscale\undefined%
  \makeatother%
  \begin{picture}(1,0.2447777)%
    \lineheight{1}%
    \setlength\tabcolsep{0pt}%
    \put(0,0){\includegraphics[width=\unitlength,page=1]{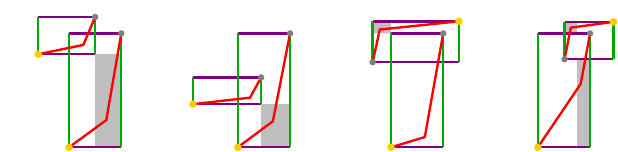}}%
    \put(-0.00180623,0.2247376){\color[rgb]{0,0,0}\makebox(0,0)[lt]{\lineheight{1.25}\smash{\begin{tabular}[t]{l}I-1\end{tabular}}}}%
    \put(0.29689645,0.2247376){\color[rgb]{0,0,0}\makebox(0,0)[lt]{\lineheight{1.25}\smash{\begin{tabular}[t]{l}I-2\end{tabular}}}}%
    \put(0.53484604,0.2247376){\color[rgb]{0,0,0}\makebox(0,0)[lt]{\lineheight{1.25}\smash{\begin{tabular}[t]{l}I-3\end{tabular}}}}%
    \put(0.81825439,0.2247376){\color[rgb]{0,0,0}\makebox(0,0)[lt]{\lineheight{1.25}\smash{\begin{tabular}[t]{l}I-4\end{tabular}}}}%
  \end{picture}%
\endgroup%

    \caption{\Cref{lem:anchorrectcases1} holds if we place buoys in the shaded rectangles.}
    \label{fig:anchorrectbuoy1}
\end{figure}

\textbf{Type (I-1).}
For each type (I-1) pair $(R'_1,R'_2)$, we let $Q(R'_1,R'_2) = R(h_{R'_1},h_{R'_2})$ be the subrectangle of $R'_2$ spanned by the hooks $h_{R'_1}$ and $h_{R'_2}$.
Note that if $\mathcal{B}$ contains a point in $Q(R'_1,R'_2)$, then the half-diagonals in $R'_1$ and $R'_2$ will be disjoint.
See \Cref{fig:anchorrectbuoy1}.
Thus if we can argue that for every fixed $R'_2$, the intersection $Q(R'_2) = \bigcap_{R'_1} Q(R'_1,R'_2)$ taken over all $R'_1$ for which $(R'_1,R'_2)$ is a type (I-1) pair is a (nonempty) rectangle, then we can be done by simply picking $\mathcal{B}_0$ to consist of one periodic point lying in $Q(R'_2)$ for one representative $R'_2$ in each $\pi_1 M$-orbit, then taking the $\pi_1 M$-translates.

To argue that $Q(R'_2)$ is a rectangle, we analyze which $R'_1$ can come up.
If $(R'_1,R'_2)$ is a type (I-1) pair, then the anchor $\alpha_{R_1}$ must lie in the edge rectangle $R_2$, but not on an unstable leaf that passes through $R'_2$.
Since the anchors are periodic, there are only finitely many anchors in $R_2$.
Meanwhile, for each such anchor $\alpha$, there are only finitely many $\langle [\alpha] \rangle$-orbits of edge rectangles that can have $\alpha$ as its anchor.
Hence the collection of rectangles $R'_1$ that can come up for a fixed $R'_2$ lies within a collection of the form $\{[\alpha_i]^k \cdot S_{i,j} \mid i=1,...,p, j=1,...,J_i, k \in \mathbb{Z}\}$, where $\alpha_i$ are the anchors that lie in $R_2$ but not in $R'_2$.

In fact, we can further restrict to a collection of the form $\{[\alpha_i]^k \cdot S_{i,j} \mid i=1,...,p, j=1,...,J_i, k=-K_{i,j},...,K_{i,j} \}$, since for each $i,j$, $[\alpha_i]^k \cdot S_{i,j}$ becomes thinner than $R_2$ for large $k$, and is disjoint from $R'_2$ for large $-k$.
In other words, $Q(R'_2)$ is only a finite intersection of rectangles $Q(R'_1,R'_2)$. Since each $Q(R'_1,R'_2)$ contains the hook $h_{R'_2}$, a finite intersection of them is nonempty, thus a rectangle.

\textbf{Type (I-2).}
For each type (I-2) pair $(R'_1,R'_2)$, we let $Q(R'_1,R'_2) = R(h_{R'_1},h_{R'_2})$.
As in the type (I-1) case, if $\mathcal{B}$ contains a point in $Q(R'_1,R'_2)$, then the half-diagonals in $R'_1$ and $R'_2$ will be disjoint.
See \Cref{fig:anchorrectbuoy1}.
Thus it suffices to argue that every fixed $R'_2$, the intersection $Q(R'_2) = \bigcap_{R'_1} Q(R'_1,R'_2)$ taken over all $R'_1$ for which $(R'_1,R'_2)$ is a type (I-2) pair is a rectangle.

As in the type (I-1) case, we analyze which $R'_1$ can come up.
If $(R'_1,R'_2)$ is a type (I-2) pair, then the anchor $\alpha_{R_1}$ must lie in $R'_2$.
Since the anchors are periodic, there are only finitely many anchors in $R'_2$.
Meanwhile, for each such anchor $\alpha$, there are only finitely many $\langle [\alpha] \rangle$-orbits of edge rectangles that can have $\alpha$ as its anchor.
Hence the collection of rectangles $R'_1$ that can come up for a fixed $R'_2$ lies within a collection of the form $\{[\alpha_i]^k \cdot S_{i,j} \mid i=1,...,p, j=1,...,J_i, k \in \mathbb{Z}\}$, where $\alpha_i$ are the anchors that lie in $R'_2$.

In fact, we can further restrict to a collection of the form $\{[\alpha_i]^k \cdot S_{i,j} \mid i=1,...,p, j=1,...,J_i, k \leq K_{i,j} \}$, since for each $i,j$, $[\alpha_i]^k \cdot S_{i,j}$ becomes thinner than $R'_2$ for large $k$.
Note that \emph{unlike} the type (I-1) case, this collection is infinite.
Nevertheless, since $[\alpha_i]^k \cdot S_{i,j}$ becomes shorter as $k \to -\infty$, the intersection $Q(R'_2) = \bigcap_{R'_1} Q(R'_1,R'_2)$ is still a rectangle.

\textbf{Type (I-3).}
This case is very similar to the type (I-2) case, but with the roles of $R'_1$ and $R'_2$ exchanged.
For each type (I-3) pair $(R'_1,R'_2)$, we let $Q(R'_1,R'_2) = R(h_{R'_1},l_{R'_2})$, where $l_{R'_2}$ is the remaining fourth corner of $R'_2$. 
If $\mathcal{B}$ contains a point in $Q(R'_1,R'_2)$, then the half-diagonals in $R'_1$ and $R'_2$ will be disjoint.
See \Cref{fig:anchorrectbuoy1}.
Thus it suffices to argue that every fixed $R'_1$, the intersection $Q(R'_1) = \bigcap_{R'_2} Q(R'_1,R'_2)$ taken over all $R'_2$ for which $(R'_1,R'_2)$ is a type (I-3) pair is a rectangle

A similar argument as in the type (I-2) case shows that the collection of $R'_2$ that can come up for a fixed $R'_1$ lies within a collection of the form $\{[\alpha_i]^k \cdot S_{i,j} \mid i=1,...,p, j=1,...,J_i, k \geq -K_{i,j} \}$.
From this, we see that $Q(R'_1)$ is a rectangle.

\textbf{Type (I-4).} 
This case is a bit more interesting.
For each type (I-4) pair $(R'_1,R'_2)$, we choose a stable leaf $\ell$ passing through the interior of $R'_1 \cap R'_2$.
We let $Q_1(R'_1,R'_2)$ be the subrectangle of $R_1$ consisting of points lying in the same component of $R_1 \backslash \ell$ as $\alpha_{R'_1}$, but lying outside of $R_2$. Symmetrically, we let $Q_2(R'_1,R'_2)$ be the subrectangle of $R_2$ consisting of points lying in the same component of $R_2 \backslash \ell$ as $\alpha_{R'_2}$, but lying outside of $R_1$.
Then if $\mathcal{B}$ contains a point in $Q_1(R'_1,R'_2)$ and a point in $Q_2(R'_1,R'_2)$, then the half-diagonals in $R'_1$ and $R'_2$ will be disjoint.
See \Cref{fig:anchorrectbuoy1}.
Thus we have to argue that the intersections $Q_1(R'_1) = \bigcap_{R'_2} Q_1(R'_1,R'_2)$ and $Q_2(R'_2) = \bigcap_{R'_1} Q_2(R'_1,R'_2)$ are rectangles.

A priori this could be tricky without choosing $\ell$ more carefully.
Fortunately, each intersection, say the one for $Q_1(R'_1)$, is taken over a finite collection $\{[\alpha_i]^k \cdot S_{i,j} \mid i=1,...,p, j=1,...,J_i, k=-K_{i,j},...,K_{i,j} \}$, where $\alpha_i$ are the anchors in $R'_1$, since for each $i,j$, $[\alpha_i]^k \cdot S_{i,j}$ is thinner than $R'_1$ for large $k$, and is shorter than $R'_1$ for large $-k$.
Thus $Q_1(R'_1)$ are $Q_2(R'_2)$ are indeed rectangles.
\end{proof}

\subsection{Tightening remaining half-diagonals} \label{subsec:anchorrectcases2}

We turn to the type (II) pairs.
The analysis in this case is simple once we prove the following lemma.

\begin{lem} \label{lem:rectcrossnosharecorner}
Let $\mathcal{B}$ be a discrete collection of points in $\mathcal{O}$.
Let $P$ and $Q$ be two rectangles in $\mathcal{O}$, with $Q > P$, and which do not share a corner.
Let $p$ be a tight arc between two opposite corner in $P$ with respect to $\mathcal{B} \cap \intr(P)$, and $q$ be a tight arc between two opposite corner in $Q$ with respect to $\mathcal{B} \cap \intr(Q)$.
Suppose that the interiors of $p$ and $q$ lie in $\mathcal{O}^\circ$.

Suppose $p$ and $q$ are of the same color. 
Then $p$ and $q$ overlap along a (possibly degenerate) connected subarc $r$, and are disjoint away from $r$.
Moreover, we have $|\slope_r(p)| < |\slope_r(q)|$.
\end{lem}
\begin{proof}
If the unstable side of $P$ containing the hook of $p$ lies closer to the unstable side of $Q$ containing the hook of $q$ than the other unstable side of $P$, then the same argument as in \Cref{lem:rectcrosssharecorner} works, since we can still consider the slit rectangle $(P \cap Q)^!$. So in the following we restrict to the remaining case when the unstable side of $P$ containing the hook of $p$ lies further from the unstable side of $Q$ containing the hook of $q$ than the other unstable side of $P$.
For ease of reference, we position the rectangles so that the hook of $p$ is the top-left corner of $P$ and the hook of $q$ is the bottom-right corner of $Q$. See \Cref{lem:rectcrossnosharecorner} left.

\begin{figure}
    \centering
    \fontsize{10pt}{10pt}
    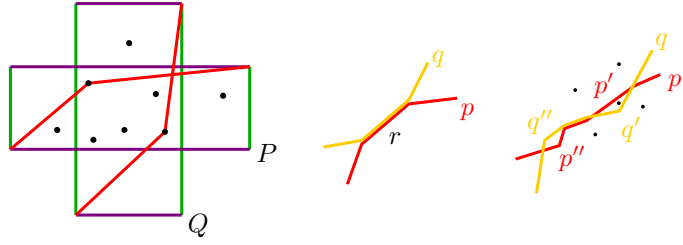
    \caption{Left: The setting of \Cref{lem:rectcrossnosharecorner}. Middle: By convexity, $p$ cannot osculate $q$ from above. Right: There cannot exists subarcs $p' \subset p$ and $q' \subset q$ that cobound a disc.}
    \label{fig:rectcrossnosharecorner}
\end{figure}

We first claim that there cannot exist subarcs $p' \subset p$ and $q' \subset q$ which cobound a disc, with $q'$ lying above $p'$.
Suppose otherwise, then let $R$ be the rectangle with corners at the common endpoints of $p'$ and $q'$.
Since all points of $\partial S \cap R$ have to lie below $p'$ and above $q'$, we can consider the set $R^! \subset \mathcal{O}^\circ$ obtained by slitting the points of $\partial S \cap R$ lying below $p'$ downwards, and slitting the points lying above $q'$ upwards.
Choose a chart of $R^!$ into $\mathbb{R}^2$, both $p'$ and $q'$ have to be geodesics, hence they cannot bound a disc, contradiction.

Next, suppose $p$ and $q$ overlap along a subarc $r$.
Since $P$ and $Q$ do not share a corner, $r$ does not contain an endpoint of $p$ or $q$, so we can write $p$ as a concatenation $p_1 \ast_{v_1} r \ast_{v_2} p_2$ and $q$ as a concatenation $q_1 \ast_{v_1} r \ast_{v_2} q_2$ near $r$.
Let us first suppose that $r$ is degenerate, then since $p$ and $q$ are convex, we have $\slope_r(p_1) \geq \slope_r(p_2)$ and $\slope_r(q_1) \leq \slope_r(q_2)$.
In particular, we cannot have both $\slope_r(p_1) < \slope_r(q_1)$ and $\slope_r(p_2) > \slope_r(q_2)$.
That is, $p$ cannot osculate $q$ from above at $r$.

The same conclusion holds when $r$ is nondegenerate, for we have $\slope_{v_1}(p_1) \geq \slope_{v_1}(r)$ and $\slope_{v_2}(r) \geq \slope_{v_2}(p_2)$, and $\slope_{v_1}(q_1) \leq \slope_{v_1}(r)$ and $\slope_{v_2}(r) \leq \slope_{v_2}(q_2)$ by convexity of $p$ and $q$.
See \Cref{fig:rectcrossnosharecorner} middle.

Using this observation, we claim that there cannot exist subarcs $p' \subset p$ and $q' \subset q$ which cobound a disc, with $q'$ lying below $p'$ as well.
Suppose otherwise, then $p$ and $q$ must cross each other at the two endpoints of $p'$ and $q'$, possibly after overlapping along some segment. But since $P < Q$, we must see subarcs $p'' \subset p$ and $q'' \subset q$ cobounding a disc elsewhere, with $q''$ lying above $p''$. 
See \Cref{fig:rectcrossnosharecorner} right.
This contradicts what we have proven above.

Thus we conclude that $p$ and $q$ overlap in a (possibly degenerate) connected subarc $r$.
The same argument as in \Cref{lem:rectcrosssharecorner} shows that we have $|\slope_r(p)| < |\slope_r(q)|$.
\end{proof}

\begin{lem} \label{lem:anchorrectcases2}
There exists a discrete $\pi_1 M$-invariant collection of points $\mathcal{B}_0$ such that for every discrete $\pi_1 M$-collection of points $\mathcal{B} \supset \mathcal{B}_0$, if $\{d_R\}$ is the diagonal system consisting of anchor diagonals with respect to $\mathcal{B}$, then whenever $(R'_1,R'_2)$ is a type (II) pair of anchor rectangles, the anchor half-diagonals in $R'_1$ and $R'_2$ overlap along a (possibly degenerate) subarc $r$ containing $s_1 = s_2$, with $|\slope_r(d_{R_1})|<|\slope_r(d_{R_2})|$.
\end{lem}
\begin{proof}
We will argue for each sub type (II-1)-(II-4) as in \Cref{lem:anchorrectcases1}.

\begin{figure}
    \centering
    \fontsize{10pt}{10pt}
\begingroup%
  \makeatletter%
  \providecommand\color[2][]{%
    \errmessage{(Inkscape) Color is used for the text in Inkscape, but the package 'color.sty' is not loaded}%
    \renewcommand\color[2][]{}%
  }%
  \providecommand\transparent[1]{%
    \errmessage{(Inkscape) Transparency is used (non-zero) for the text in Inkscape, but the package 'transparent.sty' is not loaded}%
    \renewcommand\transparent[1]{}%
  }%
  \providecommand\rotatebox[2]{#2}%
  \newcommand*\fsize{\dimexpr\f@size pt\relax}%
  \newcommand*\lineheight[1]{\fontsize{\fsize}{#1\fsize}\selectfont}%
  \ifx\svgwidth\undefined%
    \setlength{\unitlength}{310.35983348bp}%
    \ifx\svgscale\undefined%
      \relax%
    \else%
      \setlength{\unitlength}{\unitlength * \real{\svgscale}}%
    \fi%
  \else%
    \setlength{\unitlength}{\svgwidth}%
  \fi%
  \global\let\svgwidth\undefined%
  \global\let\svgscale\undefined%
  \makeatother%
  \begin{picture}(1,0.2215599)%
    \lineheight{1}%
    \setlength\tabcolsep{0pt}%
    \put(0,0){\includegraphics[width=\unitlength,page=1]{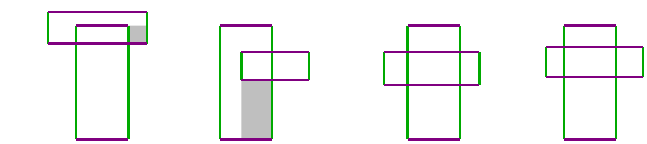}}%
    \put(-0.0017243,0.20242891){\color[rgb]{0,0,0}\makebox(0,0)[lt]{\lineheight{1.25}\smash{\begin{tabular}[t]{l}II-1\end{tabular}}}}%
    \put(0.28342885,0.20242891){\color[rgb]{0,0,0}\makebox(0,0)[lt]{\lineheight{1.25}\smash{\begin{tabular}[t]{l}II-2\end{tabular}}}}%
    \put(0.55408262,0.20242891){\color[rgb]{0,0,0}\makebox(0,0)[lt]{\lineheight{1.25}\smash{\begin{tabular}[t]{l}II-3\end{tabular}}}}%
    \put(0.8101359,0.20242891){\color[rgb]{0,0,0}\makebox(0,0)[lt]{\lineheight{1.25}\smash{\begin{tabular}[t]{l}II-4\end{tabular}}}}%
    \put(0,0){\includegraphics[width=\unitlength,page=2]{anchorrectbuoy2.pdf}}%
  \end{picture}%
\endgroup%

    \caption{\Cref{lem:anchorrectcases2} holds if we place buoys in the shaded rectangles.}
    \label{fig:anchorrectbuoy2}
\end{figure}

\textbf{Type (II-1).}
For each type (II-1) pair $(R'_1,R'_2)$, we let $Q(R'_1,R'_2) = R(h_{R'_1},\alpha_{R'_2})$ be the subrectangle of $R'_2$ spanned by the hook $h_{R'_1}$ and the anchor $\alpha_{R'_2}$.
If $\mathcal{B}$ contains a point $b$ in $Q(R'_1,R'_2)$, then we can apply \Cref{lem:rectcrossnosharecorner} to the rectangles $R(s_1,b) < R'_2$ in order to conclude.
See \Cref{fig:anchorrectbuoy2}.
Thus it suffices to argue that for every fixed $R'_1$, the intersection $Q(R'_1) = \bigcap_{R'_2} Q(R'_1,R'_2)$ taken over all $R'_2$ for which $(R'_1,R'_2)$ is a type (II-1) pair is a rectangle.

But the set of such $R'_2$ is contained in a collection of the form $\{[\alpha_i]^k \cdot S_{i,j} \mid i=1,...,p, j=1,...,J_i, k \geq -K_{i,j} \}$, where $\alpha_i$ are the anchors in $R'_1$, since for each $i,j$, $[\alpha_i]^k \cdot S_{i,j}$ is shorter than $R'_2$ for large $-k$.
From this, we deduce that $Q(R'_1)$ is a rectangle.

\textbf{Type (II-2).}
This case is similar to the type (II-1) case, but with the roles of $R'_1$ and $R'_2$ interchanged. We omit the details.

\textbf{Type (II-3) and (II-4).}
If $(R'_1,R'_2)$ is of type (II-3) or (II-4), then $R'_2 < R'_1$ and we can directly apply \Cref{lem:rectcrossnosharecorner}. Note that in these cases we do not have to specify a collection $\mathcal{B}_0$.
\end{proof}

\subsection{Putting the diagonals back together} \label{subsec:pldiaggoalsproof}

\begin{proof}[Proof of \Cref{prop:pldiaggoals}]
Let $\mathcal{B}$ be the union of the discrete $\pi_1 M$-invariant collections of points in \Cref{lem:diagonalsoffboundaryorbits}, \Cref{lem:anchorrectcases1}, and \Cref{lem:anchorrectcases2}.
Let $\{d_R\}$ be the diagonal system consisting of anchor diagonals with respect to $\mathcal{B}$.
Then (1) is true by construction.

Next, we show (2).
Suppose $R_1 < R_2$ are two edge rectangles of the same color. 
If $R_1$ and $R_2$ share a corner $s$, then the pair of anchor subrectangles $R'_1 \subset R_1$ and $R'_2 \subset R_2$ that contain $s$ intersect and is of type (0). Aside from this pair, the only other pair of anchor subrectangles that could possibly intersect are the other anchor subrectangles $R''_1 \subset R_1$ and $R''_2 \subset R_2$, and they would be of type (I-1).
Thus in this case (2) follows from \Cref{lem:anchorrectcases0} and \Cref{lem:anchorrectcases1}.

If $R_1$ and $R_2$ share an anchor, then two pairs of anchor subrectangles intersect and both are of type (0).
Thus in this case (2) follows from \Cref{lem:anchorrectcases0}.

In the remaining cases, the pairs of anchor subrectangles can only be of types (I) or (II).
In fact, exactly one pair will be of type (II), since the union of the two hooked arcs (defined in \Cref{defn:type012}) in the two anchor subrectangles of $R_i$, for $i=1,2$, intersect exactly once.
Thus in this case (2) follows from \Cref{lem:anchorrectcases1} and \Cref{lem:anchorrectcases2}.

Finally, we address (3) and (4). 
Since the half-diagonals are tight paths with respect to $\mathcal{B}$, by definition their nodes are exactly the point that lie on $\mathcal{B}$. Thus (3) holds.
Suppose there is a segment $\sigma$ whose endpoints are nodes that lie in the same $\pi_1 M$-orbit. We can add a node at any point in the interior of $\sigma$. Since the $\pi_1 M$-orbits of the existing nodes is a countable set, while $\sigma$ consists of uncountably many points, we can choose the new node so that $\sigma$ is split into two segments, each satisfying (4).

A priori one has to worry about the new node violating (3). This only happens if the new node is chosen to be a point lying on a segment $\sigma'$ intersecting $\sigma$ transversely. But since there are only countably many segments in total, there are only countably many choices of the new node that can result in this, so this situation can be avoided as well.
\end{proof}

\section{Perturbing the diagonals} \label{sec:perturbdiag}

In this section, we explain how the diagonal systems of \Cref{prop:pldiaggoals} can be perturbed to satisfy the slope criterion.

\subsection{Perturbing for transversality} \label{subsec:perturbfortransverse}

We refer to a diagonal system $\{d_R\}$ satisfying \Cref{prop:pldiaggoals}(1)-(4) as a \textbf{PLO (Piecewise Linear Overlapping) diagonal system}.
Observe that \Cref{prop:pldiaggoals}(3) allows us to talk about segments of the diagonal system unambiguously.

\begin{defn}[Overlap] \label{defn:overlap}
Suppose $\sigma$ is a nondegenerate segment in the intersection of two diagonals $d_{R_1}$ and $d_{R_2}$.
Then the rectangles $R_1$ and $R_2$ are of the same color and they intersect. Up to swapping $R_1$ and $R_2$, we can assume that $R_1<R_2$.
We say that the triple $(d_{R_1},d_{R_2},\sigma)$ is an \textbf{overlap} of $d_{R_2}$ over $d_{R_1}$ along $\sigma$.
\end{defn}

\begin{lem} \label{lem:overlapfinite}
For any PLO diagonal system, the set of $\pi_1 M$-orbits of overlaps is finite.
\end{lem}
\begin{proof}
Lift the diagonals $\{d_R\}$ to the translation orbit space $\mathcal{O}^\circ$. Here we can talk about the exact values of slopes.
Since there are only finitely many $\pi_1 M$-orbits of edge rectangles, thus diagonals in $\mathcal{O}$, so there are only finitely many $\pi_1 M^\circ$-orbits of diagonals in $\mathcal{O}^\circ$. 
Let $\widetilde{d}_1,...,\widetilde{d}_r$ be a collection of one representative in each $\pi_1 M^\circ$-orbit.
To show the lemma, it suffices to show that for each $i=1,...,r$, there are only finitely many diagonals that can form an overlap with $\widetilde{d}_i$.

To this end, pick a large $N$ so that $\lambda^{2N} > \frac{\max_i \max|\slope(d_i)|}{\min_i \min|\slope(d_i)|}$, where $\lambda$ is the dilatation of the pseudo-Anosov first return map $f$, and $\max|\slope(d_i)|$ and $\min|\slope(d_i)|$ denotes the maximum and minimum, respectively, over $|\slope(\sigma)|$ where $\sigma$ is a segment of $d_i$. 
Then for every overlap $(\widetilde{d}_i,g \cdot \widetilde{d}_j,\sigma)$, we must have $|\hght(g)| \leq N$. 
Meanwhile, by \Cref{lem:transorbitspacelocfinite}, the collection $\pi_1(M^\circ)_{[-N,N]} \cdot \widetilde{d}_j = \bigcup_{h=-N}^N \pi_1(M^\circ)_h \cdot \widetilde{d}_j$ is locally finite, so only finitely many diagonals can form an overlap with $\widetilde{d}_i$.
\end{proof}

The content of this subsection is to show that one can inductively reduce the number of $\pi_1 M$-orbits of pairs $(R_1,R_2)$ where there is an overlap of $d_{R_2}$ over $d_{R_1}$. 
Once this is arranged, we have a diagonal system consisting of piecewise linear arcs that intersect in points.

\begin{defn}[Peripheral, elementary overlap] \label{defn:peripheralelementaryoverlap}
We say that an overlap $(d_{R_1},d_{R_2},\sigma)$ is 
\begin{itemize}
    \item \textbf{peripheral} if $\sigma$ contains an endpoint of the intersection $d_{R_1} \cap d_{R_2}$ that is not a point of $\widetilde{\mathcal{C}}$, and
    \item \textbf{elementary} if there does not exist overlaps $(d_{R_1},d_{R_3},\sigma)$ and $(d_{R_3},d_{R_2},\sigma)$.
\end{itemize}
The following observation underlies the terminology of `elementary': 
If $(d_{R_1},d_{R_3},\sigma)$ and $(d_{R_3},d_{R_2},\sigma)$ are overlaps, then $(d_{R_1},d_{R_2},\sigma)$ is an overlap. 
Thus elementary overlaps are the ones that cannot be `decomposed' in this way.
\end{defn}

\begin{lem} \label{lem:peripheralelementaryoverlapexist}
If there exists an overlap, then there exists a peripheral elementary overlap.
\end{lem}
\begin{proof}
Suppose we are given an overlap $(d_{R_1},d_{R_2},\sigma)$.
We write $d_{Q_1} = d_{R_1}$ and keep this diagonal fixed. 
Consider the collection $\Sigma = \{\sigma \mid \text{there exists an overlap $(d_{Q_1},d_{R_2},\sigma)$}\}$.
We claim that $\bigcup_{\sigma \in \Sigma} \sigma$ cannot be the entire $d_{Q_1}$. 

To see this, observe that if $\sigma, \sigma' \in \Sigma$ are adjacent segments on $d_{Q_1}$, then there is an overlap $(d_{Q_1},d_{R_2},\sigma)$ if and only if there is an overlap $(d_{Q_1},d_{R_2},\sigma')$ for the same $R_2$.
Indeed, otherwise there will exist overlaps $(d_{Q_1},d_{R_2},\sigma)$ and $(d_{Q_1},d_{R'_2},\sigma')$ where $R_2 \neq R'_2$. But then $d_{R_2}$ and $d_{R'_2}$ intersect at the shared node of $\sigma$ and $\sigma'$, where it fails the slope inequality of \Cref{prop:pldiaggoals}(2).
See \Cref{fig:peripheralelementaryoverlapexist}.

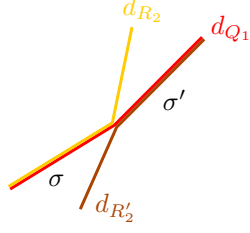
\begin{figure}
    \centering
    \fontsize{10pt}{10pt}
\begingroup%
  \makeatletter%
  \providecommand\color[2][]{%
    \errmessage{(Inkscape) Color is used for the text in Inkscape, but the package 'color.sty' is not loaded}%
    \renewcommand\color[2][]{}%
  }%
  \providecommand\transparent[1]{%
    \errmessage{(Inkscape) Transparency is used (non-zero) for the text in Inkscape, but the package 'transparent.sty' is not loaded}%
    \renewcommand\transparent[1]{}%
  }%
  \providecommand\rotatebox[2]{#2}%
  \newcommand*\fsize{\dimexpr\f@size pt\relax}%
  \newcommand*\lineheight[1]{\fontsize{\fsize}{#1\fsize}\selectfont}%
  \ifx\svgwidth\undefined%
    \setlength{\unitlength}{99.63562612bp}%
    \ifx\svgscale\undefined%
      \relax%
    \else%
      \setlength{\unitlength}{\unitlength * \real{\svgscale}}%
    \fi%
  \else%
    \setlength{\unitlength}{\svgwidth}%
  \fi%
  \global\let\svgwidth\undefined%
  \global\let\svgscale\undefined%
  \makeatother%
  \begin{picture}(1,0.78315182)%
    \lineheight{1}%
    \setlength\tabcolsep{0pt}%
    \put(0,0){\includegraphics[width=\unitlength,page=1]{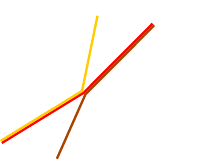}}%
    \put(0.15418629,0.10453328){\color[rgb]{0,0,0}\makebox(0,0)[lt]{\lineheight{1.25}\smash{\begin{tabular}[t]{l}$\sigma$\end{tabular}}}}%
    \put(0.5869697,0.39434849){\color[rgb]{0,0,0}\makebox(0,0)[lt]{\lineheight{1.25}\smash{\begin{tabular}[t]{l}$\sigma'$\end{tabular}}}}%
    \put(0.43343907,0.75225801){\color[rgb]{1,0.8,0}\makebox(0,0)[lt]{\lineheight{1.25}\smash{\begin{tabular}[t]{l}$d_{R_2}$\end{tabular}}}}%
    \put(0.33067985,0.00768431){\color[rgb]{0.66666667,0.26666667,0}\makebox(0,0)[lt]{\lineheight{1.25}\smash{\begin{tabular}[t]{l}$d_{R'_2}$\end{tabular}}}}%
    \put(0.76759081,0.68180602){\color[rgb]{1,0,0}\makebox(0,0)[lt]{\lineheight{1.25}\smash{\begin{tabular}[t]{l}$d_{Q_1}$\end{tabular}}}}%
  \end{picture}%
\endgroup%

    \caption{If $\sigma, \sigma' \in \Sigma$ are adjacent segments on $d_{Q_1}$, then there is an overlap $(d_{Q_1},d_{R_2},\sigma)$ if and only if there is an overlap $(d_{Q_1},d_{R_2},\sigma')$ for the same $R_2$.}
    \label{fig:peripheralelementaryoverlapexist}
\end{figure}

Thus if $\bigcup_{\sigma \in \Sigma} \sigma = d_{Q_1}$, then there would exist $d_{R_2}$ overlapping over $d_{Q_1}$ along all segments of $d_{Q_1}$, which is impossible since $Q_1 < R_2$.

Thus we can pick a element $\rho \in \Sigma$ that is 
\textbf{outermost}, in the sense that $\rho$ contains an endpoint of the subarc $\bigcup_{\sigma \in \Sigma} \sigma$ that is not a point of $\widetilde{\mathcal{C}}$.
This property implies that any overlap $(d_{Q_1},d_{R_2},\rho)$ is automatically peripheral.

We now fix $\rho$ and consider the collection $\mathcal{R} = \{d_{R_2} \mid \text{there exists an overlap $(d_{Q_1},d_{R_2},\rho)$}\}$. 
Pick some $d_{R'_2} \in \mathcal{R}$. If $(d_{Q_1},d_{R'_2},\rho)$ is non-elementary, then there exists overlaps $(d_{Q_1},d_{R''_2},\rho)$ and $(d_{R''_2},d_{R'_2},\rho)$.
Here, $(d_{R''_2},d_{R'_2},\rho)$ being an overlap implies that $R''_2 < R'_2$.
If $(d_{Q_1},d_{R''_2},\rho)$ is non-elementary, then we repeat the argument, and so on.

By \Cref{lem:overlapfinite}, $\mathcal{R}$ is finite, so the process stops eventually and we end up with a peripheral elementary overlap $(d_{Q_1},d_{Q_2},\rho)$.
\end{proof}

Suppose we are given a peripheral elementary overlap $(d_{Q_1},d_{Q_2},\rho)$.
We will explain how to modify the PLO diagonal system so that there are no longer any overlaps of $d_{Q_2}$ over $d_{Q_1}$, while not creating any new pairs of rectangles that overlap.
To this end, we introduce the following definition.

\begin{defn}[Bookkeeping nodes] \label{defn:bookkeepingnodes}
A set of \textbf{bookkeeping nodes} for $(Q_1,Q_2)$ is a subset $V$ of the set of nodes of the diagonal system $\{d_R\}$ that contain the endpoints and turns of $r = d_{Q_1} \cap d_{Q_2}$, such that adjacent elements of $V$ on $d_{Q_1}$ or $d_{Q_2}$ lie in different $\pi_1 M$-orbits.
In particular, $V$ cuts $d_{Q_1}$ and $d_{Q_2}$ into linear subarcs, which we refer to as the \textbf{bookkeeping segments}.

We emphasize that we allow nodes of the diagonal system inbetween bookkeeping nodes, thus bookkeeping segments are unions of segments in general.
\end{defn}

A set of bookkeeping nodes always exist: One can simply choose the $\pi_1 M$-orbits of all nodes on $r$.
Our modification of the PLO diagonal system will be broken into smaller steps, each retaining the property that we have a PLO diagonal system, but reducing the number of bookkeeping nodes on $d_{Q_1} \cap d_{Q_2}$.
Once $d_{Q_1} \cap d_{Q_2}$ only has one bookkeeping node, there will no longer be any overlap of $d_{Q_2}$ over $d_{Q_1}$.

For the rest of this subsection, we fix a peripheral elementary overlap $(d_{Q_1},d_{Q_2},\rho)$ and fix a set of bookkeeping nodes $V$.
Without loss of generality we suppose that $d_{Q_1}$ and $d_{Q_2}$ are red.
Throughout, we will work in the translation orbit space $\widetilde{\mathcal{O}^\circ}$.

As noted in the proof of \Cref{lem:overlapfinite}, there are only finitely many $\pi_1 M^\circ$-orbits of diagonals in $\widetilde{\mathcal{O}^\circ}$. 
Let $\widetilde{d}_1,...,\widetilde{d}_r$ be a collection of one representative in each $\pi_1 M^\circ$-orbit.
Pick a large $N$ so that 
\begin{equation} \label{eq:largeheightignore}
\lambda^{2N} > \frac{\max_i \max|\slope(d_i)|}{\min_i \min|\slope(d_i)|}.
\end{equation}

Since the overlap is peripheral, $\rho$ contains an endpoint $u$ of $d_{Q_1} \cap d_{Q_2}$.
Since at least one of $d_{Q_1}$ and $d_{Q_2}$ is non-smooth at $u$, $\rho$ is contained in a bookkeeping segment $\sigma$ that has $u$ as an endpoint.
Let $v$ be the other endpoint of $\sigma$.
We write $d_{Q_1}$ as $\sigma \ast_u \sigma'_{Q_1}$ and $d_{Q_2}$ as $\sigma \ast_u \sigma'_{Q_2}$ near $u$.
Since $u$ is an endpoint of $d_{Q_1} \cap d_{Q_2}$, we have $\slope(\sigma'_{Q_2}) > \slope(\sigma)$ or $\slope(\sigma'_{Q_1}) < \slope(\sigma)$. Without loss of generality suppose the former is true.
See \Cref{fig:perturbfortransversesetup} an illustration of our setup.

\begin{figure}
    \centering
    \fontsize{10pt}{10pt}
\begingroup%
  \makeatletter%
  \providecommand\color[2][]{%
    \errmessage{(Inkscape) Color is used for the text in Inkscape, but the package 'color.sty' is not loaded}%
    \renewcommand\color[2][]{}%
  }%
  \providecommand\transparent[1]{%
    \errmessage{(Inkscape) Transparency is used (non-zero) for the text in Inkscape, but the package 'transparent.sty' is not loaded}%
    \renewcommand\transparent[1]{}%
  }%
  \providecommand\rotatebox[2]{#2}%
  \newcommand*\fsize{\dimexpr\f@size pt\relax}%
  \newcommand*\lineheight[1]{\fontsize{\fsize}{#1\fsize}\selectfont}%
  \ifx\svgwidth\undefined%
    \setlength{\unitlength}{105.18625214bp}%
    \ifx\svgscale\undefined%
      \relax%
    \else%
      \setlength{\unitlength}{\unitlength * \real{\svgscale}}%
    \fi%
  \else%
    \setlength{\unitlength}{\svgwidth}%
  \fi%
  \global\let\svgwidth\undefined%
  \global\let\svgscale\undefined%
  \makeatother%
  \begin{picture}(1,0.75234092)%
    \lineheight{1}%
    \setlength\tabcolsep{0pt}%
    \put(0,0){\includegraphics[width=\unitlength,page=1]{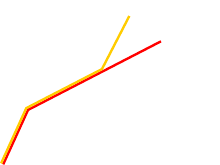}}%
    \put(0.29122352,0.25404764){\color[rgb]{0,0,0}\makebox(0,0)[lt]{\lineheight{1.25}\smash{\begin{tabular}[t]{l}$\sigma$\end{tabular}}}}%
    \put(0.08370801,0.30062622){\color[rgb]{0,0,0}\makebox(0,0)[lt]{\lineheight{1.25}\smash{\begin{tabular}[t]{l}$v$\end{tabular}}}}%
    \put(0.39095094,0.45990353){\color[rgb]{0,0,0}\makebox(0,0)[lt]{\lineheight{1.25}\smash{\begin{tabular}[t]{l}$u$\end{tabular}}}}%
    \put(0.77985495,0.5539782){\color[rgb]{1,0,0}\makebox(0,0)[lt]{\lineheight{1.25}\smash{\begin{tabular}[t]{l}$d_{Q_1}$\end{tabular}}}}%
    \put(0.56128703,0.72307737){\color[rgb]{1,0.8,0}\makebox(0,0)[lt]{\lineheight{1.25}\smash{\begin{tabular}[t]{l}$d_{Q_2}$\end{tabular}}}}%
  \end{picture}%
\endgroup%

    \caption{The setup for modifying a PLO diagonal system.}
    \label{fig:perturbfortransversesetup}
\end{figure}
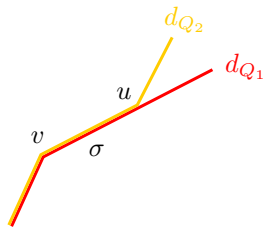

Let $\mathcal{D}_+ = \{d_R \mid \text{$Q_2 \leq R$ and $d_{Q_2} \cap d_R$ contains a subarc $\sigma_R \subset \sigma$ containing $v$}\}$.
For every $d_R \in \mathcal{D}_+$, we modify $d_R$ as follows:
Consider the linear arc $\sigma^+_R$ with slope $\slope(\sigma)+\epsilon$ starting at $v$ and ending at a point $b_R$ on the segment $\sigma'_R$ of $d_R$ after $\sigma_R$. We delete from $d_R$ the union of $\sigma_R$ and the subarc of $\sigma'_R$ from $u$ to $b_R$, then add back in $\sigma^+_R$.
See \Cref{fig:perturbfortransversemodify}.

\begin{figure}
    \centering
    \fontsize{8pt}{8pt}
    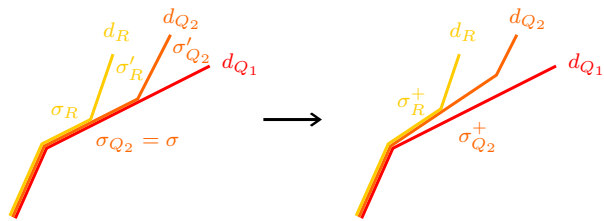
    \caption{Modifying a PLO diagonal system to reduce the number of bookkeeping nodes on $d_{Q_1} \cap d_{Q_2}$.}
    \label{fig:perturbfortransversemodify}
\end{figure}

We extend the modification equivariantly to the $\pi_1 M$-orbit of $d_R$.
Here the condition on consecutive elements in \Cref{defn:bookkeepingnodes} plays a role in ensuring that no two consecutive bookkeeping segments on a diagonal are to be modified, so that the procedure is well-defined.

We verify that the modified diagonal system is still PLO:
Each diagonal is still piecewise linear, and with $\epsilon$ small enough, they stay within $\mathcal{O}^\circ$, so \Cref{prop:pldiaggoals}(1) is clear.

Upon adding $b_R$ as the new nodes after modification, \Cref{prop:pldiaggoals}(3) is violated only if some new segment $\sigma^+_R$ passes through some node, but such a scenario can only occur for countably many values of $\epsilon$.
Similarly, \Cref{prop:pldiaggoals}(4) is violated after modification only if a new node $b_R$ happens to lie in the orbit of a pre-existing node, but this only happens for countably many values of $\epsilon$.
The conclusion is that \Cref{prop:pldiaggoals}(3) and (4) hold after modification for generic choices of $\epsilon$.

We remark that the fact that we have to add new nodes $b_R$ is the reason we have introduced the technical \Cref{defn:bookkeepingnodes}, since we wish to keep track of our progress in the metric of nodes \emph{prior} to the modification.

Thus it remains to show \Cref{prop:pldiaggoals}(2).
Suppose that $R_1 < R_2$ have the same color. 
Lift $d_{R_1}$ and $d_{R_2}$ to $\widetilde{d_{R_1}}$ and $\widetilde{d_{R_2}}$ in $\widetilde{\mathcal{O}^\circ}$.
Up to translation, we can suppose that $\widetilde{d_{R_1}}$ is one of the representatives we chose for \Cref{eq:largeheightignore}. We can then write $\widetilde{d_{R_2}} = g \cdot d_j$ for some $g \in \pi_1(M^\circ)$.
Notice that \Cref{eq:largeheightignore} is preserved if we pick $\epsilon$ small enough. This implies that we must have $\hght(g) \geq -N$, and \Cref{prop:pldiaggoals}(2) holds with $r$ degenerate whenever if $\hght(g) \geq N$.
The upshot is that by \Cref{lem:transorbitspacelocfinite}, we only have to check \Cref{prop:pldiaggoals}(2) for finitely many possibilities of $R_2$.

If neither $d_{R_1}$ nor $d_{R_2}$ has a translate that lies in $\mathcal{D}_+$, then the modification does not involve $d_{R_1}$ and $d_{R_2}$ so \Cref{prop:pldiaggoals}(2) holds from before.

Suppose $d_{R_1}$ has a translate that lies in $\mathcal{D}_+$. Up to applying the translation, we arrange it so that $d_{R_1} \in \mathcal{D}_+$. 
If $d_{R_2} \in \mathcal{D}_+$, then after modification $d_{R_1}$ and $d_{R_2}$ overlap in the same number of segments and the slope inequality continues to hold. See \Cref{fig:perturbfortransverseanalysis1}.

\begin{figure}
    \centering
    \fontsize{8pt}{8pt}
    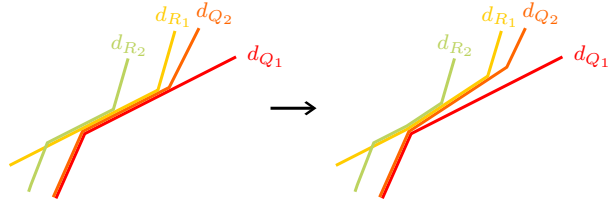
    \caption{Analysis for \Cref{prop:pldiaggoals}(2) if $d_{R_1}, d_{R_2} \in \mathcal{D}_+$.}
    \label{fig:perturbfortransverseanalysis1}
\end{figure}

If $d_{R_2} \not\in \mathcal{D}_+$, then depending on whether $d_{R_1}$ and $d_{R_2}$ overlapped along some segment of $\sigma$, after modification $d_{R_1}$ and $d_{R_2}$ either overlap in the same number of segments, or one less segment, and the slope inequality continues to hold. See \Cref{fig:perturbfortransverseanalysis2}.

\begin{figure}
    \centering
    \fontsize{8pt}{8pt}
    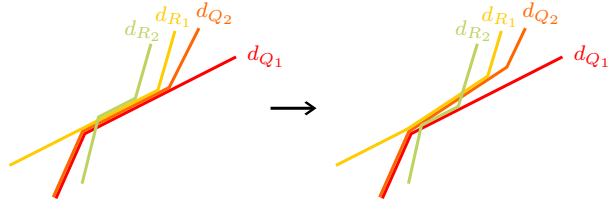
    \caption{Analysis for \Cref{prop:pldiaggoals}(2) if $d_{R_1} \in \mathcal{D}_+$ but $d_{R_2} \not\in \mathcal{D}_+$.}
    \label{fig:perturbfortransverseanalysis2}
\end{figure}

Now suppose $d_{R_1}$ does not have a translate that lies in $\mathcal{D}_+$, but $d_{R_2}$ does.
Up to applying the translation, we arrange it so that $d_{R_2} \in \mathcal{D}_+$. 
By the hypotheses that $(d_{Q_1},d_{Q_2},\rho)$ is elementary, $d_{R_1}$ cannot contain both $v$ and $\sigma'_{R_2}$
Then depending on whether $d_{R_1}$ contained a subarc of $\sigma$, after modification $d_{R_1}$ and $d_{R_2}$ either overlap in the same number of segments, or one less segment, and the slope inequality continues to hold. See \Cref{fig:perturbfortransverseanalysis3}.

\begin{figure}
    \centering
    \fontsize{8pt}{8pt}
    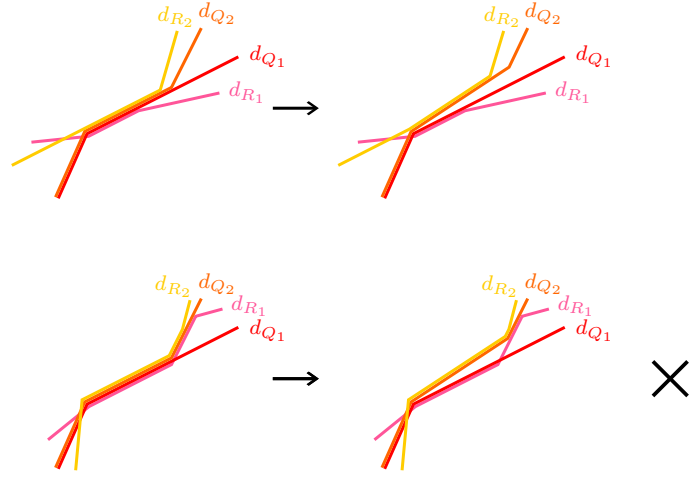
    \caption{Analysis for \Cref{prop:pldiaggoals}(2) if $d_{R_2} \in \mathcal{D}_+$ but $d_{R_1} \not\in \mathcal{D}_+$.
    By the elementary hypothesis, the bottom scenario cannot happen.}
    \label{fig:perturbfortransverseanalysis3}
\end{figure}

Finally, it is clear that the number of bookkeeping nodes on $d_{Q_1} \cap d_{Q_2}$ has been reduced by one.

Performing this construction inductively, we will eventually reduce the number of bookkeeping nodes $d_{Q_1} \cap d_{Q_2}$ to one, at which point $d_{Q_1}$ and $d_{Q_2}$ no longer overlap.
In turn, performing such a sequence of constructions inductively, we will eventually get rid of all overlaps.

We summarize our progress so far in the following lemma.

\begin{lem} \label{lem:perturbfortransverse}
Under the hypothesis of \Cref{prop:pldiaggoals}, there exists a diagonal system $\{d_R\}$
such that:
\begin{enumerate}
    \item The interior of each $d_R$ is a piecewise linear arc in $\mathcal{O}^\circ$, under the quasi-translation structure induced from $S$.
    \item The system $\{d_R\}$ satisfies the slope criterion for diagonals of the same color, i.e. for every $R_1 < R_2$ of the same color, 
    \begin{itemize}
        \item if $R_1$ and $R_2$ share a corner, then $\intr(d_{R_1})$ and $\intr(d_{R_2})$ are disjoint, and
        \item if $R_1$ and $R_2$ do not share a corner, then $\intr(d_{R_1})$ and $\intr(d_{R_2})$ intersect at exactly one point $v$, with $|\slope_v(d_{R_1})| < |\slope_v(d_{R_2})|$.
    \end{itemize}
\end{enumerate}
\end{lem}

\subsection{Perturbing for smoothness} \label{subsec:perturbforsmooth}

The goal of this subsection is to upgrade \Cref{lem:perturbfortransverse} from piecewise linear to smooth.

\begin{lem} \label{lem:perturbforsmooth}
Under the hypothesis of \Cref{prop:pldiaggoals}, there exists a diagonal system $\{d_R\}$
such that:
\begin{enumerate}
    \item The interior of each $d_R$ is a smooth arc in $\mathcal{O}^\circ$ with constant slope near its endpoints, under the quasi-translation structure induced from $S$.
    \item The system $\{d_R\}$ satisfies the slope criterion for diagonals of the same color, i.e. for every $R_1 < R_2$ of the same color, 
    \begin{itemize}
        \item if $R_1$ and $R_2$ share a corner, then $\intr(d_{R_1})$ and $\intr(d_{R_2})$ are disjoint, and
        \item if $R_1$ and $R_2$ do not share a corner, then $\intr(d_{R_1})$ and $\intr(d_{R_2})$ intersect at exactly one point $v$, with $|\slope_v(d_{R_1})| < |\slope_v(d_{R_2})|$.
    \end{itemize}
\end{enumerate}
\end{lem}

\begin{proof}
Let $\{d_R\}$ be a diagonal system as in \Cref{lem:perturbfortransverse}.
We lift everything to the translation orbit space $\widetilde{\mathcal{O}^\circ}$.
As noted in the proof of \Cref{lem:overlapfinite}, there are only finitely many $\pi_1 M^\circ$-orbits of diagonals in $\mathcal{O}^\circ$. 
Let $\widetilde{d}_1,...,\widetilde{d}_r$ be a collection of one representative in each $\pi_1 M^\circ$-orbit.
Pick a large $N$ so that \Cref{eq:largeheightignore} holds.

By \Cref{lem:transorbitspacelocfinite}, the set of $\pi_1(M^\circ)_{[-N,N]}$-orbits of $\widetilde{d}_i$ is locally finite. Thus we can choose a rectangular neighborhood $\nu_v$ of each non-smooth node $v$ of $\widetilde{d}_i$, such that 
\begin{itemize}
    \item each $\widetilde{d}_i \cap \nu_v$ can be written as $\sigma_1 \ast_v \sigma_2$ where $\sigma_1$ and $\sigma_2$ are linear and exit $\nu_v$ through its vertical sides,
    \item $\nu_v$ is disjoint from $g \cdot \nu_w$, for every $g \in \pi_1(M^\circ)_{[-N,N]}$, unless $g \cdot w = v$, and
    \item $\nu_v$ is disjoint from $g \cdot \widetilde{d}_j$, for every $g \in \pi_1(M^\circ)_{[-N,N]}$, $j = 1,...,r$, unless $g \cdot \widetilde{d}_j$ passes through $v$.
\end{itemize}
See \Cref{fig:perturbforsmooth} left.

\begin{figure}
    \centering
    \fontsize{8pt}{8pt}
\begingroup%
  \makeatletter%
  \providecommand\color[2][]{%
    \errmessage{(Inkscape) Color is used for the text in Inkscape, but the package 'color.sty' is not loaded}%
    \renewcommand\color[2][]{}%
  }%
  \providecommand\transparent[1]{%
    \errmessage{(Inkscape) Transparency is used (non-zero) for the text in Inkscape, but the package 'transparent.sty' is not loaded}%
    \renewcommand\transparent[1]{}%
  }%
  \providecommand\rotatebox[2]{#2}%
  \newcommand*\fsize{\dimexpr\f@size pt\relax}%
  \newcommand*\lineheight[1]{\fontsize{\fsize}{#1\fsize}\selectfont}%
  \ifx\svgwidth\undefined%
    \setlength{\unitlength}{209.0907588bp}%
    \ifx\svgscale\undefined%
      \relax%
    \else%
      \setlength{\unitlength}{\unitlength * \real{\svgscale}}%
    \fi%
  \else%
    \setlength{\unitlength}{\svgwidth}%
  \fi%
  \global\let\svgwidth\undefined%
  \global\let\svgscale\undefined%
  \makeatother%
  \begin{picture}(1,0.37798163)%
    \lineheight{1}%
    \setlength\tabcolsep{0pt}%
    \put(0,0){\includegraphics[width=\unitlength,page=1]{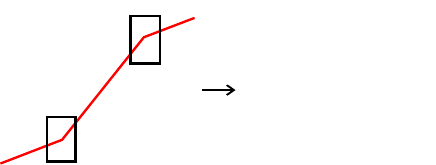}}%
    \put(0.12711997,0.02069925){\color[rgb]{0,0,0}\makebox(0,0)[lt]{\lineheight{1.25}\smash{\begin{tabular}[t]{l}$v_1$\end{tabular}}}}%
    \put(0.11603157,0.12908986){\color[rgb]{0,0,0}\makebox(0,0)[lt]{\lineheight{1.25}\smash{\begin{tabular}[t]{l}$\nu_{v_1}$\end{tabular}}}}%
    \put(0.31013196,0.36326015){\color[rgb]{0,0,0}\makebox(0,0)[lt]{\lineheight{1.25}\smash{\begin{tabular}[t]{l}$\nu_{v_2}$\end{tabular}}}}%
    \put(0.30576011,0.30929265){\color[rgb]{0,0,0}\makebox(0,0)[lt]{\lineheight{1.25}\smash{\begin{tabular}[t]{l}$v_2$\end{tabular}}}}%
    \put(0,0){\includegraphics[width=\unitlength,page=2]{perturbforsmooth.pdf}}%
  \end{picture}%
\endgroup%

    \caption{Perturbing the diagonals to be smooth.}
    \label{fig:perturbforsmooth}
\end{figure}

We now modify each $\widetilde{d}_i$ as follows: 
For each non-smooth node $v$ of $\widetilde{d}_i$, we write $\widetilde{d}_i \cap \nu_v = \sigma_1 \ast_v \sigma_2$ as above.
Choose a chart of $\nu_v$, identifying it with a rectangle $[x_0,x_1] \times [y_0,y_1] \subset \mathbb{R}^2$.
Without loss of generality suppose $0 < \slope(\sigma_1) < \slope(\sigma_2)$, so that we can write $\widetilde{d}_i \cap \nu_v$ as a graph $y = f(x)$, where $f$ is a piecewise linear function that concaves upwards.
We replace $\widetilde{d}_i \cap \nu_v$ with a graph $y = g(x)$ instead, where $g$ is a smooth function that concaves upwards, and agreeing with $f$ near $x = x_0$ and $x = x_1$.
See \Cref{fig:perturbforsmooth} right.

We then extend these modifications in a $\pi_1(M^\circ)$-equivariant way.

We now check the items in the statement of the proposition.
(1) is clear provided that the neighborhoods $\nu_v$ are chosen to be small enough.

Again, the bulk of the proof is checking (2).
Since $g$ concaves upwards, $\max\slope(\widetilde{d}_i)$ and $\min\slope(\widetilde{d}_i)$ are unchanged by the modification.
Thus \Cref{eq:largeheightignore} is preserved.

Suppose that $R_1 < R_2$ have the same color. 
Lift $d_{R_1}$ and $d_{R_2}$ to $\widetilde{d_{R_1}}$ and $\widetilde{d_{R_2}}$ in $\widetilde{\mathcal{O}^\circ}$.
Up to translation, we can suppose that $\widetilde{d_{R_1}}$ is one of the representatives we chose for \Cref{eq:largeheightignore}. We can then write $\widetilde{d_{R_2}} = g \cdot d_j$ for some $g \in \pi_1(M^\circ)$.
Then as in \Cref{subsec:perturbfortransverse}, we can assume that $|\hght(g)| \leq N$.

Let $v$ be the point of intersection between $\widetilde{d_{R_1}}$ and $\widetilde{d_{R_2}}$ prior to the modification.
If both $\widetilde{d_{R_1}}$ and $\widetilde{d_{R_2}}$ were smooth at $v$, then by the choice of $\nu_v$, the modifications on $\widetilde{d_{R_1}}$ and $\widetilde{d_{R_2}}$ were done away from $v$, so the slope inequality continues to hold.

If $\widetilde{d_{R_1}}$ was smooth at $v$ but not $\widetilde{d_{R_2}}$, then by the choice of $\nu_v$, the modifications on $\widetilde{d_{R_1}}$ were done away from $\nu_v$, so that after modification, $\widetilde{d_{R_1}}$ and $\widetilde{d_{R_2}}$ is a straight line and a convex curve within $\nu_v \cong [x_0,x_1] \times [y_0,y_1]$, with $\widetilde{d_{R_1}}$ lying above $\widetilde{d_{R_2}}$ near $x = x_0$ and the other way around near $x = x_1$. Thus they intersect once and satisfy the slope inequality.

If both $\widetilde{d_{R_1}}$ and $\widetilde{d_{R_2}}$ were not smooth at $v$, then $\widetilde{d_{R_1}}$ and $\widetilde{d_{R_2}}$ are convex curves within $\nu_v \cong [x_0,x_1] \times [y_0,y_1]$, with $\widetilde{d_{R_1}}$ lying above $\widetilde{d_{R_2}}$ near $x = x_0$ and the other way around near $x = x_1$. Thus they intersect once and satisfy the slope inequality as well.
\end{proof}

\subsection{Rotating to eliminate misalignments} \label{subsec:realigningedges}

The goal of this subsection is to remove `of the same color' in \Cref{lem:perturbforsmooth}(2). This will finally give us a diagonal system satisfying the slope criterion.

Let $\{d_R\}$ be a diagonal system as in \Cref{lem:perturbforsmooth}.
As noted in the proof of \Cref{lem:overlapfinite}, there are only finitely many $\pi_1 M^\circ$-orbits of diagonals in $\mathcal{O}^\circ$. 
Let $\widetilde{d}_1,...,\widetilde{d}_r$ be a collection of one representative in each $\pi_1 M^\circ$-orbit.
Pick a large $N$ so that \Cref{eq:largeheightignore} holds.

\begin{defn}[Misalignment] \label{defn:misalignment}
Suppose $R_1 < R_2$ are of different colors.
Then the diagonals $d_{R_1}$ and $d_{R_2}$ must intersect at a single point $v$.
Suppose $|\slope_v(d_{R_1})| \geq |\slope_v(d_{R_2})|$.
If $R_1$ is red while $R_2$ is blue, then we say that $(d_{R_1},d_{R_2},v)$ is a \textbf{upward misalignment}.
If $R_1$ is blue while $R_2$ is red, then we say that $(d_{R_1},d_{R_2},v)$ is a \textbf{downward misalignment}.
In both cases, we say that $v$ is the \textbf{projection} of the misalignment.
\end{defn}

It is helpful to have a 3-dimensional mental picture of a misalignment:
Let $\{d^\wedge_R\}$ be the collection of canonical lifts of $\{d_R\}$. 
Let $v_1 \in d^\wedge_{R_1}$ and $v_2 \in d^\wedge_{R_2}$ be the points projecting down to $v$. 
Then there is an orbit segment of $\widetilde{\phi}$ from $v_2$ to $v_1$. 
One should think of this orbit segment as the misalignment.
Furthermore, if we orient the orbit segment to go from the blue canonical lift to the red canonical lift, then the orbit segment is oriented upwards/downwards if the misalignment is upwards/downwards respectively.

\begin{lem} \label{lem:misalignmentfinite}
There are only finitely many $\pi_1 M$-orbits of misalignments.
\end{lem}
\begin{proof}
Suppose we have a misalignment between $d_{R_1}$ and $d_{R_2}$.
Up to translation, we can suppose that $\widetilde{d_{R_1}}$ is one of the representatives we chose for \Cref{eq:largeheightignore}. We can then write $\widetilde{d_{R_2}} = g \cdot d_j$ for some $g \in \pi_1(M^\circ)$.
Then as in \Cref{subsec:perturbfortransverse}, we have $|\hght(g)| \leq N$, thus there are only finitely many possibilities for $R_2$.
\end{proof}

We are now ready to modify our diagonals.

\begin{prop} \label{prop:smoothdiaggoals}
Under the hypothesis of \Cref{prop:pldiaggoals}, there exists a diagonal system $\{d_R\}$ satisfying the slope criterion.
\end{prop}
\begin{proof}
We first argue that up to a $C^1$-small perturbation, we can assume that all misalignments have distinct projections:
For each $i=1,...,r$, we consider the set of points on $\widetilde{d}_i$ that is the projection of a misalignment involving $\widetilde{d}_i$. By a $C^1$-small perturbation of $\widetilde{d}_i$, we can make these points all disjoint. See \Cref{fig:rotateformisalignment} top.
Then up to a further perturbation, we can make these points all non-periodic and occupy distinct $\pi_1 M^\circ$-orbits.
Once this is arranged, all misalignments have distinct projections.

\begin{figure}
    \centering
    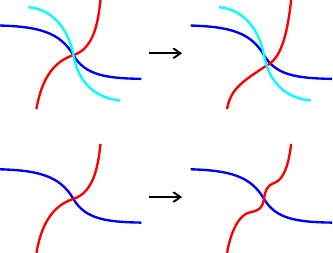
    \caption{Modifications in the proof of \Cref{prop:smoothdiaggoals}.}
    \label{fig:rotateformisalignment}
\end{figure}

Note that this also implies that if $(d_{R_1},d_{R_2},w)$ is a misalignment, then there does not exist a third diagonal $d_{R_3}$ such that $|\slope_w(d_{R_1})| \geq |\slope_w(d_{R_3})| \geq |\slope_w(d_{R_2})|$, for otherwise either $(d_{R_1},d_{R_3},w)$ or $(d_{R_3},d_{R_2},w)$ is a misalignment with the same projection.

By doing the perturbation $C^1$-small enough, we preserve \Cref{eq:largeheightignore}.
Now let $\ell_1,...,\ell_s$ be a collection of one representative in each $\pi_1 M^\circ$-orbit of misalignments, so that each $\ell_j$ involves some $\widetilde{d}_i$, and let $w_1,...,w_s$ be their projections.

By \Cref{lem:transorbitspacelocfinite}, the set of $\pi_1(M^\circ)_{[-3N,3N]}$-orbits of the points $w_j$, and the set of $\pi_1(M^\circ)_{[-3N,3N]}$-orbits of diagonals is locally finite.
Thus we can choose neighborhoods $\nu_{w_j}$ of $w_j$ so that 
\begin{itemize}
    \item $\nu_{w_j}$ and $g \cdot \nu_{w_i}$ are disjoint for every $g \in \pi_1(M^\circ)_{[-3N,3N]}$ and $i,j = 1,...,s$, unless $g=1$ and $i=j$, and
    \item $\nu_{w_j}$ and $g \cdot \widetilde{d}_i$ are disjoint for every $g \in \pi_1(M^\circ)_{[-3N,3N]}$ and $i,j = 1,...,s$, unless $g \cdot \widetilde{d}_i$ passes through $w_j$.
\end{itemize}

We now explain how to modify the diagonals within the neighborhoods $\nu_{w_j}$. Without loss of generality suppose $w_j$ is the projection of an upward misalignment $(\widetilde{d}_1,\widetilde{d}',w_j)$.
Then we modify $\widetilde{d}_1$ so that its slope at $w_j$ is $(1-\epsilon)\slope_{w_i}(\widetilde{d}')$, and such that its slope in $\nu_{w_j}$ stays within a factor of $[-\lambda^{-N},\lambda^N]$ of its original slope at $w_j$.
See \Cref{fig:rotateformisalignment} bottom.

We then extend these modifications in a $\pi_1(M^\circ)$-equivariant way.

By the choice of $\nu_{w_j}$, these modifications eliminate the misalignments at $w_j$ while not creating any new misalignments.
Thus the slope criterion is satisfied. 
\end{proof}

\subsection{Steady position} \label{subsec:steadypositionproof}

\Cref{thm:steadyposition} now quickly follows.

\begin{proof}[Proof of \Cref{thm:steadyposition}]
By \Cref{prop:smoothdiaggoals}, we have a diagonal system satisfying the slope criterion. Hence by \Cref{prop:diagtoveertri}, we can put $\Delta$ in Legendrian position. By \Cref{prop:legendriansteady}, this implies that $\Delta$ can be put in steady position.
\end{proof}

\section{Filling in the bicontact form} \label{sec:fillinginbicontactform}

In this section, we finish the proof of \Cref{thm:legendrianposition}. The key construction is filling the bicontact form $(\alpha^S_+,\alpha^S_-)$ from \Cref{eq:surfacebicontact} over the boundary orbits $\partial S$.

\subsection{Filling models} \label{subsec:fillingmodels}

Fix a Birkhoff section $S$ as in \Cref{thm:birkhoffsectionexist} and a diagonal system $\{d_R\}$ as in \Cref{prop:smoothdiaggoals}.
Applying \Cref{prop:diagtoveertri} as in the proof of \Cref{thm:steadyposition}, we can put $\Delta$ in Legendrian position with respect to $\phi$ and the bicontact form $(\alpha^S_+,\alpha^S_-)$ from \Cref{eq:surfacebicontact}, and such that the projections to $\mathcal{O}$ of the edges of $\widetilde{\Delta}$ are $d_R$.
In particular, since $d_R$ have interiors in $\mathcal{O}^\circ$, the edges of $\Delta$ are disjoint from the boundary orbits $\partial S$.
In other words, $\partial S$ intersects $\Delta$ by passing through the interior of faces.

Let $\gamma$ be a boundary orbit of $S$. Let $\nu_\gamma$ be a tubular neighborhood of $\gamma$, and let $m_\gamma \in \pi_1(\partial \nu_\gamma)$ be the (oriented) meridian.
Since we assumed that $\gamma$ is orientation-preserving, the local stable leaf of $\gamma$ intersects $\nu_\gamma$ in two parallel curves, each of slope $d_\gamma \in \pi_1(\partial \nu_\gamma)$ with intersection number one with $m_\gamma$. 
(In the literature $d_\gamma$ is usually referred to as the \textbf{degeneracy slope} at $\gamma$.)
We can then write the slope on $\partial \nu_\gamma$ induced by $S$ as $p_\gamma \mu_\gamma + q_\gamma d_\gamma$, where $p_\gamma > 0$ and $q_\gamma \neq 0$. We refer to $\frac{p_\gamma}{q_\gamma}$ as the \textbf{slope} of $S$ at $\gamma$.

In the following construction, we will specify a choice of tubular neighborhood $\nu_\gamma$ for each boundary orbit $\gamma$. 
We will also construct a `shell' around $\gamma$, i.e. a collar neighborhood of $\partial \nu_\gamma$, and endow specific coordinates on it.

\begin{constr}[Shell around boundary orbits] \label{constr:shell}
The foliations $\ell^{s/u}$ have $4p_\gamma$ quadrants at the puncture $\gamma^\circ$, which we label as $\mathfrak{q}_1,...,\mathfrak{q}_{4p_\gamma}$ in a counterclockwise way starting from an unstable half-leaf.
Recall that the measures on $\overline{\ell^{s/u}}$ determine closed 1-forms $ds$ and $du$ on each quadrant $\mathfrak{q}_k$.
We can locally integrate these closed 1-forms into functions $s$ and $u$ such that $(s,u)=(0,0)$ at $\gamma^\circ$.
Set 
\begin{align*}
x &= \lambda^t s - \lambda^{-t} u \\
y &= \lambda^t s + \lambda^{-t} u.
\end{align*}
Then for each $j \in \mathbb{Z}/4$, the region $\bigcup_k \{(t,x,y) \in \mathbb{R} \times \mathfrak{q}_{j+4k} \mid x \in [-\epsilon,\epsilon],y \in [-\epsilon,\epsilon]\} \subset \mathbb{R} \times S^\circ$ descends to a quadrant $Q_j$ of $\gamma$ in $M^\circ$.
Taking the union over $Q_j$ and $\gamma$ itself, we get a tubular neighborhood $\nu_\gamma$ of $\gamma$ in $M$.

Note that
\begin{align*}
Q_0 &\cong \{(t,x,y) \mid \mathbb{R}/(p_\gamma \mathbb{Z}) \times \mathbb{R}^2 \backslash \{(0,0)\} \mid x \in [0,\epsilon], -x \leq y \leq x\} \\
Q_1 &\cong \{(t,x,y) \mid \mathbb{R}/(p_\gamma \mathbb{Z}) \times \mathbb{R}^2 \backslash \{(0,0)\} \mid y \in [0,\epsilon], -y \leq x \leq y\} \\
Q_2 &\cong \{(t,x,y) \mid \mathbb{R}/(p_\gamma \mathbb{Z}) \times \mathbb{R}^2 \backslash \{(0,0)\} \mid x \in [-\epsilon,0],\epsilon], x \leq y \leq -y\} \\
Q_3 &\cong \{(t,x,y) \mid \mathbb{R}/(p_\gamma \mathbb{Z}) \times \mathbb{R}^2 \backslash \{(0,0)\} \mid y \in [-\epsilon,0], y \leq x \leq -y\}
\end{align*}
We set
\begin{align*}
Q_0^\square &\cong \{(t,x,y) \mid \mathbb{R}/(p_\gamma \mathbb{Z}) \times \mathbb{R}^2 \backslash \{(0,0)\} \mid x \in [\frac{\epsilon}{2},\epsilon], -x \leq y \leq x\} \\
Q_1^\square &\cong \{(t,x,y) \mid \mathbb{R}/(p_\gamma \mathbb{Z}) \times \mathbb{R}^2 \backslash \{(0,0)\} \mid y \in [\frac{\epsilon}{2},\epsilon], -y \leq x \leq y\} \\
Q_2^\square &\cong \{(t,x,y) \mid \mathbb{R}/(p_\gamma \mathbb{Z}) \times \mathbb{R}^2 \backslash \{(0,0)\} \mid x \in [-\epsilon,-\frac{\epsilon}{2}], x \leq y \leq -y\} \\
Q_3^\square &\cong \{(t,x,y) \mid \mathbb{R}/(p_\gamma \mathbb{Z}) \times \mathbb{R}^2 \backslash \{(0,0)\} \mid y \in [-\epsilon,-\frac{\epsilon}{2}], y \leq x \leq -y\}
\end{align*}
We refer to the union $\bigcup Q_j^\square$ as the \textbf{shell} around $\gamma$.

We can uniformize the coordinates on $Q_j^\square$ by mapping them into the set 
$$\{(x,y,z) \in \mathbb{R}^2 \times \mathbb{R}/(p_\gamma \mathbb{Z}) \mid (x,y) \in [-\epsilon,\epsilon]^2 \backslash (-\frac{\epsilon}{2},\frac{\epsilon}{2})^2 \}$$
as follows:
\begin{align} \label{eq:shellcoordinpos}
\begin{split}
(t,x,y) &\mapsto (x,y,t+\rho(x)) \text{ on $Q_1$} \\
(t,x,y) &\mapsto (x,y,t) \text{ on $Q_2,Q_3,Q_4$}
\end{split}
\end{align}
where
$$\rho(x) = \begin{cases}
0 & \text{for $x \leq -\frac{\epsilon}{2}$} \\
-q_\gamma & \text{for $x \geq \frac{\epsilon}{2}$}
\end{cases}$$
See \Cref{fig:shell} left.

\begin{figure}
    \centering
    \fontsize{8pt}{8pt}
    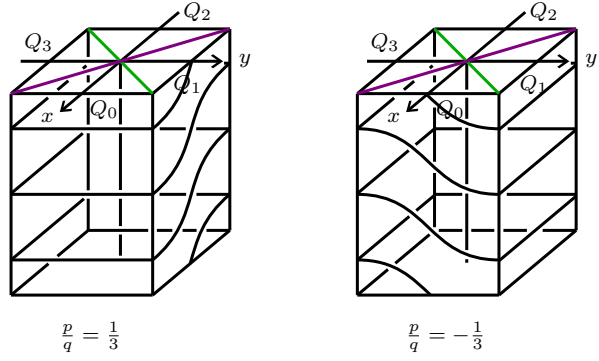
    \caption{The tubular neighborhoods constructed in \Cref{constr:shell}.}
    \label{fig:shell}
\end{figure}

Alternatively, we can map
\begin{align} \label{eq:shellcoordinneg}
\begin{split}
(t,x,y) &\mapsto (x,y,t+\rho(y)) \text{ on $Q_0$} \\
(t,x,y) &\mapsto (x,y,t) \text{ on $Q_1,Q_2,Q_3$}
\end{split}
\end{align}
where
$$\rho(x) = \begin{cases}
0 & \text{for $y \leq -\frac{\epsilon}{2}$} \\
q_\gamma & \text{for $y \geq \frac{\epsilon}{2}$}
\end{cases}$$
See \Cref{fig:shell} right.

Under either of these coordinates, we can isotope the faces of $\Delta$ along orbits so that they are surfaces of the form $\{z=z_0\}$ within the shells $\bigcup Q^\square_j$.

Now recall the bicontact form $(\alpha^S_+,\alpha^S_-)$ constructed in \Cref{subsec:canonicallift}.
We can compute that
\begin{align*}
\alpha^S_+ &= \lambda^t ds + \lambda^{-t} du = dy - (\log \lambda) x dt \\
\alpha^S_- &= \lambda^t ds - \lambda^{-t} du = dx - (\log \lambda) y dt
\end{align*}
on each $Q_j$.

Thus under \Cref{eq:shellcoordinpos}, we have
\begin{align*}
\alpha^S_+ &= -(\log \lambda)x\rho'(x)\eta'(y)dx + dy - (\log \lambda) x dz \\
\alpha^S_- &= (1-(\log \lambda)\rho'(x)\eta(y)) dx - (\log \lambda) y dz
\end{align*}
where
$$\eta(y) = 
\begin{cases}
0 & \text{for $y \leq - \frac{\epsilon}{2}$} \\
y & \text{for $y \geq \frac{\epsilon}{2}$}
\end{cases}$$

Similarly, under \Cref{eq:shellcoordinneg}, we have
\begin{align*}
\alpha^S_+ &= (1-(\log \lambda)\eta(x)\rho'(y)) dy - (\log \lambda) x dz \\
\alpha^S_- &= dx - (\log \lambda)\eta'(x)\rho'(y)ydy - (\log \lambda) y dz
\end{align*}
where
$$\eta(x) = 
\begin{cases}
0 & \text{for $x \leq - \frac{\epsilon}{2}$} \\
x & \text{for $x \geq \frac{\epsilon}{2}$}
\end{cases}$$
\end{constr}

Next, we will construct, for each $\frac{p}{q} \in \mathbb{Q} \backslash \{0\}$, a strongly adapted bicontact form $(\alpha^{\frac{p}{q}}_+,\alpha^{\frac{p}{q}}_-)$ on $\nu^{\frac{p}{q}} = [-\epsilon,\epsilon] \times [-\epsilon,\epsilon] \times \mathbb{R}/p\mathbb{Z}$ such that $(\alpha^{\frac{p}{q}}_+,\alpha^{\frac{p}{q}}_-) = (\alpha^S_+,\alpha^S_-)$ on the shell constructed in \Cref{constr:shell}.
These will be used to fill the bicontact form $(\alpha^S_+,\alpha^S_-)$ into the boundary orbits.

Note that the constructions for the cases when $\frac{p}{q} > 0$ and $\frac{p}{q} < 0$ are different. 
This asymmetry is due to the asymmetry between $\alpha_+$ and $\alpha_-$ in the definition of strongly adapted.

\begin{constr}[Filling bicontact form when $\frac{p}{q} > 0$] \label{constr:fillbicontactformpos}
Choose a smooth decreasing function $\rho:\mathbb{R} \to \mathbb{R}$ such that 
$$\rho(x) = 
\begin{cases}
0 & \text{for $x \leq - \frac{\epsilon}{2}$} \\
-q & \text{for $x \geq \frac{\epsilon}{2}$}
\end{cases}$$
and such that $\rho'$ is an even function.
Let 
$$h(x) = \int_{-\epsilon}^x -(\log \lambda) u\rho'(u) du.$$
Note that $h(x) = 0$ for $x \leq -\frac{\epsilon}{2}$ and $x \geq \frac{\epsilon}{2}$.

Choose a smooth increasing convex function $\eta:\mathbb{R} \to \mathbb{R}$ such that
$$\eta(y) = 
\begin{cases}
0 & \text{for $y \leq - \frac{\epsilon}{2}$} \\
y & \text{for $y \geq \frac{\epsilon}{2}$}
\end{cases}$$

We record the following positivity properties
\begin{align} \label{eq:fillbicontactformpossign}
\begin{split}
\rho'(x) \leq 0 \\
h(x) \geq 0 \\
\eta(y) \geq 0 \\
\eta(y)-\eta'(y)y \geq 0 \\
\eta''(y) \geq 0.
\end{split}
\end{align}

We now define
\begin{align*}
\alpha^{\frac{p}{q}}_+ &= h'(x)\eta'(y) dx + (1+h(x)\eta''(y)) dy - (\log \lambda) x dz \\
&= dy - (\log \lambda) x dz + d(h(x)\eta'(y))
\end{align*}
We compute 
$$d\alpha^{\frac{p}{q}}_+ = -(\log \lambda) dxdz$$
thus 
$$\alpha^{\frac{p}{q}}_+ \wedge d\alpha^{\frac{p}{q}}_+ = (\log \lambda)(1+h(x)\eta''(y)) dxdydz > 0$$
under the orientation $dxdydz > 0$, by \Cref{eq:fillbicontactformpossign}. Thus $\alpha_+$ is a positive contact form with Reeb vector field $R_+ = \frac{\partial}{\partial y}$.

Meanwhile, we define
$$\alpha^{\frac{p}{q}}_- = (1-(\log \lambda)\rho'(x)\eta(y)) dx - (\log \lambda) y dz.$$
We compute 
$$d\alpha^{\frac{p}{q}}_- = (\log \lambda)\rho'(x)\eta'(y) dxdy - (\log \lambda) dydz$$
thus
$$\alpha^{\frac{p}{q}}_- \wedge d\alpha^{\frac{p}{q}}_- = (\log \lambda)^2(-1+\rho'(x)(\eta(y)-\eta'(y)y)) dxdydz < 0$$
by \Cref{eq:fillbicontactformpossign}. Thus $\alpha^{\frac{p}{q}}_-$ is a negative contact form.

We compute that 
\begin{align*}
\alpha^{\frac{p}{q}}_- \wedge \alpha^{\frac{p}{q}}_+ &= (1+h(x)\eta''(y))(\log \lambda)y dydz \\
&~+ (-h'(x)\eta'(y)(\log \lambda)y+(1-(\log \lambda)\rho'(x)\eta(y))(\log \lambda)x) dzdx \\
&~+(1-(\log \lambda)\rho'(x)\eta(y))(1+h(x)\eta''(y)) dxdy \\
&= (1+h(x)\eta''(y))(\log \lambda)y dydz \\
&~+ (1+\rho'(x)(\eta(y)-\eta'(y)y))(\log \lambda)x dzdx \\
&~+(1-(\log \lambda)\rho'(x)\eta(y))(1+h(x)\eta''(y)) dxdy.
\end{align*}
By \Cref{eq:fillbicontactformpossign}, the last coefficient is positive, thus $\alpha^{\frac{p}{q}}_- \wedge \alpha^{\frac{p}{q}}_+ \neq 0$.
That is, $(\alpha^{\frac{p}{q}}_+,\alpha^{\frac{p}{q}}_-)$ is a bicontact form, and the line field $\ker \alpha^{\frac{p}{q}}_+ \cap \ker \alpha^{\frac{p}{q}}_-$ is transverse to the meridional discs $\{z=z_0\}$.

Since $\alpha^{\frac{p}{q}}_-(R_+) = 0$, $(\alpha^{\frac{p}{q}}_+,\alpha^{\frac{p}{q}}_-)$ is a strongly adapted bicontact form.
Also note that $(\alpha^S_+,\alpha^S_-) = (\alpha^{\frac{p}{q}}_+,\alpha^{\frac{p}{q}}_-)$ on the shell constructed in \Cref{constr:shell}.
\end{constr}

\begin{constr}[Filling bicontact form when $\frac{p}{q} < 0$] \label{constr:fillbicontactformneg}
Choose a smooth decreasing function $\rho:\mathbb{R} \to \mathbb{R}$ such that 
$$\rho(y) = 
\begin{cases}
0 & \text{for $y \leq - \frac{\epsilon}{2}$} \\
q & \text{for $y \geq \frac{\epsilon}{2}$}
\end{cases}$$

Choose a smooth increasing convex function $\eta:\mathbb{R} \to \mathbb{R}$ such that
$$\eta(x) = 
\begin{cases}
0 & \text{for $x \leq - \frac{\epsilon}{2}$} \\
x & \text{for $x \geq \frac{\epsilon}{2}$}
\end{cases}$$

We record the following positivity properties
\begin{align} \label{eq:fillbicontactformnegsign}
\begin{split}
\rho'(y) \leq 0 \\
\eta(x) \geq 0 \\
\eta(x)-\eta'(x)x \geq 0 \\
\eta''(x) \geq 0.
\end{split}
\end{align}

We now define
\begin{align*}
\alpha^{\frac{p}{q}}_+ &= (1-(\log \lambda)\eta(x)\rho'(y)) dy - (\log \lambda) x dz
\end{align*}
We compute 
$$d\alpha^{\frac{p}{q}}_+ = -(\log \lambda)\eta'(x)\rho'(y) dx dy -(\log \lambda) dxdz$$
thus 
$$\alpha^{\frac{p}{q}}_+ \wedge d\alpha^{\frac{p}{q}}_+ = (\log \lambda)^2 (1+\rho'(y)(\eta'(x)x-\eta(x))) > 0$$
by \Cref{eq:fillbicontactformnegsign}. Thus $\alpha_+$ is a positive contact form with Reeb vector field $R_+ = \frac{\partial}{\partial y} - \eta'(x)\rho'(y) \frac{\partial}{\partial t}$.

Meanwhile, we define
$$\alpha^{\frac{p}{q}}_- = dx - (\log \lambda) \eta'(x)\rho'(y)y dy - (\log \lambda) y dz.$$
We compute 
$$d\alpha^{\frac{p}{q}}_- = - (\log \lambda) \eta''(x)\rho'(y)y dxdy - (\log \lambda) dydz$$
thus
$$\alpha^{\frac{p}{q}}_- \wedge d\alpha^{\frac{p}{q}}_- = -(\log \lambda)(1-(\log \lambda) \eta''(x)\rho'(y)y^2) dxdydz < 0$$
by \Cref{eq:fillbicontactformnegsign}. Thus $\alpha^{\frac{p}{q}}_-$ is a negative contact form.

We compute that 
\begin{align*}
\alpha^{\frac{p}{q}}_- \wedge \alpha^{\frac{p}{q}}_+ &= (\log \lambda)y(1-(\log \lambda)(\eta(x)-\eta'(x)x)\rho'(y)) dydz \\
&~+ (\log \lambda)x dzdx \\
&~+ (1-(\log \lambda)\eta(x)\rho'(y)) dxdy 
\end{align*}
The third coefficient is positive by \Cref{eq:fillbicontactformnegsign}, thus $\alpha^{\frac{p}{q}}_- \wedge \alpha^{\frac{p}{q}}_+ \neq 0$.
That is, $(\alpha^{\frac{p}{q}}_+,\alpha^{\frac{p}{q}}_-)$ is a bicontact form, and the line field $\ker \alpha^{\frac{p}{q}}_+ \cap \ker \alpha^{\frac{p}{q}}_-$ is transverse to the meridional discs $\{z=z_0\}$.

Since $\alpha^{\frac{p}{q}}_-(R_+) = 0$, $(\alpha^{\frac{p}{q}}_+,\alpha^{\frac{p}{q}}_-)$ is a strongly adapted bicontact form.
Also note that $(\alpha^S_+,\alpha^S_-) = (\alpha^{\frac{p}{q}}_+,\alpha^{\frac{p}{q}}_-)$ on the shell constructed in \Cref{constr:shell}.
\end{constr}

\subsection{Legendrian position} \label{subsec:legpositionproof}

We are now ready to prove \Cref{thm:legendrianposition}.

\begin{proof}[Proof of \Cref{thm:legendrianposition}]
For each boundary orbit $\gamma$, we first isotope the faces of $\Delta$ along orbits so that they are surfaces of the form $\{z=z_0\}$ within the shells.
Next, we cut out $\bigcup_\gamma \nu_\gamma$ from $M$ and glue in $\bigcup_\gamma \nu^{\frac{p_\gamma}{q_\gamma}}$ in order to get a strongly adapted bicontact form $(\alpha_+,\alpha_-)$ defined on a 3-manifold homeomorphic to $M$.
In particular $(\alpha_+,\alpha_-)$ supports an Anosov flow $\psi$.

Meanwhile, we can extend $\Delta \cap (M \backslash \bigcup_\gamma \nu_\gamma)$ into a copy of $\Delta$ placed in transverse position with respect to $\psi$ on the reglued manifold by extending the faces by meridional discs $\{z=z_0\}$.
Applying the uniqueness part of \Cref{thm:transpositionexistunique}, this implies that $\phi$ and $\psi$ are orbit equivalent through a map isotopic to identity.

Moreover, since the blue edges are Legendrian with respect to $\xi_+ = \ker \alpha_+$ and the red edges are Legendrian with respect to $\xi_- = \ker \alpha_-$, $\Delta$ is in fact placed in Legendrian position with respect to $\psi$.
We pull back $(\alpha_+,\alpha_-)$ using the orbit equivalence to get a bicontact structure supporting $\phi$.
Then the pullback of $\Delta$ is in Legendrian position with respect to $\phi$ and the pulled back bicontact structure.
\end{proof}

\bibliographystyle{alphaurl}

\bibliography{bib.bib}

\end{document}